\documentclass[10.5pt,reqno]{amsart}
\usepackage{longtable} 
\usepackage{hyperref}
\usepackage[T1]{fontenc}
\usepackage[utf8]{inputenc}
\usepackage[english]{babel} 
\usepackage{textcomp}
\usepackage{dsfont}
\usepackage{latexsym}
\usepackage{amssymb}
\usepackage{amsthm}
\usepackage{amsmath}
\DeclareMathAlphabet{\mathpzc}{OT1}{pzc}{m}{en}
\usepackage{yfonts}
\usepackage{xfrac}
\usepackage{newlfont}
\usepackage{graphicx}
\usepackage{mathtools}
\usepackage{comment}
\usepackage{indentfirst}
\usepackage{braket}
\usepackage{mathrsfs}
\usepackage{xcolor}

\usepackage{etoolbox}

\usepackage{scalerel}[2014/03/10]
\usepackage[usestackEOL]{stackengine}
\newcommand{\dashint}{\,\ThisStyle{\ensurestackMath{%
			\stackinset{c}{.2\LMpt}{c}{.5\LMpt}{\SavedStyle-}{\SavedStyle\phantom{\int}}}%
		\setbox0=\hbox{$\SavedStyle\int\,$}\kern-\wd0}\int}

\newcommand{\Car}{\mathrm{C}}
\newcommand{\RC}{\mathrm{RC}}
\newcommand{\Samp}{\mathrm{S}}

\DeclareMathOperator{\card}{Card}

\DeclareMathOperator{\supp}{Supp}

\DeclareMathOperator{\ad}{ad}

\newcommand{\ee}{\mathrm{e}}

\newcommand{\loc}{\mathrm{loc}}
\newcommand{\vect}[1]{\mathbf{{#1}}}
\newcommand{\dd}{\mathrm{d}}

\DeclarePairedDelimiter{\abs}{\lvert}{\rvert}

\DeclarePairedDelimiter{\norm}{\lVert}{\rVert}

\let\originalleft\left
\let\originalright\right
\renewcommand{\left}{\mathopen{}\mathclose\bgroup\originalleft}
\renewcommand{\right}{\aftergroup\egroup\originalright}

\newcommand{\grado}{\Df}
\newcommand{\N}{\mathds{N}}
\newcommand{\Z}{\mathds{Z}}
\newcommand{\Q}{\mathds{Q}}

\newcommand{\C}{\mathds{C}}

\newcommand{\R}{\mathds{R}}

\newcommand{\gf}{\mathfrak{g}}

\newcommand{\Df}{\mathfrak{D}}

\newcommand{\Mf}{\mathfrak{M}}

\newcommand{\unif}{\mathrm{unif}}

\newcommand{\Bs}{\mathscr{B}}

\newcommand{\Dc}{\mathcal{D}}

\newcommand{\Ic}{\mathcal{I}}

\newcommand{\Lc}{\mathcal{L}}
\renewcommand{\Mc}{\mathcal{M}}

\newcommand{\Rc}{\mathcal{R}}
\newcommand{\Sc}{\mathcal{S}}

\newcommand{\Fs}{\mathscr{F}}

\newcommand{\meg}{\leqslant}
\newcommand{\Meg}{\geqslant}
\newcommand{\eps}{\varepsilon}
\renewcommand{\phi}{\varphi}
\newcommand{\mi}{\mu}

\newcommand{\leftexp}[2]{{\vphantom{#2}}^{#1}{#2}} 
\newcommand{\trasp}{\leftexp{t}}

\thanks{{\em Math Subject Classification 2020}: 46E36, 22E30.}
\thanks{The author is a member of the 	Gruppo Nazionale per l'Analisi
	Matematica, la Probabilit\`a e le	loro Applicazioni (GNAMPA) of
	the Istituto Nazionale di Alta Matematica (INdAM). The author was partially funded by the INdAM-GNAMPA Project CUP\_E5324001950001.
}

\begin{document}
	\title[Besov and Triebel--Lizorkin Spaces]{Besov and Triebel--Lizorkin Spaces on Filtered Lie Groups, II: Pointwise Multiplication, Localization, Differences}
	
	\author[M.\ Calzi]{Mattia Calzi} 
	\address{Dipartimento di Matematica, Universit\`a degli Studi di
		Milano, Via C. Saldini 50, 20133 Milano, Italy}
	\email{{\tt mattia.calzi@unimi.it}}
	
	\theoremstyle{definition}
	\newtheorem{deff}{Definition}[section]

	\newtheorem{oss}[deff]{Remark}
	
	\newtheorem{ass}[deff]{Assumptions}
	
	\newtheorem{nott}[deff]{Notation}

	\theoremstyle{plain}
	\newtheorem{teo}[deff]{Theorem}
	
	\newtheorem{lem}[deff]{Lemma}
	
	\newtheorem{prop}[deff]{Proposition}
	
	\newtheorem{cor}[deff]{Corollary}
	
	\begin{abstract}
		We continue to develop a theory of  Besov and Triebel--Lizorkin spaces associated with weighted subcoercive operators on a real connected Lie group, focusing on the algebra properties, pointwise multipliers, localized norms, and characterization by differences.
	\end{abstract}
	
	\maketitle
	
	\section{Introduction}
	
	Besov and Triebel--Lizorkin spaces form a large class of function spaces on the Euclidean spaces which provides a uniform, albeit somewhat technical, way to study simultaneously several classical function spaces, such as Sobolev spaces with integer regularity, fractional Sobolev spaces, both in the version of Sovolev--Slobodeckij--Gagliardo spaces and in the version of Bessel potential spaces, Lipschitz spaces, Hardy and BMO spaces, etc. These spaces have been extensively studied in the classical Euclidean setting (cf., e.g.,~\cite{TriebelFS,TriebelFS2,TriebelFS3}), but have also been extended to more general contexts, such as: open subsets of $\R^n$ (cf., e.g.,~\cite[Chapter 5]{TriebelFS2}); Riemannian manifolds with bounded geometry (cf., e.g.,~\cite[Chapter 7]{TriebelFS2}); Lie groups endowed with a left-invariant Riemannian (cf., e.g.,~\cite[Chapter 7]{TriebelFS2}) or sub-Riemannian structure (cf., e.g.,~\cite{BPV,BPV2,BPV3}); metric spaces endowed with suitable operators resembling a (sub-)Laplacian (cf., e.g.,~\cite{TriebelFS3,Hu}).\footnote{The literature on the subject is quite extensive and the above mentioned reference should only be intended as a very short list of examples, which is by no means complete.} 
	There are nonetheless some contexts where `measuring regularity' in a `Riemannian way,' that is, grouping together all differential operators of the same order, or, more generally, in a `sub-Riemannian way,' appears to be inconvenient since it either clashes with the underlying geometry of the space or with the structure of the differential operator at hand. For instance, let $G$ be a homogeneous group, that is, a simply connected nilpotent Lie group whose Lie algebra $\gf$ has a graduation $(\gf_\lambda)_{\lambda>0}$; in other words, $\gf=\bigoplus_{\lambda>0} \gf_\lambda$ and $[\gf_\lambda,\gf_\nu]\subseteq \gf_{\lambda+\mi}$ for every $\lambda,\mi>0$. Then, $G$ may be identified with $\gf$ by means of the exponential map and $\gf$ may be endowed with a family of automorphic dilations $(\delta_r)_{r>0}$ defined so that $\delta_r(X)=r^\lambda X$ for every $X\in \gf_\lambda$. Operators which are compatible with these dilations (that is, homogeneous operators) are consequently quite natural in this context and one is therefore led to consider (`Goodman type') Sobolev spaces of the form $\Set{f\in L^p(G)\colon\forall \alpha\; (d_\alpha\meg k \implies\vect X^\alpha f\in L^p(G))}$, where $\vect X^\alpha=X_1^{\alpha_1}\cdots X_n^{\alpha_n}$ for some homogeneous basis $(X_1,\dots, X_n)$ of $\gf$, and where $d_\alpha=\sum_j \alpha_j \deg(X_j)$. It turns out, however, that these spaces behave quite weirdly for general $k$ -- for instance, they do not interpolate as one may expect. In fact, a different class of `Bessel potential' Sobolev spaces exhibiting a more natural behviour was introduced in~\cite{FischerRuzhansky}, and was shown to coincide with the previous `Goodman type' Sobolev spaces only for specific values of $k$. It is now worthwhile remarking that, whereas the usual `Bessel potential' Sobolev spaces (and, more generally, several of the Besov and Triebel--Lizorkin spaces briefly mentioned above) are essentially constructed using a (sub-)Laplacian, these `homogeneous' Sobolev spaces were constructed using positive Rockland operators instead, that is, homogeneous and hypoelliptic left-invariant differential operators. A definition using second order differential operator would simply not be possible. 
	Notice, by the way, that replacing a second-order subelliptic differential operator with a higher-order one provides several technical difficulties, since the associated heat kernel cannot be positive, and there is no longer any finite speed property for the corresponding wave propagator -- tools which often lie at the core of several proofs in the literature.
	
	It is therefore natural to wonder whether there is a more general framework which allows to deal at the same time with sub-Laplacians and positive Rockland operators. As a matter of fact, ter Elst and Robinson showed that weighted subcoercive operators provide are quite reasonable and natural choice (cf.~\cite{ElstRobinson}). Indeed, given a group $G$ whose Lie algebra is endowed with a suitable increasing filtration $(\gf_\lambda)_{\lambda>0}$, it is possible to associate a homogeneous group $G_*$ (the `contraction') to $G$ in a natural way, and to associate to every left-invariant differential operator $\Lc$ of degree $d$ some homogeneous left-invariant differential operator $P$ of degree $d$ on $G_*$ (which plays the r\^ole of the `principal part' of $\Lc$). The operator $\Lc$ is then said to be weighted subcoercive if $P+P^*$ is a positive Rockland operator. Notice that this definition mimics closely that of elliptic operators, and that Rockland operators play the r\^ole of homogeneous elliptic operators. Weighted subcoercive operators then enjoy several useful properties. For example, they generate a heat semigroup $(\ee^{-t\Lc})$ whose convolution kernel satisfies suitable Gaussian estimates (even though the exponential decay depends on the degree $d$ and is milder than the classical one). If, in addition, $\Lc$ is formally self-adjoint, then the closure of $\Lc$ on the space of test functions is self-adjoint on $L^2$, hence generates a functional calculus which enjoys particularly interesting properties when $G$ has polynomial growth.
	As shown in~\cite{BCP}, using the heat semigroup associated with a weighted subcoercive operator allows one to define natural Besov and Triebel--Lizorkin  spaces  $B^{p,q}_\alpha$ and $F^{p,q}_\alpha$ on a general connected filtered Lie group. The resulting spaces then do not depend on the chosen operator. In~\cite{BCP} several properties of these spaces were proved, including:
	\begin{itemize}
		\item $B^{p,q}_\alpha$ and $F^{p,q}_\alpha$ are Banach spaces;
		
		\item the space $C^\infty_c(G)$ of test functions is dense in $B^{p,q}_\alpha$ and $F^{p,q}_\alpha$ for $p,q<\infty$;
		
		\item $B^{p',q'}_{-\alpha}$ and $F^{p',q'}_{-\alpha}$ may be canonically identified with the duals of $B^{p,q}_\alpha$ and $F^{p,q}_\alpha$, respectively, when $p,q<\infty$;
		
		\item $B^{p_1,q_1}_{\alpha_1}\subseteq B^{p_2,q_2}_{\alpha_2}$ when $p_1\meg p_2$, $\alpha_2 -Q_*/p_2\meg \alpha_1-Q_*/p_1$, and either $q_1\meg q_2$ or $\alpha_2 -Q_*/p_2< \alpha_1-Q_*/p_1$, where $Q_*$ denotes the homogeneous dimension of $G_*$;\footnote{This and the following fact are actually true when $G$ is endowed with a left Haar measure, and should be slightly modified in the general case.}
		
		\item   $F^{p_1,q_1}_{\alpha_1}\subseteq F^{p_2,q_2}_{\alpha_2}$ when $p_1\meg p_2$, $\alpha_2 -Q_*/p_2\meg \alpha_1-Q_*/p_1$, and either $q_1\meg q_2$ or $p_1<p_2$;
		
		\item $B^{p,q}_\alpha,F^{p,q}_\alpha \subseteq L^p$ when $\alpha>0$, and $F^{p,2}_0=L^p$ for $p\in (1,\infty)$;
		
		\item if $\omega\in \R$ is sufficiently large, then $(\Lc+\omega I)^{-\alpha'}$ induces canonical isomorphisms of $B^{p,q}_\alpha$ and $F^{p,q}_\alpha$ onto $B^{p,q}_{\alpha+\alpha'}$ and $F^{p,q}_{\alpha+\alpha'}$, respectively;
		
		\item if $(X_j)$ is a family of elements of $\gf$ such that the corresponding elements $Y_j$ of $\gf_*$ induce a basis of $\gf_*/[\gf_*,\gf_*]$, and if $\dd$ is the least common multiple of the $d_j=\deg(X_j)$, then $f\in B^{p,q}_\alpha$ (resp.\ $f\in F^{p,q}_\alpha$) if and only if $\ee^{-\Lc}f\in L^p$ and $X_j^{\dd/d_j}\in B^{p,q}_{\alpha-\dd}$ (resp.\ $X_j^{\dd/d_j}\in F^{p,q}_{\alpha-\dd}$) for every $j$. In particular $F^{p,2}_{k\dd}$ coincides with the above-mentioned `Goodman type' Sobolev spaces when $k\in\N$ and $p\in (1,\infty)$;
		  
		\item the spaces $B^{p,q}_\alpha$ and the spaces $F^{p,q}_\alpha$ interpolate as the classical ones.
	\end{itemize}
	
	The purpose of this paper is to prove several additional properties of these Besov and Triebel--Lizorkin spaces. In particular, we shall:
	\begin{itemize}
		\item prove that $B^{p,q}_\alpha\cap L^\infty$ and $F^{p,q}_\alpha\cap L^\infty$ are algebras under pointwise multiplication for every $\alpha>0$. In particular, $B^{p,q}_\alpha$ and $F^{p,q}_\alpha$ are algebras for every $\alpha>Q_*/p$;
		
		\item describe the spaces of pointwise multipliers of $B^{p,q}_\alpha$ and $F^{p,q}_\alpha$, for $\alpha>Q_*/p$;
		
		\item show that the spaces $F^{p,q}_\alpha$  enjoy a localization property as the classical ones;
		
		\item provide a `localized' norm for the spaces $B^{p,q}_\alpha$ and $F^{p,q}_\alpha$ with $\alpha>0$;
		
		\item provide an equivalent description of some Besov and Triebel--Lizorkin spaces in terms of (finite) differences.
	\end{itemize}
	
	Concerning the algebra properties, we shall basically follow the procedure of~\cite{BPV}. Analogously, we shall basically follow the procedures of~\cite{BPV2} to consider the characterization by differences and the procedures of~\cite{BPV3} to study pointwise multipliers. In addition to what was done on~\cite{BPV2},   we shall provide one of the two required norm inequalities   in full generality; we are able to prove the reverse inequality only for first order differences (as well as and second order differences for Besov spaces only). Concerning the description of pointwise multipliers, we shall actually improve the methods of~\cite{BPV3} and provide a characterization of the spaces of pointwise multipliers of the Besov spaces $B^{p,q}_\alpha$ for every $\alpha>Q_*/p$, and not only for $\alpha\in (Q_*/p,1)+\N$, when $p\meg q$, and for $\alpha\in (Q_*/p,1)$, when $p<q$.
	
	Let us also mention that several additional works  (cf.~\cite{Calzi3,Calzi4,Calzi5,CalziRizzo}) on this subject are in preparation. In particular, we plan to provide a definition for the spaces $F^{\infty,q}_\alpha$ following the approach by Frazier and Jawerth~\cite{FrazierJawerth}, and well as to study the full scale of Besov and Triebel--Lizorkin spaces (that is, for $p,q\in (0,\infty]$) in the more particular case in which $G$ has polynomial volume growth. In this latter context, we plan to study also additional properties, such as discretization, and to study the relationship between Triebel--Lizorkin spaces and local Hardy and bmo spaces.
	
	\smallskip
	
	Here is a plan of the paper. In Section~\ref{sec:2}, we shall collect several preliminary definitions and properties, mostly referring to~\cite{BCP} for proofs. In particular, we shall present a precise definition of Besov and Triebel--Lizorkin spaces in this context.
	In section~\ref{sec:3}, we shall prove the algebra properties of Besov and Triebel--Lizorkin spaces using the paraproducts introduced in~\cite{Feneuil}. In Section~\ref{sec:4}, we shall discuss the localization property of Triebel--Lizorkin spaces. In Sections~\ref{sec:5} and~\ref{sec:6}, we shall describe the spaces of pointwise multipliers of Triebel--Lizorkin spaces and of Besov spaces, respectively. Finally, in Section~\ref{sec:7} we shall provide some characterization of Besov and Triebel--Lizorkin spaces by differences.

	\section{Preliminaries}\label{sec:2}
	
	\subsection{Relatively Invariant Measures and Convolution}

	Throughout the paper, we shall denote with $G$ a connected (finite-dimensional, real) Lie group  with Lie algebra $\gf$. We shall endow $G$ with a non-trivial positive relatively invariant (Radon) measure $\beta$. Thus, there are two positive characters $\Delta_L$ and $\Delta_R$ of $G$ such that 
	\[
	\Delta_L(y)\int_G f(y x)\,\dd \beta(x) =  \int_G f(x)\,\dd \beta(x) = \Delta_R(y) \int_G f(x y )\,\dd \beta(x)
	\]
	for every $f\in L^1(\beta)$ and for every $y\in G$. In particular,
	\[
	\beta(x A)=\Delta_L(x) \beta(A) \qquad\text{and}\qquad \beta(A x)=\Delta_R(x) \beta(A)
	\]
	for every $\beta$-measurable subset $A$ of $G$ and for every $x\in G$. In addition,
	\[
	\int_G f(x^{-1}) \,\dd \beta(x)= \int_G f(x) \Delta_L(x^{-1})\Delta_R(x^{-1})\,\dd \beta(x)
	\]
	for every positive $\beta$-measurable function $f$ on $G$. We shall define $\beta_L\coloneqq \Delta_L^{-1}\cdot \beta$ and $\beta_R\coloneqq \Delta_R^{-1}\cdot \beta$, so that $\beta_L$ is a left Haar measure and $\beta_R$ is a right Haar measure.
	
	We now recall some facts about convolution. Given two convolvable\footnote{We shall not provide a precise definition of this concept, since this would drag us too far away from the main topic. } distributions $f,g$ on $G$, one has
	\[
	\langle f*g, \phi\rangle= \langle f\otimes g, (x,y)\mapsto \phi(xy)\rangle
	\]
	for every $\phi\in C^\infty_c(G)$.
	We shall identify each $f\in L^1_\loc(\beta)$ with $f\cdot \beta$, that is, the measure with density $f$ with respect to $\beta$. Given two convolvable functions $f,g$ such that $f*g$ is absolutely continuous with respect to $\beta$, we shall generally identify $f*g$ with its density with respect to $\beta$. Thus, under very mild conditions which will be always verified in the applications,
	\begin{equation}\label{eq:1}
	(f* g)(x)=\int_G f(x y^{-1}) g(y) \Delta_R(y^{-1})\,\dd \beta(y)= \int_G f(y) g(y^{-1}x)\Delta_L(y^{-1})\,\dd \beta(y).
	\end{equation} 
	
	Similar formulae apply when either $f$ or $g$ is a distribution (and the convolution is still a function).
	Observe that, if $f,g$ are convolvable distributions, $X$ is a left-invariant differential operator, and $Y$ is a right-invariant differential operator, then $Yf $ and $X g$ are convolvable and
	\[
	YX(f*g)=(Yf)*(Xg).
	\]
	
	If $f,g,h$ are distributions then, under some reasonable conditions that will always be satisfied in the applications, 
	\[
	\langle f * g, h \rangle =\langle f, h* (\Delta_R \Rc g)\rangle =\langle g, (\Delta_L \Rc f)* h\rangle,
	\]
	where $\langle \Rc f,\phi\rangle=\langle f,\check \phi \rangle$ for every $\phi\in C^\infty_c(G)$, and $\check \phi=\phi(\,\cdot\,^{-1})$. A word of caution here: $\Rc \beta= \Delta_L^{-1} \Delta_R^{-1} \beta$, so that $\Rc(f\cdot \beta) =(\Delta_L^{-1} \Delta_R^{-1}  \check f)\cdot \beta$. In other words, one must be careful not to confuse $\Rc f$ with $\check f$ when $f\in L^1_\loc(\beta)$.

	We now recall Young's inequality in this context. Take $p_1,p_2,p_3\in [1,\infty]$ so that $\frac{1}{p_1'}+\frac{1}{p_2'}=\frac{1}{p_3'}$. Then,
	\[
	\norm{(\Delta_L^{1/p_2'} f)*(\Delta_R^{1/p_1'}g)}_{L^{p_3}(\beta)}\meg \norm{f}_{L^{p_1}(\beta)}\norm{g}_{L^{p_2}(\beta)}
	\]
	or, equivalently,
	\[
	\norm{f*g}_{L^{p_3}(\beta)}\meg \norm{\Delta_L^{-1/p_2'}f}_{L^{p_1}(\beta)}\norm{\Delta_R^{-1/p_1'}g}_{L^{p_2}(\beta)}
	\]
	for every two positive $\beta$-measurable functions $f,g$ (for positive measurable functions, convolution may be defined by means of~\eqref{eq:1}). See~\cite[Lemma 2.1]{KR78} when $\Delta_L=1$. The general case follows easily.
	 
	\subsection{Differential Operators}
	
	\begin{deff}
		We shall generally identify the (complexification of the) enveloping algebra $U(G)$ of $\gf$ with the algebra of left-invariant differential operators. We shall fix a scalar product on $\gf$, and we shall endow $U(G)$ with the corresponding scalar product, namely the quotient of the natural scalar product on the (complexfication of the) tensor algebra over $\gf$.\footnote{The actual scalar product on $U(G)$ will essentially not matter in the sequel. This one is relatively convenient, since $\abs{X}=\abs{X^+}=\abs{\overline X}$ for every $X\in U(G)$.}
		
		Let $X$ be a left-invariant differential operator. We  denote with $X^R$ the right-invariant differential operator which induces the same point distribution as $X$ at $e$. In other words, $(X f)(e)=(X^R f)(e)$ for every $f\in C^\infty (G)$. We denote with $X^+$ the transpose of $X$ in the enveloping algebra $U(G)$. In other words, the mapping $X\mapsto X^+$ is the unique anti-automorphism of $U(G)$ (that is, such that $(XY)^+=Y^+ X^+$) which extends the automorphism $X\mapsto -X$ of $\gf$. 
		
		We denote with $X^\dag$ the formal transpose of $X$ (with respect to $\beta$), that is, the unique left-invariant differential operator such that
		\[
		\int_G (X f) g\,\dd \beta=\int_G f X^\dag g\,\dd \beta
		\]
		for every $f,g\in C^\infty_c(G)$. We denote with $X^*$ the formal adjoint of $X$, that is, $\overline X^\dag$. We define the formal transpose and the formal adjoint of right-invariant differential operators in a similar way. If $u$ is a distribution, we then define $X u$ so that
		\[
		\langle X u,\phi\rangle =\langle u, X^\dag \phi\rangle
		\]
		for every $\phi\in C^\infty_c(G)$. In this way, if $f\in C^\infty(G)$, then $(X f)\cdot \beta=X(f\cdot \beta)$.

		In order to simplify the notation, we write $X^{R\dag}$ instead of $(X^R)^\dag$, etc.
	\end{deff}
	
	Notice that $X f$ \emph{depends} on $\beta$ if $f$ is a distribution, whereas $X^\dag f$ does not. On the contrary, $X^\dag f$ \emph{depends} on $\beta$ if $f\in C^\infty(G)$, whereas $X f$ does not.  This unfavorable dichotomy disappears if $\beta$ is right-invariant. In the following proposition we collect some elementary facts; cf.~\cite[Proposition 2.2]{BCP} for a proof.
	
	\begin{prop}\label{prop:9}
		The following hold:
		\begin{enumerate}
			\item[\textnormal{(1)}] $X^\dag f= \Delta_R^{-1} X^+(\Delta_R f)$ and $X^{R\dag} f=\Delta_L^{-1} X^{+R}(\Delta_L f)$ for every $X\in U(G)$ and for every $f\in C^\infty(G)$;
			
			\item[\textnormal{(2)}] $X\delta_e= \Delta_R \Rc(X^\dag \delta_e)$  and $X^R\delta_e=  \Delta_L \Rc (X^\dag \delta_e)$   for every $X\in U(G)$;
			
			\item[\textnormal{(3)}] $X^+ f=(X^R \check f) \check{\,}$ for every $X\in U(G)$ and for every $f\in C^\infty(G)$; 
			\item[\textnormal{(4)}] $X\delta_e=X^{\dag R\dag}\delta_e$ for every $X\in U(G)$;
			
			\item[\textnormal{(5)}] $X\chi= (X\chi)(e)\chi$ for every $X\in U(G)$ and for every character $\chi\colon G\to \C\setminus \Set{0}$.
		\end{enumerate}
	\end{prop}
	
	Notice that, by (4), 
	\[
	\begin{split}
		(Xf)*g=(f*X\delta_e)*g=f*(X\delta_e*g)=f*(X^{\dag R \dag} \delta_e*g)= f*(X^{\dag R \dag} g)
	\end{split}
	\]
	under some reasonable conditions on $f$ and $g$ (which are needed to grant associativity of convolution).
	Notice that $X^{\dag R \dag}-X^R$ is a differential operator whose order is strictly less than the order of $X$ (and whose degree is strictly less than the degree of $X$, with the terminology of the following subsection).

	\subsection{Filtrations and Weighted Subcoercive Operators}
	
	Throughout the paper, $(\gf_\lambda)_{\lambda\Meg 0}$ will denote an increasing filtration of $\gf$ (that is, $[\gf_\lambda, \gf_\mi]\subseteq \gf_{\lambda+\mi}$ for every $\lambda, \mi\Meg 0$) such that $\gf_\lambda=0$ for every $\lambda<1$, $\bigcup_{\lambda\Meg 0} \gf_\lambda=\gf$, and $\bigcap_{\mi>\lambda} \gf_\mi=\gf_\lambda$ for every $\lambda\Meg 0$.\footnote{We consider the full range $\lambda\Meg 0$  for notational convenience, in analogy with the filtration $(U_\lambda)$, for which $U_\lambda\neq \Set{0}$ for every $\lambda\Meg 0$.}
	
	For every $X\in \gf$, we define  $\deg X\coloneqq \min\Set{\lambda\Meg0\colon X\in \gf_\lambda}$ and we call $\deg X$ the degree of $X$.
	Define, for every $\lambda>0$, $\gf_{\lambda^-}\coloneqq \bigcup_{\mi<\lambda} \gf_\mi$, $\gf_{*,\lambda}\coloneqq\gf_\lambda/\gf_{\lambda^-} $, and 
	\[
	\gf_*\coloneqq \bigoplus_{\lambda>0} \gf_{*,\lambda}.
	\]
	Define a Lie algebra structure on $\gf_*$ as follows: if $X=\sum_{\lambda>0} (X_\lambda+ \gf_{\lambda^-})$ and $Y= \sum_{\lambda>0} (Y_\lambda+ \gf_{\lambda^-})$ for some $(X_\lambda),(Y_\lambda)\in \prod_{\lambda>0} \gf_\lambda$, then
	\[
	[X,Y]\coloneqq \sum_{\lambda,\mi>0}\left(  [X_{\lambda},Y_{\mi}]+\gf_{(\lambda+\mi)^-}\right) .
	\]
	It is easily seen that $(\gf_{*,\lambda})_{\lambda>0}$ is a graduation of type $((0,+\infty),+)$ of $\gf_*$, that is, $[\gf_{*,\lambda}, \gf_{*,\mi}]\subseteq \gf_{*,\lambda+\mi}$ for every $\lambda,\mi>0$. One may then endow $\gf_*$ with the automorphic dilations $(\delta_r)_{r>0}$ defined so that $\delta_r(X)=r^\lambda X$ for every $X\in \gf_{*,\lambda}$ and for every $\lambda>0$. Thus, $\gf_*$ is the Lie algebra of some homogeneous group $G_*$, with homogeneous dimension $Q_*\coloneqq \sum_{\lambda>0} \dim \gf_{*,\lambda}$.
	\emph{We shall always assume, in the sequel, that $\Lambda\coloneqq \Set{\lambda \Meg 1\colon \gf_{*,\lambda}\neq 0}$ is contained in a one-dimensional $\Q$-vector space. In other words, setting $\dd_0\coloneqq \min \Lambda$, we assume that $\Lambda \subseteq \Q \dd_0$. This assumption guarantees that (positive) Rockland operators on $G_*$ do exist (and is essentially necessary).}
	
	We observe explicitly that we required $\gf_\lambda=\Set{0}$ for $\lambda<1$ in order for the control modulus $\abs{\,\cdot\,}_*$ (cf.~Definition~\ref{def:3} below) to induce a left-invariant \emph{distance} on $G$ (rather than a quasi-distance). There may be also good reasons to require $\gf_1\neq \Set{0}$ (cf.~Theorem~\ref{teo:11}). We preferred to avoid imposing this condition in order to keep a natural comparison with graded (or, more generally, homogeneous) groups: if $\gf$ has a graduation $(\tilde \gf_j)_{j\in \Z_+^*}$ (with integer degrees), then it is natural to set $\gf_\lambda=\bigoplus_{j=1}^{[\lambda]} \tilde \gf_j$ for every $\lambda \Meg 0$, but there is no guarantee that $\tilde \gf_1$ should be non-trivial.

	We now extend this filtration to the  enveloping algebra $U(G)$.	
	For every $\lambda\Meg 0$, define $U_\lambda$ as the vector space generated by the products of the form $X_1\cdots X_k$, for $k\Meg 0$, $X_1,\dots, X_k\in \gf$ and $\deg X_1+\cdots+ \deg X_k\meg \lambda$.\footnote{Thus, the identity operator, corresponding to the case $k=0$, belongs to all $U_\lambda$.} 
	Then, $(U_\lambda)$ is an increasing filtration of $U(G)$, that is, $U_\lambda, U_\mi\subseteq U_\lambda U_\mi\subseteq U_{\lambda+\mi}$ for every $\lambda, \mi\Meg 0$.
	For every $X\in U(G)$, we define $\deg X\coloneqq \min\Set{\lambda\Meg 0\colon X\in U_\lambda}$. Notice that $\gf\cap U_\lambda=\gf_\lambda$ for every $\lambda\Meg 0$ as a consequence of Proposition~\ref{prop:8} below (and its proof), so that this definition is consistent with the previous one.
	
	Define $U_{\lambda^-}\coloneqq \bigcup_{\mi<\lambda} U_\mi$ for $\lambda>0$, $U_{0^-}\coloneqq\Set{0}$, $U_{*,\lambda}\coloneqq U_\lambda/U_{\lambda^-}$ for every $\lambda\Meg 0$, and 
	\[
	U_*\coloneqq \bigoplus_{\lambda\Meg 0} U_{*,\lambda} .
	\] 
	We define an algebra structure on $U_*$ as follows:  if $X=\sum_{\lambda\Meg 0} (X_\lambda+ U_{\lambda^-})$ and $Y= \sum_{\lambda\Meg 0} (Y_\lambda+ U_{\lambda^-})$ for some $(X_\lambda),(Y_\lambda)\in \prod_{\lambda\Meg 0} U_\lambda$, then
	\[
	XY\coloneqq \sum_{\lambda,\mi\Meg 0}\left(  X_{\lambda}Y_{\mi}+U_{(\lambda+\mi)^-}\right) .
	\]
	It is easily seen that $U_*$ becomes a graded algebra of type $([0,+\infty),+)$ with this structure.
	
	\begin{prop}\label{prop:8}
		The canonical inclusions $\gf_\lambda \subseteq U_\lambda$, $\lambda> 0$, induce a linear mapping $\pi \colon \gf_*\to U_*$. The canonical extension $U(\pi)\colon U(G_*)\to U_*$ of $\pi$ is an  isomorphism of graded algebras.
	\end{prop}
	Cf.~\cite[Proposition 4.2]{BCP} for a proof of this result.
	From now on, we shall identify $U_*$ and $U(G_*)$ by means of $U(\pi)$.

	\begin{deff}
		We say that a basis $(X_j)_{j\in J}$ of $\gf$ is compatible with the filtration $(\gf_\lambda)$ if $\gf_\lambda$ is generated by the $X_j$ with $\deg(X_j)\meg \lambda$ for every $\lambda \Meg 1$.
	\end{deff}
	
	By~\cite[Proposition 4.3]{BCP}, $(X_j)$ is a basis adapted to the filtration if and only if the $Y_j=X_j+\gf_{\deg(X_j)^-}$, $j\in J$, form a homogeneous basis of $\gf_*$.
	
	\begin{deff}		
		We say that a family $(X_j)_{j\in J}$ of elements of $\gf$ is a minimal basis if, setting $Y_j\coloneqq X_j+ \gf_{(\deg X_j)^-}$, the family $(Y_j)$ induces a (homogeneous) basis of $\gf_*/[\gf_*,\gf_*]$.
		We denote with $\dd$ the least common multiple of the degrees of the $X_j$, $j\in J$.\footnote{Notice that $\dd$ is well defined since by our assumption the $d_j$ are all integer multiples of some element of $(0,+\infty)$.}
		
		We say that $\Lc $ is weighted subcoercive if $ \Lc+\Lc^*+U_{\grado^-}$ is a positive Rockland operator on $G_*$.
	\end{deff}

	If $(X_j)$ is a minimal basis of $\gf$, then $(X_j)$ generates $\gf$ as a Lie algebra, and also generates the filtrations $(\gf_\lambda)$ and $U_\lambda$ (cf.~\cite[Proposition 4.4]{BCP}). In other words, $\gf_\lambda$  is the vector space generated by the vector fields of the form $\ad(X_{j_1})\cdots \ad(X_{j_{k-1}})X_{j_k}$, where $k\Meg 1$, $j_1,\dots, j_k\in J$, and $\deg(X_{j_1})+\cdots +\deg(X_{j_k})\meg \lambda$. Analogously, $U_\lambda$ is the vector space generated by the differential operators for the form $X_{j_1}\cdots X_{j_k}$, where $k\Meg 0$, $j_1,\dots, j_k\in J$, and $\deg(X_{j_1})+\cdots +\deg(X_{j_k})\meg \lambda$.
	In particular, $(X_j)$ is a `reduced weighted algebraic basis' of $\gf$, in the terminology of~\cite{ElstRobinson}.
	
	Observe that $\dd$ does \emph{not} depend on the choice of $(X_j)$, since it is the least common multiple of the degrees of the non-zero elements of $\gf_*/[\gf_*,\gf_*]$.

	\begin{prop}\label{prop:4}
		Let $(X_j)_{j\in J}$ be a minimal basis of $\gf$. Then, $\Lc\coloneqq \sum_{j\in J} (X_j^{k\dd/d_j})^\dag X^{k\dd/d_j}_j$, where $d_j=\deg X_j$ for every $j\in J$, is a real,  positive and formally self-adjoint  weighted subcoercive operator for every integer $k\Meg 1$.
	\end{prop}
	
	Here, by `positive' we mean that $\int \Lc f \overline f \,\dd \beta\Meg 0$ for every $f\in C^\infty_c(G)$.
	
	Notice that, if $\gf_1$ is a vector subspace of $\gf$ which generates $\gf$ as a Lie algebra, and if $(\gf_\lambda)$ is the filtration generated by $\gf_1$ (that is, $\gf_\lambda$ is the vector space generated by the elements of $\gf_1$ -- if $\lambda\Meg 1$ -- and their commutators up to order $[\lambda]$), then any basis of $\gf_1$ is a minimal basis (and conversely), and the operator $\Lc$ defined above is a sub-Laplacian (with drift, unless $\beta$ is right-invariant). This observation provides the link with the theory developed in~\cite{BPV,BPV2,BPV3}.

	We shall now define a control modulus on $G$. 
	\begin{deff}\label{def:3}
		Given an absolutely continuous curve $\gamma\colon [0,1]\to G$, we  define the content of $\gamma$ as the greatest lower bound of the $\eps>0$ such that 	
		\[
		\abs{P_\lambda \dd L_{\gamma(t)}^{-1}\gamma'(t)} \meg \min(\eps, \eps^{\lambda}) 
		\]
		for almost every $t\in [0,1]$ and for every $\lambda>0$, where $P_\lambda$ is the orthogonal projector of $\gf$ onto $\gf_\lambda \ominus \gf_{\lambda^-}$ (this is non-zero only for finitely many $\lambda>0$) and $L_{\gamma(t)}$ is the left translation by $\gamma(t)$. 
		
		Given $x\in G$, we shall define $\abs{x}_*$  as the greatest lower bound of the contents of the absolutely continuous curves $\gamma\colon [0,1]\to G$ such that $\gamma(0)=e$ and $\gamma(1)=x$.
		
		We  endow $G$ with the left-invariant distance $d(x,y)\coloneqq \abs{y^{-1}x}_*$.
	\end{deff} 
	
	Choosing a different scalar product on $\gf$   gives rise to bi-Lipschitz equivalent control moduli. Equivalence at infinity follows from the fact that all these control moduli are `connected moduli,'\footnote{In fact, every $x\in G$ may be written as $x_1\cdots x_k$, where $\abs{x_j}_*\meg 1$ for $j=1,\dots, k$, and $k\meg \abs{x}_*+1$.}  whereas equivalence near $e$ follows from~\cite[Corollary 6.5]{ElstRobinson}.  
	Notice that here we are not requiring $\gamma$ to be `horizontal.' One may also require $\gamma$ to be horizontal with respect to some fixed weighted algebraic basis compatible with the filtration $(\gf_\lambda)$ (for instance, a minimal basis), in the terminology of~\cite{ElstRobinson}, and still get an equivalent control modulus. This would provide a better mean value theorem for the corresponding `horizontal gradient,' but we shall not need this kind of more precise estimates.

	It is known that 
	\[
	\beta(B(e,r))\asymp r^{Q_*}, \qquad r\to 0^+,
	\]
	while
	\[
	\beta(B(e,r))\meg \ee^{C r}
	\]
	for some constant $C>0$ and for every $r\Meg 1$ (cf., e.g.,~\cite[Section 4.3]{BCP}).

	\begin{deff}
		From now on, we shall fix a  weighted subcoercive operator $\Lc$ with degree $\grado$. We shall denote with $(h_t)_{t>0}$ the corresponding heat kernel. In other words, $\ee^{-t\Lc}f=f*h_t$ for every $f\in L^2(\beta)$, where $(\ee^{-t\Lc})_{t>0}$ is the semigroup generated by the closure of $\Lc$, with initial domain $C^\infty_c(G)$, in $L^2(\beta)$ (cf.~Theorem~\ref{teo:7} below).
		
		For every $\omega\in \R$, we shall set $\Lc_\omega \coloneqq \Lc+ \omega I$.
	\end{deff}
	
	Cf.~\cite[Theorems 4.8 and 5.4]{BCP} for a proof of the following result.
	
	\begin{teo}\label{teo:7}
		There is $\omega\in \R$ such that the following hold:
		\begin{enumerate} 
			\item[\textnormal{(1)}] the closure of $\Lc$, with initial domain $C^\infty_c(G)$,  generates a semigroup of operators of $L^2(\beta)$; if $\Lc=\Lc^*$, then $\Lc$ is essentially self-adjoint on $C^\infty_c(G)$;
			
			\item[\textnormal{(2)}] for every $\lambda,\lambda'\Meg 0$ there are $b,C>0$ such that  
			\[
			\abs{X Y^R h_t(x)}\meg C \abs{X}\abs{Y} t^{-(Q_* + \deg(X)+\deg(Y))/\grado} \ee^{\omega t} \ee^{- b (\abs{x}_*^{\grado}/t)^{1/(\grado-1)}}
			\]
			for every $X\in U_\lambda$, for every $Y\in U_{\lambda'}$, and for every $x\in G$;
		\end{enumerate}
	\end{teo}
	 
	\subsection{Gaussian Estimates}

	\begin{deff}
		In order to simplify the notation, given a measure space $(X,\mi)$, $p\in (0,\infty]$, and a $\mi$-measurable function $f$ on $X$, we shall also write $\norm{f(x)}_{L^p_x(\mi)}$ instead of $\norm{f}_{L^p(\mi)}$ (we are therefore allowing the possibility $\norm{f}_{L^p(\mi)}=+\infty$). This will be particularly useful in the presence of nested $L^p$-norms.
	\end{deff}
	
	\begin{deff}
		For every $b,t>0$ and for every $d>1$, define 
		\[
		p_{b,t,d}\colon x\mapsto t^{-Q_*/d} \ee^{-b   (\abs{x}_*^{d}/t)^{1/(d-1)}}
		\]
		and
		\[
		T_{b,t,d}f\coloneqq \abs{f}*p_{b,t,d}
		\]
		for every $\beta$-measurable function $f$ on $G$. 
		We shall simply write $T_{b,t}$ and $p_{b,t}$ instead of $T_{b,t,\grado}$ and $p_{b,t,\grado}$, respectively.
	\end{deff}
	
	We now collect some elementary facts. Cf.~\cite[Lemmas 5.5, 5.6, 5.7, 5.8]{BCP} for the relative proofs.
	
	\begin{lem}\label{lem:18}\label{lem:32}\label{lem:3}\label{cor:3}\label{lem:4}\label{lem:4bis}
		Take $b, b'>0$, $\rho\Meg 0$, and $d'\Meg d>1$. Then, the following hold:
		\begin{itemize}
			\item if $b\Meg b'$, then $		p_{b,t,d}\meg \ee^{b'} p_{b',t^{d'/d},d'}$	for every $t>0$;
			
			\item for every $x,y\in G$ and for every $\beta$-measurable function $f$ on $G$,
			\[
			\ee^{-2^{1/(d-1)}b(\abs{y^{-1}x}^d/t)^{1/(d-1)}}(T_{2^{1/(d-1)}b,t,d} f)(y)\meg (T_{b,t,d} f)(x)\meg \ee^{2^{1/(d-1)}b(\abs{y^{-1}x}^d/t)^{1/(d-1)}} (T_{2^{-1/(d-1)}b,t,d} f)(y);
			\]
			
			\item there is a constant $C>0$ such that
			\[
			\norm*{ p_{b,t,d}(x)\ee^{\rho \abs{x}_*}}_{L^p_x(\beta)}\meg C t^{-Q_*/(d p')}\ee^{C t}
			\]
			for every $t>0$ and for every $p\in [1,\infty]$;
			
			\item the mapping $(0,+\infty)\ni t \mapsto \ee^{\rho\abs{\,\cdot\,}_*}p_{t,b,d}\in L^p( \beta)$ is continuous for every $p\in [1,\infty]$;
			
			\item for every $b''\in (0,2^{-1/(d-1)}\min(b,b'))$  there is a constant $C>0$ such that
			\[
			\int_G p_{b,t,d}(x y^{-1}) p_{b',t',d}(y) \ee^{\rho\abs{y}_*}\,\dd \beta(y) \meg C \ee^{C (t+t')} p_{ b'',t+t',d}(x)
			\]
			for every $x\in G$ and for every $t,t'>0$;
			
			\item for every $b''> 2^{1/(d-1)}\max(b,b')$ there is a constant $C>0$ such that
			\[
			\int_G p_{b,t,d}(x y^{-1}) p_{b',t',d}(y) \ee^{-\rho\abs{y}_*}\,\dd \beta(y) \Meg C \ee^{-C(t+t')} p_{ b'',t+t',d}(x)
			\]
			for every $x\in G$ and for every $t,t'>0$.
		\end{itemize}
	\end{lem}

	We shall also need the following elementary lemma (cf.~\cite[Lemma 4.14]{BCP}).
	
	\begin{lem}\label{lem:31}
		Let $\chi\colon G\to \C\setminus \Set{0}$ be a character of $G$. Then, there is $c>0$ such that $\abs{\chi(x)}\meg \ee^{c(\abs{x}_*+1)}$ for every $x\in G$.
	\end{lem}
	
	We shall mainly apply this lemma to the characters $\Delta_L$ and $\Delta_R$.

	\begin{deff}
		Take $ b>0$, $\kappa>0$, and $d>1$. Then, we define  
		\[
		T^*_{b,\kappa,d} f(x) \coloneqq  \sup_{t\in (0,\kappa]}  (T_{b,t,d} f)(x)
		\]
		for every $\beta$-measurable function $f$ on $G$ and for every $x\in G$. 
		
		We shall generally omit $d$ if it is $\grado$.
	\end{deff}

 	Cf.~\cite[Lemmas 5.10 and 5.11]{BCP} for a proof of the following result.
	
	\begin{lem}\label{lem:5}
		Take $b>0$, $\kappa > 0$, $d>1$, $p,q\in (1,\infty)$, and a $\sigma$-finite measure space $(Y,\Mf, \mi)$. Then there is a constant $C>0$ such that
		\[
		\norm*{ (T^*_{b,\kappa, d} f(y,\,\cdot\,))(x)}_{L^{q,p}_{y,x}(\mi,\beta)}\meg C\norm*{f}_{L^{q,p}(\mi,\beta)}
		\]
		for every   $f\in L^{q,p}(\mi,\beta)$.  
	\end{lem}

	\begin{lem}\label{lem:2} 
		There is $\omega\in \R$ such that the following hold.
		Take $p\in (1,\infty)$, $q\in [1,\infty]$, $b>0$, $d>1$,   a $\sigma$-finite measure space $(Y,\Mf, \mi)$,  and a $\mi$-measurable function $\kappa\colon Y\to (0,+\infty)$. 
		Then there is   $C>0$ such that, for every $(\mi\otimes \beta)$-measurable function $f$ on $Y\times G$,
		\[
		\norm*{   \ee^{-\omega\kappa(y)} [T_{b,\kappa(y),d} f(y,\,\cdot\,)](x)  }_{L^{q,p}_{y,x}(\mi,\beta)}\meg  C \norm{f}_{L^{q,p}(\mi,\beta)}.
		\]
	\end{lem}

\subsection{The `Schwartz Space'}

\begin{deff}\label{def:1}
	We define $\Sc(G)$ as the space of $f\in C^\infty(G)$ such that the seminorms $\norm{\ee^{c \abs{\,\cdot\,}_*} X^R f}_{L^1(\beta)}$, $c>0$, $X\in U(G)$, are finite, endowed with the corresponding topology.
	
	We denote with $\Sc'(G)$ the dual of $\Sc(G)$, endowed with the topology of uniform convergence on the bounded subsets of $\Sc(G)$.
\end{deff}

Cf.~\cite[Theorem 4.16, Corollary 4.17, and Lemma 5.13]{BCP} for a proof of the following results.
	
\begin{teo}\label{teo:8}
	The following hold:
	\begin{enumerate}
		\item[\textnormal{(1)}] $\Sc(G)$ is a nuclear Fréchet space;
		
		\item[\textnormal{(2)}] $\Sc(G)$ is reflexive;
		
		\item[\textnormal{(3)}] the bounded subsets of $\Sc(G)$ are relatively compact;
		
		\item[\textnormal{(4)}] $\Sc(G)$ is a Fréchet $*$-algebra under convolution and under pointwise multiplication;
		
		\item[\textnormal{(5)}] for every $p\in [1,\infty]$, $\Sc(G)$ is the space of $f\in C^\infty(G)$ such that the seminorms  $\norm{\ee^{c \abs{\,\cdot\,}_*} X Y^R f}_{L^p(\beta)}$, $c>0$, $X,Y\in U(G)$ (resp.\ $\norm{\ee^{c \abs{\,\cdot\,}_*} X   f}_{L^p(\beta)}$, $c>0$, $X\in U(G)$; $\norm{\ee^{c \abs{\,\cdot\,}_*} X^R f}_{L^p(\beta)}$, $c>0$, $X\in U(G)$), are finite, and has the corresponding topology;
		
		\item[\textnormal{(6)}] 	$\Sc'(G)$ is a complete, reflexive, bornological, and nuclear space.  The bounded subsets of $\Sc'(G)$ are relatively compact.
	\end{enumerate} 
\end{teo}

\begin{lem}\label{lem:9}
	Take $F\in \Set{\Sc(G), \Sc'(G)}$ and $f\in F$. Then, $f*h_t$ is well-defined for every $t>0$. In addition, the mapping $[0,+\infty)\ni t\mapsto f*h_t\in F$ is of class $C^\infty$, and $\frac{\dd^k}{\dd t^k} (f*h_t)=\Lc^k (f*h_t)$   for every $k\in\N$ (with the convention $h_0=\delta_e$). In particular, for every $t>0$, and for every $m\in\N$,
	\[
	f=\sum_{k=0}^m \frac{1}{k!} W_t^{(k)}f + \frac{1}{m!} \int_0^t W_s^{(m+1)} f  \,\frac{\dd s}{s}.
	\]
\end{lem}

\subsection{Supplementary Lemmas}
 
\begin{deff}
	We denote with $\Mc_\Car$ (`Carleson measures') the space of positive Radon measures $\mi$ on $(0,+\infty)$ with bounded support such that there is a constant $C>0$ such that
	\[
	\mi((r,2r])\meg C
	\]
	for every $r>0$.
	
	We denote with $\Mc_\RC$  (`reverse Carleson measures') the space of positive Radon measures $\mi$ on $(0,+\infty)$ such that there are constants $\eps,C\in (0,1)$   such that
	\[
	\mi((\eps r, r])\Meg C
	\]
	for every $r\in (0,\eps]$.
	
	We define $\Mc_\Samp\coloneqq \Mc_\Car\cap \Mc_\RC$  (`sampling measures').
	
	Finally, for every $\kappa>0$ we define $\mi_\kappa$ as the Radon measure on $(0,+\infty)$  supported in $(0,\kappa]$ and such that $\dd\mi_\kappa(t)=\frac{\dd t}{t}$ on $(0,\kappa]$.
\end{deff}

Cf.~\cite[Lemmas 6.4 and 6.6]{BCP} for a proof of the following results.

\begin{lem}\label{lem:25}
	Take $\eta,\gamma>0$,  $\mi,\nu\in \Mc_\Car$. Then, there is a constant $C>0$ such that
	\[
	\norm*{\int_0^\infty \frac{s^\gamma t^{\eta}}{(s+t)^{\eta+\gamma}} \abs{f(t)}\,\dd \nu(t)  }_{L^q_s(\mi)} \meg C \norm{f}_{L^q(\nu)}
	\]
	for every  $q\in [1,\infty]$ and for every   $\nu$-measurable function $f$.
\end{lem}

\begin{lem}\label{lem:25bis}
	Take $\eta,\gamma>0$ and  let $\mi$ be a Haar measure on $(0,+\infty)$. Take $\nu\in \Mc_\Car$, and assume that either $\eta\Meg 1$ or $\nu\in L^\infty(\mi)\cdot \mi$.	
	Then, there is a constant $C>0$ such that
	\[
	\norm*{\int_0^\infty t^\eta s^\gamma \abs{f(s+t)}\,\dd \mi(t)  }_{L^q_s(\nu)} \meg C \norm{ t^{\gamma+\eta}f(t)}_{L^q_t(\mi)}
	\]
	for every  $q\in [1,\infty]$ and for every  $\mi$-measurable function $f$.
\end{lem}

\subsection{Besov and Triebel-Lizorkin Spaces}

\begin{deff}\label{def:2}
	Define, for every $k\in \N$, for every $t>0$, and for every $f\in \Sc'(G)$,
	\[
	W_t^{(k)} f\coloneqq (t \Lc)^k \ee^{-t \Lc} f.
	\]
	In addition, define,  for every $\alpha\in\R$ and for every $\eps\in (0,1]$,
	\[
	W_{t,\eps}^{(\alpha),*} f\coloneqq t^\alpha \max_{s\in [\eps t,t/\eps ]} \max_{\substack{\deg(X)+\deg(Y)\meg \alpha \grado \\ \abs{X} ,\abs{Y}\meg 1}} \abs{X \ee^{-s \Lc} Y f }.
	\]
\end{deff}

\begin{deff}
	For every $\alpha \in \R$, for every $m\in\N$ with $m>\alpha/\grado$, for every $\mi\in \Mc_\Samp$ and for every $p,q\in [1,\infty]$,  define
	\[
	\Bs^{p,q}_{\alpha,m,\mi}(f)\coloneqq \norm{   t^{-\alpha/\grado} \norm{ W_t^{(m)} f }_{L^p(\beta)}  }_{L^q_t(\mi)},
	\]
	and, if $p\in (1,\infty)$,
	\[
	\Fs^{p,q}_{\alpha,m,\mi}(f)\coloneqq \norm*{ \norm{ t^{-\alpha/\grado}  W_t^{(m)} f(x)  }_{L^q_t(\mi_1)}}_{L^p_x(\beta)} ,
	\]
	Then, define, for every $f\in \Sc'(G)$,
	\[
	\norm{f}_{B^{p,q}_\alpha(\beta)}\coloneqq \norm{\ee^{-\Lc} f}_{L^p(\beta)}+ \Bs^{p,q}_{\alpha,([\alpha/\grado]+1)_+,\mi_1}(f),
	\]
	and, if $p\in (1,\infty)$,
	\[
	\norm{f}_{F^{p,q}_\alpha(\beta)}\coloneqq \norm{\ee^{-\Lc} f}_{L^p(\beta)}+ \Fs^{p,q}_{\alpha,([\alpha/\grado]+1)_+,\mi_1}(f).
	\] 
	We define $B^{p,q}_\alpha(\beta)$ and $F^{p,q}_\alpha(\beta)$ accordingly. We define $\mathring B^{p,q}_\alpha(\beta)$ and $\mathring F^{p,q}_\alpha(\beta)$ as the closures of $\Sc(G)$ in $B^{p,q}_\alpha(\beta)$ and $F^{p,q}_\alpha(\beta)$, respectively.
\end{deff}

By~\cite[Proposition 6.9]{BCP}, the norms $\norm{\ee^{-\Lc} \,\cdot\,}_{L^p(\beta)}+\Bs^{p,q}_{\alpha,m,\mi}$ are all equivalent on $\Sc'(G)$, as well as the norms $\norm{\ee^{-\Lc} \,\cdot\,}_{L^p(\beta)}+\Fs^{p,q}_{\alpha,m,\mi}$. Hence, they may be used to define the spaces $B^{p,q}_\alpha(\beta)$ and $F^{p,q}_\alpha(\beta)$, respectively; the same holds if we replace $W^{(m)}$ with $W^{(m),*}$ in the definition of $\Bs$ and $\Fs$.
 
Let us also recall the following simple result (cf.~\cite[Proposition 7.4]{BCP}) which will be often useful to reduce to the case in which $\beta$ is left invariant.

\begin{prop}\label{prop:5}
	Suppose $\alpha\in \R$ and $p,q\in [1,\infty]$. Then 
	\begin{enumerate}
		\item[\textnormal{(i)}]  $B^{p,q}_\alpha(\beta)= \Delta_L^{-1/p} B^{p,q}_\alpha(\beta_L)$, and
		\item[\textnormal{(ii)}]  $F^{p,q}_\alpha(\beta)= \Delta_L^{-1/p} F^{p,q}_\alpha(\beta_L)$  if $p\in (1,\infty)$.
	\end{enumerate}
\end{prop}

We shall also recall the following `Sobolev embeddings' (cf.~\cite[Theorem 8.2 and Proposition 8.3]{BCP}).

\begin{prop}\label{prop:2}
	Suppose $p_1,p_2,q_1,q_2\in [1,\infty]$ and $\alpha_1,\alpha_2\in \R$. Then, the following continuous embeddings hold.
	\begin{enumerate}
		\item[\textnormal{(1$\:$)}] $B^{p_1,q_1}_{\alpha_1}(\beta)\subseteq B^{p_1,q_2}_{\alpha_2}(\beta)$ if either $\alpha_2<\alpha_1$, or $\alpha_2=\alpha_1$ and $ q_1\meg q_2$.
		
		\item[\textnormal{(1$'$)}] $F^{p_1,q_1}_{\alpha_1}(\beta)\subseteq F^{p_1,q_2}_{\alpha_2}(\beta)$ if $p_1\in (1,\infty)$ and either $\alpha_2<\alpha_1$ or $\alpha_2=\alpha_1$ and $ q_1\meg q_2$.
		
		\item[\textnormal{(2$\:$)}] $\Delta_L^{1/p_1-1/p_2} B^{p_1,q_1}_{\alpha_1}(\beta)\subseteq B^{p_2,q_1}_{\alpha_2}(\beta)$ if $p_1\meg p_2$ and $\alpha_1-\frac{Q_*}{p_1}\Meg\alpha_2-\frac{Q_*}{p_2} $.
		
		\item[\textnormal{(2$'$)}] $\Delta_L^{1/p_1-1/p_2} F^{p_1,q_1}_{\alpha_1}(\beta  )\subseteq F^{p_2,q_2}_{\alpha_2}(\beta )$  if $1<p_1< p_2<\infty$ and $\alpha_1-\frac{Q_*}{p_1}\Meg\alpha_2-\frac{Q_*}{p_2} $.
		\item[\textnormal{(3$\:$)}] $B^{p_1,1}_{0}(\beta)\subseteq L^{p_1}(\beta)$, and if $p_1<p_2$ and $\alpha_1> \frac{Q_*}{p_1}-\frac{Q_*}{p_2} $, then $\Delta_L^{1/p_1-1/p_2}B^{p_1,q_1}_{\alpha_1}(\beta )\subseteq   L^{p_2}(\beta )$.
		
		\item[\textnormal{(3$'$)}]  	$\Delta_L^{1/p_1-1/p_2}F^{p_1,q_1}_{\alpha_1}(\beta )\subseteq L^{p_2}(\beta )$ if $1<p_1< p_2$ and $\alpha_1\Meg  \frac{Q_*}{p_1}-\frac{Q_*}{p_2} $, and either $p_2<\infty$ or $\alpha_1>  \frac{Q_*}{p_1}-\frac{Q_*}{p_2} $.
	\end{enumerate}
	In addition,
	\[
	F^{p,2}_0(\beta)=L^p(\beta)
	\]
	for every $p\in (1,\infty)$.
\end{prop}

\subsection{`Goodman' Sobolev Spaces}

\begin{deff}
	Suppose  $\alpha\Meg0$ and $p\in [1,\infty]$. We define $L^p_0(\beta)$ as the closure of $C_c(G)$ in $L^p(\beta)$. Then, we define the Sobolev space $W^{\alpha,p}(\beta)$ (resp.\ $W^{\alpha,p}_0(\beta)$) as the space of $f\in L^p(\beta)$ (resp.\ $f\in L^p_0(\beta)$) such that $X f\in L^p(\beta)$ (resp.\ $X f\in L^p_0(\beta)$) for every $X\in U_\alpha$, endowed with the norm 
	\[
	f\mapsto \max_{\substack{X\in U_\alpha, \;  \abs{X}\meg 1}} \norm{X f}_{L^p(\beta)}.
	\]
	We shall also define $W^{\infty,p}(\beta)\coloneqq \bigcap_{\alpha \Meg 0} W^{\alpha,p}(\beta)$, endowed with the corresponding topology.
\end{deff}
Observe that by definition $L^p_0(\beta)= L^p(\beta)$ if $p<\infty$, while $L^\infty_0(\beta) =  C_0(G)$.  

Cf.~\cite[Proposition 9.2]{BCP} for a proof of the following result.

\begin{prop}\label{prop:7}
	Suppose $\alpha\Meg 0$, $\alpha'\in \R$ such that $([\alpha'/\grado]+1)_+\meg \alpha/\grado$, and $p,q\in [1,\infty]$. Then, the following hold.
	\begin{enumerate}
		\item[\textnormal{(i)}] $W^{\alpha,p}(\beta)$ is a Banach space.
		
		\item[\textnormal{(ii)}]  $W^{\alpha,p}_0(\beta)$ is the closure of $C^\infty_c(G)$  in $W^{\alpha,p}(\beta)$.
		
		\item[\textnormal{(iii)}]  $W^{\alpha,p}(\beta)\subseteq B^{p,q}_{\alpha'}(\beta)$ continuously.
		
		\item[\textnormal{(iv)}]  If $p\in (1,\infty)$, then $W^{\alpha,p}(\beta)\subseteq F^{p,q}_{\alpha'}(\beta)$ continuously.
	\end{enumerate}
	In particular, $\Sc(G)\subseteq B^{p,q}_\alpha(\beta)$, and if $p\in(1,\infty)$ then $\Sc(G)\subseteq F^{p,q}_\alpha(\beta)$, continuously.
\end{prop}

\section{Algebra Properties}\label{sec:3}
 
In this section we provide some results on the continuity of pointwise multiplication of elements of Besov and Triebel--Lizorkin spaces. The natural continuation of this programme is the study of the spaces of pointwise multipliers of such spaces. For Triebel--Lizorkin spaces, it will suffice to have a suitable notion of localization. For Besov spaces, we shall have to work harder and (essentially) repeat the proof of Theorem~\ref{teo:6}.

We observe explicitly that, in order to pursue this programme, it is essential to have some analogue of the case $p_1=p_3=p$ and $p_2=p_4=\infty$ in (ii) of Theorem~\ref{teo:6}. The `correct' analogue may be proved using the appropriate definition of the space $F^{\infty,q}_\alpha(\beta)$, but we  shall rely on the simpler Remark~\ref{oss:4} below, which replaces $F^{\infty,q}_\alpha(\beta)$ with the smaller space $B^{\infty,q}_\alpha(\beta)$. We shall deal with the spaces $F^{\infty,q}_\alpha(\beta)$ in a future work.

\begin{teo}\label{teo:6}
	Take $\alpha> 0$ and $p,p_1,p_2,p_3,p_4,q\in [1,\infty]$ such that $\frac{1}{p_1}+\frac{1}{p_2}=\frac{1}{p_3}+\frac{1}{p_4}=\frac{1}{p}$. Then, there is a constant $C>0$ such that the following hold:
	\begin{enumerate}
		\item[\textnormal{(i)}]   for every $f\in B^{p_1,q}_\alpha(\beta)\cap L^{p_3}(\beta)$ and for every $g\in B^{p_4,q}_\alpha(\beta)\cap L^{p_2}(\beta)$,
		\[
		\norm{f g}_{B^{p,q}_\alpha(\beta)}\meg C \norm{f}_{B^{p_1,q}_\alpha(\beta)}\norm{g}_{L^{p_2}(\beta)}+ C \norm{f}_{L^{p_3}(\beta)} \norm{g}_{B^{p_4,q}_\alpha(\beta)};
		\]
		
		\item[\textnormal{(ii)}] if $p,p_1,p_4\in (1,\infty)$ and $p_2,p_3\in (1,\infty]$,  then for every $f\in F^{p_1,q}_\alpha(\beta)\cap L^{p_3}(\beta)$ and for every $g\in F^{p_4,q}_\alpha(\beta)\cap L^{p_2}(\beta)$,
		\begin{equation}\label{eq:7}
		\norm{f g}_{F^{p,q}_\alpha(\beta)}\meg C \norm{f}_{F^{p_1,q}_\alpha(\beta)}\norm{g}_{L^{p_2}(\beta)}+ C \norm{f}_{L^{p_3}(\beta)} \norm{g}_{F^{p_4,q}_\alpha(\beta)}.
		\end{equation}
	\end{enumerate}
\end{teo}

This result extends~\cite[Theorem 7.1]{BPV} to the case of general filtrations and general $\alpha\in \R$. The proof is similar (except for \textsc{step II}, which is  simpler). In a sense, the most natural cases correspond to the choices $p_1=p_4=p$ (so that $p_2=p_3=\infty$) and $p_1=p_3=p$ (so that $p_2=p_4=\infty$). The first choice shows that $B^{p,q}_\alpha(\beta)\cap L^\infty(\beta)$ and $F^{p,q}_\alpha(\beta)\cap L^\infty(\beta)$ are Banach algebras under pointwise multiplication, whereas the second choice shows that $B^{\infty,q}_\alpha(\beta)$ acts continuously on $B^{p,q}_\alpha(\beta)$ by pointwise multiplication (the analogous assertion for Triebel--Lizorkin spaces will be proved in a future work since it requires to deal with the spaces $F^{\infty,q}_\alpha(\beta)$). 

We begin with the following result, which is a generalization of~\cite[Proposition 5.2]{Feneuil}. We state it in a quite abstract and general way for future reference, but the proof is (formally) unchanged. We mention, though, the we do not take $t'=0$ as in~\cite[Proposition 5.2]{Feneuil}, since in its main application, Corollary~\ref{lem:28} (but also in the cited reference), the mapping $t \mapsto W'(t)v$ is in general \emph{not} $\mi$-integrable on $(0,t]$ for any $t>0$, unless we take $v$ in a suitable vector subspace. Equality may actually hold interpreting the integral as a limit $\lim\limits_{t'\to 0^+} \int_{(t',t]}$ in a possibly weaker topology.

Let us also mention that there is no substantial difference in considering a diffuse measure $\mi$ instead of a Haar measure on $(0,+\infty)$. Since, however, in a future work we shall need to deal  with the discrete case (that is, $\mi=\sum_{j\in\N} \delta_{\eps^j}$, $\eps\in (0,1)$), we shall consider here the case of a general Radon measure $\mi$, even though it makes the proof longer and more involved.

\begin{lem}\label{lem:28b}
	Let $B_1,B_2,B_3$ be three Banach spaces, $\cdot\,\colon B_1\times B_2\to B_3$ a continuous bilinear mapping, and $\mi$  a   Radon measure on $(0,+\infty)$.  Let, for $j=1,2,3$, $W_j,W'_j\colon (0,+\infty)\to \Lc(B_j)$ be two mappings such that, for every $v\in B_j$, the mapping $t\mapsto  W_j'(t) v$ is locally $\mi$-integrable, and  
	\[
	W_j(t') v= W_j(t)v+\int_{(t',t]} W'_j(s)v\,\dd \mi(s)
	\]
	for every $t,t'\in \supp \mi$ with $t'<t$. We define $W^-_j(t)\coloneqq W(t)+\mi(\Set{t})W'(t)$ for every $t>0$ and for $j=1,2,3$.
	Then, for every $v\in B_1$, for every $w\in B_2$, and for every $t,t'\in \supp\mi$ with $t'<t$,
	\[
	W_3(t') (W_1(t')v  \cdot W_2(t') w )=\Pi^{(t',t)}_v w+\Pi^{(t',t)}_w v+\Pi^{(t',t)}(v,w)+\widetilde\Pi^{(t',t)}(v,w)+W_3(t)[(W_1(t) v )\cdot (W_2(t) w)] ,
	\]
	where
	\[
	\begin{aligned}
		\Pi^{(t',t)}_v w&=  \int_{(t',t]} W_3(s)[(W'_1(s) v)\cdot (W_2(s) w)]\,\dd \mi(s),\\
		\Pi^{(t',t)}_w v&=   \int_{(t',t]} W_3(s)[(W_1(s) v)\cdot (W'_2(s) w)]\,\dd \mi(s),\\
	\end{aligned}
	\]
	and
	\[
	\begin{split}
		\Pi^{(t',t)}(v,w)&=  \int_{(t',t]} W'_3(s) [(W_1^-(s) v) \cdot (W_2^-(s) w)]\,\dd \mi(s),\\
		\widetilde\Pi^{(t',t)}(v,w)&=  \int_{(t',t]} W_3(s) [(W_1'(s) v) \cdot (W_2'(s) w)] \mi(\Set{s})\,\dd \mi(s).
	\end{split}
	\]
\end{lem}

Notice that, if we redefine $W_j$ so that $W_j(t')=W_j(t)v+\int_{(t',t]} W'_j(s)v\,\dd \mi(s)$ for every $0<t'\meg t$, then $W^-_j(t)=\lim\limits_{t'\to t^-} W_j(t')$ (pointwise convergence). In addition, the assumptions on $W_j$ essentially mean that $W_j$ (redefined as above) is a right-continuous (for the topology of pointwise convergence on $\Lc(B_j)$) function of locally bounded variation with distributional derivative   $-W_j'\cdot \mi$. The function $W_j^-$ is then the associated left-continuous function of locally bounded variation.

\begin{proof}
	In order to simplify the notation, we shall simply write $v w$ instead of $v\cdot w$ and $W$ and $W'$ instead of $W_j$ and $W'_j$, respectively, for $v\in B_1$, $w\in B_2$, and $j=1,2,3$. We shall also simply write $\mi_t$ instead of $\mi(\Set{t})$ for every $t>0$. Observe that $t \mapsto W(t)v\in B_j$ is $\mi$-measurable and locally bounded for every $v\in B_j$, since it coincides with the locally bounded and right-continuous (hence Borel measurable) function $t\mapsto W(t_0)v-\int_{(t_0,t]} W'(s) v\,\dd \mi(s)$ on $\supp \mi\cap ( t_0,+\infty)$ for every $t_0>0$.
	
	Now, observe that 
	\begin{equation}\label{eq:15}
	\begin{split}
		&W(t')(W(t')v W(t')w)= W(t)(W(t')v W(t')w) +  \int_{(t',t]} W'(s) (W(t')v W(t')w)\,\dd \mi(s)\\
		&\qquad=  W(t)( ( W(t)v) ( W(t)w)) +  W(t)\Big( (W(t) v )\int_{(t',t]} W'(s) w\,\dd \mi (s)\Big) \\
		&\qquad\qquad +W(t)\Big( \int_{(t',t]} W'(s) v\,\dd \mi (s) (W(t) w)  \Big)+ W(t)\Big(\int_{(t',t]} W'(s) v\,\dd \mi (s)\int_{(t',t]} W'(s) w\,\dd \mi (s)  \Big)  \\
		&\qquad\qquad+  \int_{(t',t]} W'(s) ((W(t) v) (W(t) w))\,\dd \mi(s)+ \int_{(t',t]} W'(s) \Big((W(t) v)  \int_{(t',t]} W'(s') w\,\dd \mi(s')\Big)\,\dd \mi(s)\\
		&\qquad\qquad+ \int_{(t',t]} W'(s) \Big( \int_{(t',t]} W'(s') (v )\,\dd \mi(s') (W(t) w)\Big)\,\dd \mi(s)\\
		&\qquad \qquad+ \int_{(t',t]} W'(s) \Big( \int_{(t',t]} W'(s') v\,\dd \mi(s') \int_{(t',t]} W'(s') w\,\dd \mi(s')\Big)\,\dd \mi(s)\\
		&\qquad=W(t)( ( W(t)v) ( W(t)w)) +  \int_{(t',t]} W(t)  [(W(t) v )(W'(s) w)]\,\dd \mi (s)  + \int_{(t',t]} W(t) [(W'(s) v)(W(t) w)] \,\dd \mi (s) \\
		&\qquad\qquad + \int_{(t',t]^2} W(t) [(W'(s) v) (W'(s') w)]\,\dd (\mi\otimes \mi )(s,s')    +  \int_{(t',t]} W'(s) ((W(t) v) (W(t) w))\,\dd \mi(s)\\
		&\qquad\qquad+ \int_{(t',t]^2} W'(s) [(W(t) v)   ( W'(s') w)]\,\dd (\mi\otimes \mi)(s,s') + \int_{(t',t]^2} W'(s)[ (  W'(s') v ) (W(t) w)] \,\dd( \mi\otimes \mi)(s,s')\\
		&\qquad\qquad+ \int_{(t',t]^3} W'(s)  ( (W'(s') v)(W'(s'') w))\,\dd( \mi\otimes \mi\otimes \mi)(s,s',s'').
	\end{split}
	\end{equation}
	In order to proceed, we split   $(t',t]^3$ as $A_1\cup A_2\cup A_3\cup B_1\cup B_2\cup B_3 \cup C$ (disjoint union), where 
	\[
	\begin{aligned}
	A_j&\coloneqq\Set{(s_1,s_2,s_3)\in (t',t]\colon s_j<s_{\sigma(j+1)},s_{\sigma(j+2)} },\\
	B_j&\coloneqq \Set{(s_1,s_2,s_3)\in (t',t]\colon s_{\sigma(j+1)}=s_{\sigma(j+2)}<s_{j} },\\
	C&\coloneqq\Set{(s_1,s_2,s_3)\in (t',t]\colon s_1=s_2=s_3},
	\end{aligned}
	\]
	and where $\sigma\colon \Z\to \Set{1,2,3}$ is defined so that $n-\sigma(n)\in 3 \Z$ for every $n\in\Z$.
	Taking into account the fact that $\int_{(s,t]} W'(s')\,\dd \mi(s')= W(s)-W(t)$ for every $s\in (t',t]\cap \supp \mi$, we then deduce that
	\[
	\begin{split}
		&\int_{(t',t]^3} W'(s_1)  ( (W'(s_2) v)(W'(s_3) w))\,\dd( \mi\otimes \mi\otimes \mi)(s_1,s_2,s_3) \\
		&\qquad=\int_{(t',t]} W'(s)[ (W(s) v-W(t) v  )(W(s) w-W(t)w)] \,\dd \mi(s) \\
		&\qquad\qquad + \int_{(t',t]}  (W(s)-W(t))   [(W'(s) v)(W(s) w- W(t) w)] \,\dd \mi(s)\\
		&\qquad\qquad+\int_{(t',t]} (W(s)-W(t))  [(W(s) v-W(t)v)(W'(s) w)]\,\dd \mi(s)  \\
		&\qquad \qquad+ \int_{(t',t]} \mi_s (W(s)-W(t))((W'(s) v)(W'(s)w))\,\dd \mi(s)\\
		&\qquad \qquad + \int_{(t',t]} \mi_s W'(s) [((W(s)-W(t)) v)(W'(s)w)]\,\dd \mi(s)\\
		&\qquad \qquad + \int_{(t',t]} \mi_s W'(s) [(W'(s)v)((W(s)-W(t)) w)]\,\dd \mi(s)\\
		&\qquad \qquad + \int_{(t',t]} \mi_s^2 W'(s)((W'(s) v)(W'(s) w))\,\dd \mi(s).
	\end{split}
	\]
	In a similar way, one sees that
	\[
	\begin{split}
		&\int_{(t',t]^2} W(t) [(W'(s) v )(W'(s') w)]\,\dd (\mi\otimes \mi )(s,s')  = \int_{(t',t]} W(t) [((W(s)-W(t)) v )(W'(s) w)]\,\dd \mi(s)\\
		&\qquad \qquad + \int_{(t',t]} W(t) [(W'(s) v)((W(s)-W(t)) w )]\,\dd \mi(s)+ \int_{(t',t]} \mi_s W(t)[(W'(s) v )(W'(s) w)]\,\dd \mi(s),
	\end{split}
	\]
	\[
	\begin{split}
		&\int_{(t',t]^2} W'(s) [(W(t) v)    (W'(s') w)]\,\dd (\mi\otimes \mi)(s,s')= \int_{(t',t]} (W(s)-W(t)) [(W(t) v)   ( W'(s) w)]\,\dd \mi(s)\\
		&\qquad \qquad + \int_{(t',t]} W'(s) [(W(t) v)    (W(s)-W(t)) w]\,\dd \mi(s)+\int_{(t',t]} \mi_s W'(s) [(W(t) v)    (W'(s) w)]\,\dd \mi(s),
	\end{split}
	\]
	and
	\[
	\begin{split}
		&\int_{(t',t]^2} W'(s)[   (W'(s') v ) (W(t) w)] \,\dd( \mi\otimes \mi)(s,s')=\int_{(t',t]} (W(s)-W(t))[   (W'(s) v ) (W(t) w)]   \,\dd \mi(s)\\
		&\qquad \qquad+ \int_{(t',t]} W'(s)[   ((W(s)-W(t)) v ) (W(t) w)]   \,\dd \mi(s)+ \int_{(t',t]}\mi_s W'(s)[   (W'(s) v ) (W(t) w)]  \,\dd \mi(s).
	\end{split}
	\]
	In order to perform the remaining computations  in a streamlined way, we write $\mi^k(s,s',t)$ instead of $\int_{(t',t]}\mi_s^k W(s)((W'(s)v) (W(t)w))\,\dd \mi(s)$, $k=0,1,2$, with similar notation for the other combinations of $t,s,s'$. Then, the above computations show that\footnote{We consider first the three single integrals in in~\eqref{eq:15}, then the  triple integral, and finally the three double integrals.}
	\[
	\begin{split}
		&W(t')(W(t')v W(t')w)-W(t)(W(t)v W(t)w)=(t,t,s')+(t,s',t)+(s',t,t)+(s',s,s)-(s',t,s)\\
		&\qquad-(s',s,t)+(s',t,t)+(s,s',s)-(t,s',s)-(s,s',t)+(t,s',t)+(s,s,s')-(t,s,s')-(s,t,s')\\
		&\qquad+(t,t,s')+\mi(s,s',s')-\mi(t,s',s')+\mi(s',s,s')-\mi(s',t,s')+\mi(s',s',s)-\mi(s',s',t)+\mi^2(s',s',s')\\
		&\qquad+(t,s,s')-(t,t,s')+(t,s',s)-(t,s',t)+\mi(t,s',s')+(s,t,s')-(t,t,s')+(s',t,s)-(s',t,t)\\
		&\qquad+\mi(s',t,s')+(s,s',t)-(t,s',t)+(s',s,t)-(s',t,t)+\mi(s',s',t)\\
		&\quad= (s',s,s)+(s,s',s)+(s,s,s')+\mi(s,s',s')+\mi(s',s,s')+\mi(s',s',s)+\mi^2(s',s',s')
	\end{split}
	\]
	since all terms where some $t$ occurs cancel out.
	The assertion follows observing that $\Pi^{(t',t)}(v,w)=(s',s,s)+\mi(s',s',s)+\mi(s',s,s')+\mi^2(s',s',s')$.
\end{proof}

\begin{cor}\label{lem:28}
	Take $m\in\N$, $p,q\in [1,\infty]$ such that $\frac{1}{p}+\frac 1 q \meg 1$, $\rho\Meg 0$, $f\in \ee^{\rho\abs{\,\cdot\,}_*} L^p(\beta)$, and $g\in \ee^{\rho\abs{\,\cdot\,}_*} L^q(\beta)$. Then
	\[
	f g=\lim_{t'\to 0^+}(\Pi_f^{(t')} g+\Pi^{(t')}_g f)+\Pi(f,g)+\sum_{h,k,n=0}^{m} \frac{1}{h! k! n!} W_1^{(h)}[(W^{(k)}_1 f)(W^{(n)}_1 g)],
	\]
	where
	\[
	\begin{aligned}
	\Pi_f^{(t')} g&=\sum_{h,k=0}^{m}\frac{1}{m! h! k!}  \int_{t'}^1 W^{(h)}_t [(W^{(m+1)}_t f)(W^{(k)}_t g)]\,\frac{\dd t}{t},\\
	\Pi_g^{(t')} f&=\sum_{h,k=0}^{m}\frac{1}{m! h! k!}  \int_{t'}^1 W^{(h)}_t [(W^{(k)}_t f)(W^{(m+1)}_t g)]\,\frac{\dd t}{t},
	\end{aligned}
	\]
	and
	\[
	\Pi(f,g)=\sum_{h,k=0}^{m}\frac{1}{m! h! k!} \lim_{t'\to 0^+}\int_{t'}^1 W^{(m+1)}_t [(W^{(h)}_t f)(W^{(k)}_t g)]\,\frac{\dd t}{t},
	\]
	where the limits are taken in $\Sc'(G)$.
\end{cor}

It is unclear to us whether the limits $\lim\limits_{t'\to 0^+} \Pi_f^{(t')} g$ and $\lim\limits_{t'\to 0^+}\Pi^{(t')}_g f$ exist `individually' in the generality of the statement (the issues arise when the terms with $h=0$ are considered). 
In addition, we observe explicitly that we considered the general case $\rho\neq 0$ only because it makes the proof more natural.

\begin{proof}
	Observe first that $p'\meg q$ and that there is $\rho'>0$ such that $\ee^{-\rho'\abs{\,\cdot\,}_*}\in L^1(\beta)$. By H\"older's inequality, this implies that $\ee^{\rho\abs{\,\cdot\,}_*}L^q(\beta)\subseteq \ee^{(\rho+\rho')\abs{\,\cdot\,}_*}L^{p'}(\beta)$, so that, up to replacing $\rho$ with $\rho+\rho'$, we may assume that $q=p'$.
	We wish to apply Lemma~\ref{lem:28b} with $\mi=\mi_1$, $B_1=\ee^{\rho\abs{\,\cdot\,}_*}L^p(\beta)$, $B_2=\ee^{\rho\abs{\,\cdot\,}_*}L^{p'}(\beta)$, $B_3=\ee^{2\rho\abs{\,\cdot\,}_*}L^1(\beta) $,  $W'_j(t)=\frac{1}{m!}W^{(m+1)}_t$, and $W_j(t)=\sum_{k=0}^{m} \frac{1}{k!}W^{(k)}_t$ ($j=1,2,3$). Observe first that, by Young's inequality, convolution induces a continuous bilinear mapping $\ee^{k\rho\abs{\,\cdot\,}_*}L^r(\beta)\times \Delta_R^{1/r'}\ee^{-k\rho\abs{\,\cdot\,}_*}L^1(\beta)\to \ee^{k\rho\abs{\,\cdot\,}_*}L^r(\beta)$ for every $r\in [1,\infty]$ and for every $k\Meg 0$, so that  by Theorem~\ref{teo:7} and Lemma~\ref{lem:3} we see that $W_j(t),W'_j(t)\in \Lc(B_j)$ for $j=1,2,3$. In addition, using the continuity of the mapping $(0,+\infty)\ni t\mapsto \Delta_R^c \ee^{k\rho\abs{\,\cdot\,}_*}\Lc^{k}h_t\in L^1(\beta)$ for every $c\in \R$ and for  $k=1,2$ (which is essentially a consequence of Lemma~\ref{cor:3}), we see that the mapping $(0,+\infty)\ni t\mapsto W'_j(t) f\in B_j$ is continuous (and bounded), hence locally $\mi_1$-integrable for every $f\in B_j$ and for every $j=1,2,3$. Furthermore, by Lemma~\ref{lem:9}
	\[
	f=W_j(t) f+ \int_0^t W'_j(s) f\,\dd \mi_1(s)
	\]
	for every $f\in \Sc'(G)$ and for every $t\in (0,1]$, so that
	\[
	W_j(t')f= W_j(t)f+\int_{t'}^t W'_j(s) f\,\dd \mi_1(s)
	\]
	for every $0<t'<t\meg 1$. The same relation therefore holds, in $B_j$, for every  $f\in B_j$. In order to conclude, we have to show that $W_3(t)(W_1(t)f W_2(t) g)$ converges to $f g$ in $\Sc'(G)$ (actually, in $B_3$) and that $\Pi^{(t,1)}(f,g)$, with the notation of Lemma~\ref{lem:28b},  has a limit in $\Sc'(G)$ for $t\to 0^+$. The second assertion follows easily once one observes that the previous arguments show that there is a constant $C>0$ such that, for every $\phi\in \Sc(G)$,
	\[
	\begin{split}
		\abs{\langle W^{(m+1)}_t [(W^{(h)}_t f)(W^{(k)}_t g)]\vert \phi\rangle}&=t\abs{\langle W^{(m)}_t [(W^{(h)}_t f)(W^{(k)}_t g)]\vert \Lc^*\phi\rangle}\\
			&\meg t C \norm{\ee^{-\rho\abs{\,\cdot\,}_*}f}_{L^p(\beta)}\norm{\ee^{-\rho\abs{\,\cdot\,}_*}g}_{L^{p'}(\beta)} \norm{\ee^{2\rho \abs{\,\cdot\,}_*}\Lc^* \phi}_{L^\infty(\beta)}
	\end{split}
	\]
	for every $h,k\meg m$ and for every $t\in (0,1]$.
	The first assertion is also clear when $p\in (1,\infty)$, since in this case $W_1(t)f\to f$ in $B_1$ and $W_2(t) g\to g$ in $B_2$ for $t\to 0^+$: in fact, the assertion holds for $f,g\in \Sc(G)$ by Lemma~\ref{lem:9}, and then follows by density and equicontinuity in the general case. We may then reduce to the case $p=1$. Then, $W_1(t) f\to f$ in $B_1$, as one sees arguing as before. In addition, $W_2(t) g\to g$ locally in measure (since one has convergence, for example, in $\ee^{(\rho+\rho')\abs{\,\cdot\,}_*}L^1(\beta)$ by the previous remarks). In addition, the $W_2(t) g$ are uniformly bounded in $B_2$, by the previous remarks, for $t\in (0,1]$. It then follows that $  f W_2(t) g$ converges to $f g$ in $B_3$ by dominated convergence. Since clearly $(W_1(t)f -f) W_2(t) g\to 0$ in $B_3$, this proves that $W_1(t) f W_2(t)g$ converges to $fg$ in $B_3$. Since the $W_3(t)$ are equicontinuous on $B_3$ (for $t\in (0,1]$) and converge pointwise to the identity on $B_3$ for $t\to 0^+$, we may then conclude that $W_3(t)(W_1(t) f W_2(t)g)$ converges to $fg $ in $B_3$. The proof is complete.
\end{proof}

\begin{proof}[Proof of Theorem~\ref{teo:6}.]
	We only prove (ii). The proof of (i) is similar and slightly simpler.
	
	\textsc{Step I}  Set $m\coloneqq 2[\alpha/\grado]+2$; we apply Corollary~\ref{lem:28} with $m=m-1$. By the Gaussian estimates  (cf.~Theorem~\ref{teo:7}), there are $b,C_1>0$ such that $\abs{X\Lc^k h_t}\meg C_1\abs{X} t^{-k-\deg(X)/\grado} p_{b,t}$ for every $k=0,\dots, 2m$, for every $X\in U_m$, and for every $t\in (0,3]$. Let us first estimate $\Fs^{p,q}_{\alpha,m}(\Pi^{(t')}_f g)$. Observe that
	\[
	\begin{split}
	\abs{W^{(m)}_s W^{(h)}_t[(W^{(m)}_t f)(W^{(k)}_t g)]}&= s^m\abs{(t \Lc)^h \ee^{-(t/2) \Lc} \Lc^m  \ee^{-(s+t/2)\Lc} [(W^{(m)}_t f)(W^{(k)}_t g)] }  \\
		&\meg 2^h C_1 s^{m}  T_{b,t/2} \Lc^{m} \ee^{-(t/2+s)\Lc}[(W^{(m)}_t f)(W^{(k)}_t g)]\\
		&\meg 2^h C_1^2 s^{m}(s+t/2)^{-m}  T_{b, t/2} T_{b,t/2+s}   [(W^{(m)}_t f)(W^{(k)}_t g)]
	\end{split}
	\]
	for every $h,k=0,\dots, m-1$.
	In addition, by Lemma~\ref{lem:4} there are $C_2,b' >0$ such that
	\[
	\abs{W^{(m)}_s W^{(h)}_t[(W^{(m)}_t f)(W^{(k)}_t g)]}\meg C_2 s^{m}(s+t)^{-m}  T_{b', s} T_{b',t}  [(W^{(m)}_t f)(W^{(k)}_t g)]
	\]
	for every $h,k=0,\dots, m-1$,
	so that by Fubini's theorem and Lemmas~\ref{lem:2} and~\ref{lem:25} there is a constant $C_3>0$ such that
	\[
	\begin{split}
		\Fs^{p,q}_{\alpha,m}(\Pi^{(t')}_f g)&\meg \sum_{h,k=0}^{m-1} \norm*{\norm*{s^{-\alpha/\grado}  \int_0^1 \abs{(W^{(m)}_s W^{(h)}_t [(W^{(m)}_t f)(W^{(k)}_t g)])(x)}\,\frac{\dd t}{t}  }_{L^q_s(\mi_1)}  }_{L^p_x(\beta)}\\
			&\meg  m C_2\sum_{ k=0}^{m-1} \norm*{\norm*{s^{m-\alpha/\grado} \left(T_{b',s} \int_0^1  (s+t)^{-m}   T_{b',t}     [(W^{(m)}_t f)(W^{(k)}_t g)]\,\dd\mi_1(t) \right)(x) }_{L^q_s(\mi_1)}  }_{L^p_x(\beta)}\\
			&\meg C_3\sum_{ k=0}^{m-1} \norm*{\norm*{s^{m-\alpha/\grado}  \int_0^1  (s+t)^{-m} ( T_{b',t}    [(W^{(m)}_t f)(W^{(k)}_t g)])(x)\,\dd\mi_1(t)  }_{L^q_s(\mi_1)}  }_{L^p_x(\beta)}\\
			&\meg C_3^2\sum_{ k=0}^{m-1} \norm*{\norm*{t^{ -\alpha/\grado}    (T_{b',t}    [(W^{(m)}_t f)(W^{(k)}_t g)])(x)  }_{L^q_t(\mi_1)}  }_{L^p_x(\beta)}\\
			&\meg C_3^3\sum_{ k=0}^{m-1} \norm*{\norm*{t^{ -\alpha/\grado}     (W^{(m)}_t f)(x)(W^{(k)}_t g)(x)  }_{L^q_t(\mi_1)}  }_{L^p_x(\beta)}\\
			&\meg mC_1 C_3^3  \norm*{(T^*_{b,1} g)(x)\norm*{t^{ -\alpha/\grado}     (W^{(m)}_t f)(x)   }_{L^q_t(\mi_1)}  }_{L^p_x(\beta)}\\
			&\meg mC_1 C_3^3  \norm*{ T^*_{b,1} g }_{L^{p_2}(\beta)}  \Fs^{p_1,q}_{\alpha,m}(f) ,
	\end{split}
	\]
	which may be dealt with by means of Lemma~\ref{lem:5}.
	In a similar way, one may estimate $\Fs^{p,q}_{\alpha,m}(\Pi_g^{(t')} f)$ in terms of $\norm{f}_{L^{p_3}(\beta)} \Fs^{p_4,q}_{\alpha,m}(g)$ for every $t'\in (0,1)$.
	
	\textsc{Step II} We now estimate $\Pi(f,g)$. Observe first that there are two finite families $(X_j)_{j\in J}$ and $(Y_j)_{j\in J}$ of left-invariant differential operators on $G$ such that
	\[
	\Lc^{m}(\phi \psi)=\sum_{j \in J} (X_j \phi)(Y_j \psi)
	\]
	for every $\phi,\psi\in C^\infty(G)$, and
	such that $\deg(X_j)+\deg(Y_j)\meg m\grado $ for every $j\in J$.
	Then, observe that
	\[
	\begin{split}
	&\abs{W_s^{(m)} W^{(m)}_t[(W^{(h)}_t f)(W^{(k)}_t g)]}= (s t)^{m} \abs{ \Lc^{m}\ee^{-(s+t)\Lc} \Lc^m[(W^{(h)}_t f)(W^{(k)}_t g)]}\\
		&\qquad \meg C_1 s^{m}(s+t)^{-m} t^{m } \sum_{j\in J}T_{b,s+t} [ (X_j W^{(h)}_t f)(Y_j W^{(k)}_t g)].
	\end{split}
	\]
	Set 
	\[
	F^{(j,h,k)}_t(f,g)\coloneqq t^{m}   \abs{  (X_j W^{(h)}_t f)(Y_j W^{(k)}_t g)}
	\]
	for every $j\in J$, and $F_t^{(h,k)}(f,g)\coloneqq \sum_{j\in J} F_t^{(j,h,k)}(f,g)$.
	In addition, observe that by Lemma~\ref{lem:4bis} there are $C_4,b''>0$ such that
	\[
	\abs{W_s^{(m)} W^{(m)}_t[(W^{(h)}_t f)(W^{(k)}_t g)]}\meg C_4 s^{m}(s+t)^{-m}   \sum_{j\in J} T_{b'',s} T_{b'',t} F_t^{(j,h,k)}(f,g).
	\]
	 Then, by Lemmas~\ref{lem:2} and~\ref{lem:25}, we see that there is a constant $C_5>1$ such that
	\[
	\begin{split}
		\Fs^{p,q}_\alpha(\Pi(f,g))&\meg \sum_{h,k=0}^{m-1}    \norm*{\norm*{ s^{-\alpha/\grado}\int_0^1\abs{(W_s^{(m)} W^{(m)}_t[(W^{(h)}_t f)(W^{(k)}_t g)])(x)}\,\dd \mi_1(t)}_{L^q_s(\mi_1)}}_{L^p_x(\beta)}   \\
			&\meg C_4 \sum_{h,k=0}^{m-1}     \norm*{\norm*{ s^{m-\alpha/\grado} \left( T_{b'',s}\int_0^1  (s+t)^{-m}   T_{b'',t}F^{(h,k)}_t(f,g) \,\dd \mi_1(t)\right) (x)}_{L^q_s(\mi_1)}}_{L^p_x(\beta)} \\
			&\meg C_5\sum_{h,k=0}^{m-1}    \norm*{\norm*{ s^{m-\alpha/\grado}\int_0^1  (s+t)^{-m}    T_{b'',t}F^{(h,k)}_t(f,g)(x)\,\dd \mi_1(t)}_{L^q_s(\mi_1)}}_{L^p_x(\beta)} \\
			&\meg C_5^2\sum_{h,k=0}^{m-1}  \sum_{j\in J}  \norm*{\norm*{ t^{-\alpha/\grado}    F^{(j,h,k)}_t(f,g)(x) }_{L^q_t(\mi_1)}}_{L^p_x(\beta)} .
	\end{split}
	\]
	Now, take $j\in J$, and observe that either $\deg(X_j) \meg m\grado/2$ or $\deg(Y_j)\meg m\grado/2$. If $\deg(X_j)\meg m\grado/2$, then 
	\[
	\begin{split}
	F^{(j,h,k)}_t(f,g) 
		&\meg C_1 \abs{X_j}\abs{Y_j}( T_{b,t} f) (W^{(m-\deg(X_j)/\grado+k),*}_{t,1} g) 
	\end{split}
	\]
	so that	by Lemmas~\ref{lem:5} and~\cite[Proposition 6.9]{BCP} there is a constant $C_6>0$ such that
	\[
	\begin{split}
	&\norm*{\norm*{ t^{-\alpha/\grado}    F^{(j,h,k)}_t(f,g)(x) }_{L^q_t(\mi_1)}}_{L^p_x(\beta)}\\ &\qquad\meg C_1 \abs{X_j}\abs{Y_j} \norm*{ (T^*_{b,1} f)(x)  \norm*{ t^{-\alpha/\grado}   (W^{(m-\deg(X_j)/\grado+k),*}_{t,1} g)(x)  }_{L^q_t(\mi_1)}}_{L^p_x(\beta)}  \\
		&\qquad\meg C_1 \abs{X_j}\abs{Y_j} \norm*{ T^*_{b,1} f }_{L^{p_3}(\beta)} \norm*{\norm*{ t^{-\alpha/\grado}   (W^{(m-\deg(X_j)/\grado+k),*}_{t,1} g)(x)  }_{L^q_t(\mi_1)}}_{L^{p_4}_x(\beta)}  \\
		&\qquad\meg C_6\norm*{ f }_{L^{p_3}(\beta)}\norm{g}_{F^{p_4,q}_{\alpha}(\beta)} 
	\end{split} 
	\]
	since $ m-\deg(X_j)/\grado+k\Meg m/2>\alpha/\grado$.
	If, otherwise, $\deg(Y_j)\meg m \grado/2$, then arguing as before we see that
	\[
	\norm*{\norm*{ t^{-\alpha/\grado}    F^{(j,h,k)}_t(f,g)(x) }_{L^q_t(\mi_1)}}_{L^p_x(\beta)}\meg C_6\norm*{ f }_{F^{p_1,q}_\alpha(\beta)}\norm{g}_{L^{p_2}(\beta)} .
	\]
	 
	 \textsc{Step III}  We now complete the proof. Observe first that there is a constant $C_7>0$ such that
	 \[
	 \norm{f g}_{L^p(\beta)}\meg \norm{f}_{L^{p_1}(\beta)}\norm{g}_{L^{p_2}(\beta)}\meg C_7 \norm{f}_{F^{p_1,q}_\alpha(\beta)} \norm{g}_{L^{p_2}(\beta)}
	 \]
	 since  $\alpha>0$ (cf.~\cite[Proposition 6.9]{BCP}). By~\textsc{step I} and~\textsc{step II}, it only remains to estimate 
	 \[
	 \Fs^{p,q}_{\alpha,m}( W^{(h)}_1[ (W^{(k)}_1 f)(W^{(n)}_1 g) ] )
	 \]
	 for every $h,k,n=0,\dots, m-1$. Then, observe that
	 \[
	 \abs{W^{(m)}_t W^{(h)}_1[ (W^{(k)}_1 f)(W^{(n)}_1 g) ]}\meg C_1 t^m T_{b,1+t} [ (W^{(k)}_1 f)(W^{(n)}_1 g) ]\meg C_1 t^m 2^{Q_*/\grado}T_{b,2} [ (W^{(k)}_1 f)(W^{(n)}_1 g) ]
	 \]
	 for every $t\in (0,1]$, so that  there is a constant $C_8>1$ such that
	 \[
	 \begin{split}
	 \Fs^{p,q}_{\alpha,m}( W^{(h)}_1[ (W^{(k)}_1 f)(W^{(n)}_1 g) ] )&\meg C_1 2^{Q_*/\grado} \norm{t^{m-\alpha/\grado}}_{L^q_t(\mi_1)} \norm{T_{b,2} [ (W^{(k)}_1 f)(W^{(n)}_1 g) ]}_{L^p(\beta)}\\
	 	&\meg C_8   \norm{(W^{(k)}_1 f)(W^{(n)}_1 g) }_{L^p(\beta)}\\
	 	&\meg C_8 \norm{W^{(k)}_1 f}_{L^{p_1}(\beta)} \norm{W^{(n)}_1 g}_{L^{p_2}(\beta)}\\
	 	&\meg C_8^2 \norm{ f}_{L^{p_1}(\beta)} \norm{ g}_{L^{p_2}(\beta)}\\
	 	&\meg C_8^3\norm{ f}_{F^{p_1,q}_\alpha(\beta)} \norm{ g}_{L^{p_2}(\beta)},
	 \end{split}
	 \]
	 whence the result.	 
\end{proof}

\begin{cor}\label{cor:20}
	Take $p,q\in [1,\infty]$ and $\alpha> Q_*/p$. Then $B^{p,q}_\alpha(\beta_L)$ and, if $p<\infty$, $B^{p,1}_{Q_*/p}(\beta)$, are algebras under pointwise multiplication. If  $p\in (1,\infty)$, then the same holds for $F^{p,q}_\alpha(\beta_L)$.
\end{cor}

[Cf.~\cite[Corollary 7.2]{BPV}]

\begin{proof}
	This follows from Theorem~\ref{teo:6} and~Proposition~\ref{prop:2}.
\end{proof}

\begin{oss}\label{oss:4}
	With the notation of Theorem~\ref{teo:6}, if we take $p=p_1=p_3\in (1,\infty)$ and $p_2=\infty$, then the proof of Theorem~\ref{teo:6} shows that~\eqref{eq:7} still holds, provided that $\norm{g}_{F^{p_4,q}_\alpha(\beta)}$ is replaced by
	\[
	\norm{g}_{L^\infty(\beta)}+\norm*{\norm{ t^{-\alpha/\grado} (W^{(m)}_t f)(x) }_{L^q_t(\mi_1)}  }_{L^\infty_x(\beta)},
	\]
	where $m=2([\alpha/\grado]+1)$. Even though it is well known (cf., e.g.,~\cite[Remark 2.3.1/4]{Triebel}) that the above expression does \emph{not} lead to a good definition of a space $F^{\infty,q}_\alpha(\beta)$, it is still easy to show that it is smaller than $\norm{g}_{L^\infty(\beta)}+\Bs^{\infty,q}_{\alpha,m}(g)$. In particular, this proves that the bilinear mapping
	\[
	F^{p,q}_\alpha(\beta)\times B^{\infty,q}_\alpha(\beta)\ni (f,g)\mapsto f g \in F^{p,q}_\alpha(\beta)
	\]
	is (well defined and) continuous for every $p\in (1,\infty)$, for every $q\in [1,\infty]$, and for every $\alpha>0$. 
\end{oss}

\begin{oss}
	The results of this section may be suitably extended to the spaces $\mathring B^{p,q}_\alpha(\beta)$ and $\mathring F^{p,q}_\alpha(\beta)$. It suffices to observe that they clearly hold on $\Sc(G)$ and to conclude by density and continuity.
\end{oss}

 \section{Localization}\label{sec:4}

 The following result shows how we may perturb  the convolutors $(t\Lc)^m h_t$ without affecting the definition of Besov and Triebel--Lizorkin spaces, under certain circumstances. The advertized `localization' actually occurs when $\phi$ is compactly supported. Notice that the assumptions on $\phi$ are always (trivially) satisfied with any $\eta$ if $\phi$ is compactly supported and equals $1$ in a neighbourhood of $e$. This result will be particularly useful to characterize pointwise multipliers, since it provides norms which have `a local character,' in a certain sense.
 As an application, we present the localization properties of Triebel--Lizorkin spaces. 
 
 \begin{prop}\label{prop:34}
 	Let $\phi$ be a function of class $C^\infty$ on $G$, and assume that there are $c,\eta>0$  such that  $\abs{\phi(x)-1}\meg c\abs{x}_*^\eta \ee^{c\abs{x}_*}$ for every $x\in G$.
 	Take $p,q\in [1,\infty]$ and $\alpha\in [0,\eta)$. In addition, take $m\in \N$ such that $m>\alpha/\grado$, and take $\mi\in \Mc_\Samp$. Then, there is a constant $C>0$ such that the following hold:
 	\begin{enumerate}
 		\item[\textnormal{(i)}] if either $\alpha>0$ or $q=1$, then, for every $f\in \Sc'(G)$,
 		\[
 		\frac{1}{C}\norm{f}_{B^{p,q}_\alpha(\beta)}\meg \norm{f}_{L^p(\beta)}+ \norm{ t^{m-\alpha/\grado} [f*(\phi \Lc^m h_t)](x) }_{L^{p,q}_{x,t}(\beta,\mi)}\meg C \norm{f}_{B^{p,q}_\alpha(\beta)};
 		\]
 		
 		\item[\textnormal{(ii)}] if $p\in(1,\infty)$ and either $\alpha>0$ or $q\meg 2$, then, for every $f\in \Sc'(G)$,
 		\[
 		\frac{1}{C}\norm{f}_{F^{p,q}_\alpha(\beta)}\meg \norm{f}_{L^p(\beta)}+ \norm{ t^{m-\alpha/\grado} [f*(\phi \Lc^m h_t)](x) }_{L^{q,p}_{t,x}(\mi,\beta)}\meg C \norm{f}_{F^{p,q}_\alpha(\beta)}.
 		\]
 	\end{enumerate}
 \end{prop}
 
 Cf.~\cite{TriebelFS2} for more general versions of this result in the classical (and also in the Riemannian) case. 
 We observe explicitly that, in the classical case, suitably `localized' (quasi-)norms are available without any restrictions on $\alpha$. Our techniques, however, are more elementary and cannot achieve such more general statements.
 
 Notice that we could also prove analogous results with $f*(\phi  \Lc^m h_t)$ replaced by 
 \[
 \sup_{s\in [\eps t, t/\eps]} \sup_{\substack{\abs{X}\meg 1\\ \deg(X)\meg m\grado}} (s/t)^m\abs{f*(\phi X h_s)}
 \]
 for any fixed $\eps\in (0,1)$. In order to get the appropriate `localized' version of $W^{(m),*}_{t,\eps}$ (with the differential operators $Y$  applied to $f$ and not to $h_t$ by means of suitable right invariant differential operators), one should impose more restrictive assumptions on $\phi$.
 
 \begin{proof}
 	Observe first that, by~\cite[Propositions 6.9]{BCP} and~Proposition~\ref{prop:2}, $B^{p,q}_\alpha(\beta),F^{p,q}_\alpha(\beta)\subseteq L^p(\beta)$ under the stated assumptions. 	
 	Then, observe that by Theorem~\ref{teo:7} we may find $C_1,b>0$  such that
 	\[
 	\abs{(1-\phi)(x)(\Lc^m h_t)(x) }\meg C_1 \frac{\abs{x}_*^\eta}{t^{Q_*/\grado}} \ee^{c\abs{x}_* - b (\abs{x}_*^{\grado}/t )^{1/(\grado-1)}} 
 	\]
 	for every $x\in G$ and for every $t\in (0,\kappa]$, where $\kappa=\sup \supp(\mi)$. By Lemma~\ref{lem:3}, we may then find a constant $C_2>0$ such that
 	\[
 	\abs{(1-\phi)(x)(\Lc^m h_t)(x) }\meg C_2 \frac{\abs{x}_*^\eta}{t^{Q_*/\grado}} \ee^{  - (b/2) (\abs{x}_*^{\grado}/t )^{1/(\grado-1)}} 
 	\]
 	for every $x\in G$ and for every $t\in (0,\kappa]$, so that we may find a constant $C_3>0$ such that
 	\[
 	\abs{(1-\phi)(x)(\Lc^m h_t)(x) }\meg C_3 \frac{t^{\eta/\grado}}{t^{Q_*/\grado}} \ee^{  - (b/3) (\abs{x}_*^{\grado}/t )^{1/(\grado-1)}} =C_3 t^{\eta/\grado} p_{b/3,t}(x)
 	\]
 	for every $x\in G$ and for every $t\in (0,\kappa]$.
 	Consequently,
 	\[
 	\norm{W^{(m)}_t f- f*(\phi (t\Lc)^m h_t)  }_{L^p(\beta)} \meg C_3 t^{\eta/\grado}\norm{f}_{L^p(\beta)} \norm{p_{b/3,t}\Delta_R^{-1/p'}}_{L^1(\beta)}
 	\]
 	by Young's inequality. This gives (i) thanks to Lemma~\ref{lem:3}. Assertion (ii) is proved similarly, using the bound
 	\[
 	\abs{W^{(m)}_t f- f*(\phi (t\Lc)^m h_t)  }  \meg C_3 t^{\eta/\grado} T_{b/3,t} f
 	\]
 	and applying Lemma~\ref{lem:2}.
 \end{proof}
 
 As an application, we prove localization properties for Triebel--Lizorkin spaces (and also for the Besov spaces with $p=q$, which do not fall into the former category when $p=1$ or $p=\infty$ for technical reasons).
 
 In order to properly discuss the localization properties of the Triebel--Lizorkin spaces, we shall need to introduce `lattices', that is, uniformly spaced families of points which are also sufficiently `dense.'
 
 \begin{deff}
 	Take $\delta>0$ and $R\Meg 2$. We say that a family $(x_j)_{j\in J}$ of elements of $G$ is a $(\delta,R)$-lattice if the $B(x_j,\delta)$ are pairwise disjoint, while the $B(x_j,R\delta)$ cover $G$.
 \end{deff}

 It is clear that any maximal $2\delta$-separated family of elements of $G$ is a $(\delta,2)$-lattice, so that existence of $(\delta,R)$-lattices is ensured. It is nonetheless convenient, for technical reasons, to allow for `coarser' lattices (i.e., for the case $R>2$). 
 
 \begin{lem}\label{lem:50}
 	Take $\delta_0>0$ and $R_0\Meg 2$. Then, there is $N\in\N$ such that
 	\[
 	1\meg \sum_{j\in J} \chi_{B(x_j,R\delta)}\meg N
 	\]
 	for every $(\delta,R)$-lattice $(x_j)_{j\in J}$ on $G$ with $\delta\meg \delta_0$ and $2\meg R\meg R_0$. In particular, it is possible to find a partition $J_1,\dots, J_N$ of $J$ such that $d(x_j,x_{j'})\Meg  R\delta$ for every two distinct $j,j'\in J_h$, and for every $h=1,\dots, N$.
 \end{lem}
 
 Notice that every $(\delta,R)$-lattice is also a $(\delta',R')$-lattice for every $\delta'\meg \delta$ and for every $R'\Meg R\delta/\delta'$, so that the above assertion gives information also on the possible overlaps of balls centred at the $x_j$ and of arbitrarily large (but fixed) radii.
 
 \begin{proof}
 	Take $x\in G$, and let $J'_x$ be the set of $j\in J$ such that $x\in B(x_j,R\delta)$. Since the $B(x_j,\delta)$ are pairwise disjoint, we   infer that
 	\[
 	\card(J'_x) \beta_L(B(e,\delta))= \sum_{j\in J'_x} \beta_L(B(x_j,\delta))\meg \beta_L(B(x,(R+1)\delta))\meg\beta_L(B(e,(R_0+1)\delta)).
 	\]
 	Since $\beta_L$ is locally doubling, there is $N\in\N$ such that $\beta_L(B(e,(R_0+1)\delta))\meg N\beta_L(B(e,\delta))$ for every $\delta\in (0,\delta_0]$, so that $\card(J'_x)\meg N$. The first assertion follows.
 	Now, define by induction a sequence $(J_h)$ of subsets of $J$ as follows. Assuming that the $J_{h'}$, $h'<h$, have been constructed, define $J_h$ as a maximal subset of $J\setminus \big(\bigcup_{h'<h} J_{h'}\big)$ such that $d(x_j,x_{j'})\Meg   R\delta$ for every two distinct $j,j'\in J_h$.   Assume by contradiction that there is $j_{N+1}\in J_{N+1}$. By the maximality in the construction of the $J_1,\dots, J_N$, this means that there is $j_h\in J_h$, $h=1,\dots,N$, such that $d(x_{j_{N+1}},x_{j_h})<R\delta$. Consequently, $\sum_{j\in J} \chi_{B(x_j,R\delta)}(x_{j_{N+1}})\Meg N+1$, which is a contradiction. Then, $J_{N+1}=\emptyset$. By the definition of $J_{N+1}$, this means that $J\setminus \big(\bigcup_{h=1}^N J_{h}\big)=\emptyset$, whence the conclusion.
 \end{proof}

 \begin{prop}\label{prop:30}
 	Take $p\in (1,\infty)$, $q\in [1,\infty]$,  and $\alpha\in \R$. Take $\delta>0$, $R\Meg 2$, and a $(\delta,R)$-lattice $(x_j)_{j\in J}$ on $G$.  
 	Take two bounded families $(\phi_j)_{j\in J}$ and $(\psi_j)_{j\in J}$ of elements of $C^\infty_c(G)$. 
 	Then, the continuous linear mappings\footnote{Here, $\Dc'(G)$ denotes the space of distributions on $G$, that is, the continuous dual of the space $C^\infty_c(G)$ of smooth compactly supported functions.}
 	\[
 	\Ic_{(\phi_j)}\colon \Dc'(G)\ni u\mapsto (\phi_j(x_j^{-1}\,\cdot\,) u)\in \Dc'(G)^J
 	\]
 	and
 	\[
 	\Rc_{(\psi_j)}\colon \Dc'(G)^J \ni  (u_j) \mapsto \sum_{j\in J} \psi_j(x_j^{-1}\,\cdot\,) u_j \in \Dc'(G)
 	\]
 	induce continuous linear mappings\footnote{Here, $\ell^q_0$ denotes the closure in $\ell^q$ of the space of sequences with finite support, so that $\ell^q_0=\ell^q$ unless $q=\infty$.}
 	\[
 	\begin{aligned} 
 		F^{p,q}_\alpha(\beta)&\to \ell^p(J;F^{p,q}_\alpha(\beta)), & \mathring F^{p,q}_\alpha(\beta)&\to \ell^p(J;\mathring F^{p,q}_\alpha(\beta)),\\
 		B^{q,q}_\alpha(\beta)&\to \ell^q(J;B^{q,q}_\alpha(\beta)), & \mathring B^{q,q}_\alpha(\beta)&\to \ell^q_0(J;\mathring B^{q,q}_\alpha(\beta)),
 	\end{aligned}
 	\]
 	and
 	\[
 	\begin{aligned}
 		\ell^p(J;F^{p,q}_\alpha(\beta))&\to F^{p,q}_\alpha(\beta),& \ell^p(J;\mathring F^{p,q}_\alpha(\beta))&\to \mathring F^{p,q}_\alpha(\beta),\\
 		\ell^q(J;B^{q,q}_\alpha(\beta))&\to B^{q,q}_\alpha(\beta), &\ell^q_0(J;\mathring B^{q,q}_\alpha(\beta))&\to \mathring B^{q,q}_\alpha(\beta),
 	\end{aligned}
 	\] 
 	respectively.
 	If, in addition, $\sum_j (\phi_j\psi_j)(x_j^{-1}\,\cdot\,)=1$, then $\Rc_{(\psi_j)} \Ic_{(\phi_j)}=I$.
 \end{prop}
 
 This extends~\cite[Proposition 3.2]{BPV3} to the case of general filtrations and general $\alpha$. The proof is different and simpler.  
 
 For lack of a reference, we state and prove the following simple result before we pass to the proof. We state it a little more generally than we need for future reference.
 We state a definition first.
 
 \begin{deff}
 	Let $(A_j)_{j\in J}$ be a family of quasi-Banach space, take $p\in (0,\infty]$. We define $\ell^p((A_j))$ as the space of families $(x_j)\in \prod_j A_j$ such that $\left(\sum_{j} \norm{x_j}_{A_j}^p\right)^{1/p}<\infty$ if $p<\infty$, or $\sup_j \norm{x_j}_{A_j}<\infty$ if $p=\infty$, endowed with the corresponding quasi-norm. 
 	We define $\ell^p_0((A_j))$ as the closure of the space of families with finite support in $\ell^p((A_j))$, so that $\ell^p_0((A_j))=\ell^p((A_j))$ if $p<\infty$, while $\ell^\infty_0((A_j))$ is the space of families $(x_j)$ such that $\norm{x_j}\to 0$ along the filter of the complements of finite subsets (of $J$).
 \end{deff}

 \begin{lem}\label{lem:54}
 	Let $(A_j)_{j\in J}$ be a family of quasi-Banach space and take $p\in (0,\infty]$. Then, the bilinear mapping
 	\[
 	B\colon \ell^p((A_j))\times \ell^{p'}((A'_j))\ni ((x_j),(y_j))\mapsto \sum_j x_j y_j\in \C
 	\]
 	is well defined and continuous (with norm $1$), and induces an isometric isomorphism of $\ell^{p'}((A'_j))$ onto the dual of $\ell^p_0((A_j))$.
 \end{lem}
 
 \begin{proof}
 	Since clearly $\ell^p((A_j))\subseteq \ell^{\max(p,1)}((A_j))$, the first assertion follows from H\"older's inequality. To complete the proof, it will suffice to take $L\in (\ell^p_0((A_j)))'$ and to show that $L=B(\,\cdot\,,(y_j))$ for a unique $(y_j)\in \ell^{p'}((A'_j))$, and that $\norm{(y_j)}_{\ell^{p'}((A'_j))}\meg \norm{L}_{(\ell^p_0((A_j)))'}$. To this aim, we define $y_j\in A'_j$ so that $\langle y_j,x \rangle\coloneqq L( x e_j)$ for every $x\in A_j$ and for every $j\in J$, where $e_j(j')=\delta_{j,j'}$ for every $j,j'\in J$. Since $\abs{\langle y_j,x \rangle}\meg \norm{L}_{(\ell^p_0((A_j)))'}\norm{x}_{A_j}$ for every $x\in A_j$, we see that $y_j\in A_j'$ and that $\norm{y_j}_{A_j'}\meg \norm{L}_{(\ell^p_0((A_j)))'}$. In addition, clearly $L((x_j))=\sum_j  L(x_j e_j)=\sum_j \langle y_j, x_j\rangle=B((y_j),(x_j))$ for every $(x_j)$ with finite support, so that we only need to show that $\norm{(y_j)}_{\ell^{p'}((A_j'))}\meg \norm{L}_{(\ell^p_0((A_j)))'}$.
 	If $p\meg 1$, this follows from the previous estimates. If $p>1$, then take $\eta\in (0,1)$ and $x_{j,\eta}\in A_j$ such that $\norm{x_{j,\eta}}_{A_j}\meg 1$ and $\langle y_j, x_{j,\eta}\rangle\Meg \eta\norm{y_j}_{A'_j}$. Take a finite subset $J'$ of $J$, and define $x'_{j,\eta,J'}=\chi_{J'}(j) \norm{y_j}_{A'_j}^{p'/p} x_{j,\eta}$ for every $j\in J$, so that $\norm{(x'_{j,\eta,J'})}_{\ell^{p}((A_j))}\meg \norm{(\chi_{J'}(j)y_{j})}_{\ell^{p'}((A'_j))}^{p'/p}$ and $B((y_j),(x'_{j,\eta,J'}))\Meg \eta \sum_{j\in J'} \norm{y_j}_{A'_j}^{p'/p+1}=\eta\sum_{ j\in J' }\norm{y_j}_{A'_j}^{p'}$. Hence, $\norm{L}_{(\ell^p_0((A_j)))'}\Meg \eta\big( \sum_{ j\in J' }\norm{y_j}_{A'_j}^{p'}\big) ^{1/p'}$, whence the conclusion by the arbitrariness of $J'$ and $\eta$.
 \end{proof}

 \begin{proof}[Proof of Proposition~\ref{prop:30}.]
 	The last assertion is obvious. We prove  the remaining assertions only for Triebel--Lizorkin spaces. 
 	In order to simplify the notation, we set $\widetilde \phi_j\coloneqq \phi_j(x_j^{-1}\,\cdot\,)$ and $\widetilde \psi_j\coloneqq \psi_j(x_j^{-1}\,\cdot\,)$ for every $j\in \N$. 
 	
 	\textsc{Step I} Assume first that $\alpha>0$. Fix $m\coloneqq [\alpha/\grado]+1$ and take $\eta\in C^\infty_c(G)$ such that $\chi_{B(e,1/4)}\meg \eta \meg \chi_{B(e,1/2)}$, so that Proposition~\ref{prop:34} shows that there is a constant $C_1>0$ such that
 	\[
 	\frac{1}{C_1}\norm{f}_{F^{p,q}_\alpha(\beta)}\meg \norm{f}_{L^p(\beta)}+ \norm*{ \norm{ t^{-\alpha/\grado} (f*\eta_t )(x)  }_{L^q_t(\mi_1)}  }_{L^p_x(\beta)}\meg C_1 \norm{f}_{F^{p,q}_\alpha(\beta)}
 	\]
 	for every $f\in \Sc'(G)$, where $\eta_t=\eta (t \Lc)^m h_t$ for every $t\in (0,1]$. In addition, take $R' $  so large that the $\phi_j$ and the $\psi_j$ are supported in $B(e,R' )$, and observe that by Lemma~\ref{lem:50} there is a partition $J_1,\dots, J_N$ of $J$ such that $d(x_j,x_{j'})\Meg 2(R'+1)$ for every $j,j'\in J_h$, $j\neq j'$, and for every $h=1,\dots, N$. 
 	Then, observe that
 	\[
 	\begin{split}
 		\norm*{\norm{\widetilde \phi_j f}_{L^p(\beta)}}_{\ell^p_j(J)}&=\norm*{\norm{\widetilde \phi_j f}_{\ell^p_j(J)}}_{L^p(\beta)}\\
 			&\meg \sup_{j} \norm{\phi_j}_{L^\infty(\beta)} N^{1/p}\norm{f}_{L^p(\beta)}
 	\end{split}
 	\]
 	and that
 	\[
 	\begin{split}
 		\norm*{\norm*{ \norm{ t^{-\alpha/\grado} ((\widetilde \phi_j f)*\eta_t )(x)  }_{L^q_t(\mi_1)}  }_{L^p_x(\beta)}}_{\ell^p_j(J)}&\meg \sum_{h=1}^N\norm*{\norm*{ \norm{ t^{-\alpha/\grado} ((\widetilde \phi_j f)*\eta_t )(x)  }_{L^q_t(\mi_1)}  }_{L^p_x(\beta)}}_{\ell^p_j(J_h)} \\
 		&= \sum_{h=1}^N \norm*{\norm*{ \norm{ t^{-\alpha/\grado} ((\widetilde \phi_j f)*\eta_t )(x)  }_{L^q_t(\mi_1)}}_{\ell^p_j(J_h)}  }_{L^p_x(\beta)}\\
 		&=\sum_{h=1}^N\norm*{\norm*{ \norm{ t^{-\alpha/\grado} ((\widetilde \phi_j f)*\eta_t )(x)  }_{L^q_t(\mi_1)}}_{\ell^q_j(J_h)}  }_{L^p_x(\beta)}\\
 		&=\sum_{h=1}^N\norm*{\norm*{ \norm{ t^{-\alpha/\grado} ((\widetilde \phi_j f)*\eta_t )(x)  }_{\ell^q_j(J_h)} }_{L^q_t(\mi_1)} }_{L^p_x(\beta)}\\
 		&=\sum_{h=1}^N\norm*{\norm*{  t^{-\alpha/\grado} \sum_{j\in J_h}((\widetilde \phi_j f)*\eta_t )(x)  }_{L^q_t(\mi_1)} }_{L^p_x(\beta)}\\
 		&\meg C_1\sum_{h=1}^N\norm*{  \sum_{j\in J_h}\widetilde \phi_j  f }_{F^{p,q}_\alpha (\beta)}
 	\end{split}
 	\]
 	where the second and the fourth equality follow from the fact that, for every $x\in G$ and for every $h=1,\dots,N$, there is at most one $j\in J_h$ such that $ ((\widetilde \phi_j f)*\eta_t )(x)\neq0 $ for some $t\in (0,1]$. The continuity of $\Ic_{(\phi_j)}\colon F^{p,q}_\alpha(\beta)\to \ell^p(J; F^{p,q}_\alpha(\beta))$ then follows from Remark~\ref{oss:4} and Proposition~\ref{prop:7}, since $\sum_{j\in J_h}\widetilde \phi_j $ clearly belongs to $W^{\infty,\infty}(\beta)$ for every $h=1,\dots, N$. 
 	Since  $\Ic_{(\phi_j)}(C^\infty_c(G))\subseteq C^\infty_c(G)^{(J)}$, it is also clear that $\Ic_{(\phi_j)}$ maps $\mathring F^{p,q}_\alpha(\beta)$ into $ \ell^p(J; \mathring F^{p,q}_\alpha(\beta))$ continuously.
 	
 	Next, take $(f_j)\in  \ell^p(J;   F^{p,q}_\alpha(\beta))$, set $f\coloneqq \sum_j \widetilde \psi_j f_j$, and observe that
 	\[
 	\begin{split}
 		\norm{f}_{L^p(\beta)}&\meg N^{1/p'}\norm*{\norm{(\widetilde \psi_j f_j)(x) }_{\ell^p_j(J)}   }_{L^p_x(\beta)}\\
 			&\meg N^{1/p'} \sup_j \norm{\psi_j}_{L^\infty(\beta)} \norm*{\norm{  f_j(x) }_{\ell^p_j(J)}   }_{L^p_x(\beta)}\\
 			&=N^{1/p'} \sup_j \norm{\psi_j}_{L^\infty(\beta)} \norm*{\norm{  f_j }_{L^p (\beta)} }_{\ell^p_j(J)}  
 	\end{split}
 	\]
 	and that 
 	\[
 	\begin{split}
 		\norm*{ \norm{ t^{-\alpha/\grado}  (f*\eta_t )(x)  }_{L^q_t(\mi_1)}  }_{L^p_x(\beta)}&\meg \norm*{\sum_j \norm{ t^{-\alpha/\grado}  [(\widetilde \psi_j f_j)*\eta_t](x)  }_{L^q_t(\mi_1)}  }_{L^p_x(\beta)}\\
 		&\meg N^{1/p'}\norm*{\norm*{ \norm{ t^{-\alpha/\grado}  [(\widetilde \psi_j f_j)*\eta_t ](x)  }_{L^q_t(\mi_1)}  }_{\ell^p_j(J)}}_{L^p_x(\beta)}\\
 		&= N^{1/p'}\norm*{\norm*{ \norm{ t^{-\alpha/\grado}  [(\widetilde \psi_j f_j)*\eta_t ](x)  }_{L^q_t(\mi_1)}  }_{L^p_x(\beta)}}_{\ell^p_j(J)}\\
 		&\meg C_1 N^{1/p'}\norm*{\norm{  \widetilde \psi_j f_j  }_{F^{p,q}_\alpha(\beta)}}_{\ell^p_j(J)}.
 	\end{split}
 	\]
 	The continuity of $\Rc_{(\psi_j)}\colon \ell^p(J; F^{p,q}_\alpha(\beta))\to  F^{p,q}_\alpha(\beta) $ then follows again from Remark~\ref{oss:4} and Proposition~\ref{prop:7}, since the $\widetilde \psi_j$ are uniformly bounded in $W^{\infty,\infty}(\beta)$. Since  $\Rc_{(\psi_j)}(C^\infty_c(G)^{(J)})\subseteq C^\infty_c(G)$, it is also clear that $\Rc_{(\psi_j)}$ maps $\ell^p(J; \mathring F^{p,q}_\alpha(\beta))$ into $  \mathring F^{p,q}_\alpha(\beta) $ continuously.

 	\textsc{Step II} Take $\alpha<0$. Observe that $F^{p,q}_\alpha(\beta)=\mathring F^{p',q'}_{-\alpha}(\beta)'$ and that $\ell^p(F^{p,q}_\alpha(\beta))=\ell^{p'}(\mathring F^{p',q'}_{-\alpha}(\beta))'$ canonically, thanks to~\cite[Theorem 9.4]{BCP} and Lemma~\ref{lem:54}. Since $\trasp \Ic_{(\psi_j)}=\Rc_{(\psi_j)}$ and $\trasp \Rc_{(\phi_j)}=\Ic_{(\phi_j)}$ (with respect to the aforementioned canonical isomorphisms), by means of (and arguing as in)~\textsc{step I} we see  that the assertion holds for $\alpha<0$.
 	
 	\textsc{Step III} The continuity of $\Ic_{(\phi_j)}\colon \mathring F^{p,q}_0(\beta)\to \ell^p(J; \mathring F^{p,q}_0(\beta))$ and $\Rc_{(\psi_j)}\colon \ell^p(J; \mathring F^{p,q}_0(\beta))\to \mathring F^{p,q}_0(\beta)$ follows from~\textsc{steps I} and~\textsc{II} by complex interpolation, thanks to~\cite[Remark 10.2]{BCP}. By duality, as in~\textsc{step II}, one may then show the continuity of $\Ic_{(\phi_j)}\colon   F^{p,q}_0(\beta)\to \ell^p(J;   F^{p,q}_0(\beta))$ and $\Rc_{(\psi_j)}\colon \ell^p(J;   F^{p,q}_0(\beta))\to   F^{p,q}_0(\beta)$.	
 \end{proof}

 \section{Pointwise Multipliers of Triebel--Lizorkin Spaces}\label{sec:5}
 
 In this section we characterize the space of pointwise multipliers of the Triebel--Lizorkin space $F^{p,q}_\alpha(\beta)$ for $\alpha>Q_*/p$. Since the spaces $F^{p,q}_\alpha(\beta)$ `localize' well (cf.~Proposition~\ref{prop:30}), the corresponding space of pointwise multipliers turns out to be the space of tempered distributions which stay locally in $F^{p,q}_\alpha(\beta_L)$ with uniformly bounded norms.
 We begin with a technical lemma which is necessary to define the space of such distributions.
 
 \begin{lem}\label{lem:51}
 	Take $p\in (1,\infty)$, $q\in [1,\infty]$, and $\alpha>0$. Take two   $\eta_1,\eta_2\in C^\infty_c(G)$ with $\eta_2\neq 0$.  
 	Then, there is a constant $C>0$ such that 
 	\[
 	\norm{\eta_1(x\,\cdot\,) f}_{F^{p,q}_\alpha(\beta)}\meg C  \norm{f}_{F^{p,q}_\alpha(\beta)}
 	\]
 	for every $x\in G$ and for every $f\in \Sc'(G)$, and such that
 	\[
 	\sup_{x\in G}  \norm{\eta_1(x\,\cdot\,) f}_{F^{p,q}_\alpha(\beta)}\meg C\sup_{x\in G}  \norm{\eta_2(x\,\cdot\,) f}_{F^{p,q}_\alpha(\beta)} 
 	\]
 	for every $f\in \Sc'(G)$.
 \end{lem}

 \begin{proof} 
 	Notice first that we may reduce to the case $\beta=\beta_L$, thanks to Proposition~\ref{prop:5}.	
 	
 	\textsc{Step I}	The first assertion follows easily from Remark~\ref{oss:4}, since the $\eta_1(x\,\cdot\,)$ are uniformly bounded in $B^{\infty,q}_\alpha(\beta)$.
 	
 	\textsc{Step II}  Observe that there are $c>0$, $x_0\in G$, and $r>0$ such that $\abs{\eta_2(x)}\Meg c$ for every $x\in B(x_0,r)$. In addition, there is a finite subset $J$ of $G$ such that $\supp{\eta_1}\subseteq J B(x_0,r)$. Consequently, there is a family $(\psi_j)_{j\in J}$ of elements of $C^\infty_c(G)$ such that $\eta_1=\sum_{j\in J} \psi_j \eta_2(j^{-1}\,\cdot\,)$. Then, by~\textsc{step I} we see that there is a constant $C_2>0$ such that
 	\[
 	\norm{\eta_1(x\,\cdot\,) f}_{F^{p,q}_\alpha(\beta)}\meg C_2\sum_{j\in J} \norm{\eta_2(j^{-1}x\,\cdot\,) f}_{F^{p,q}_\alpha(\beta)}\meg C_2 \card(J) \sup_{x'} \norm{\eta_2(x'\,\cdot\,) f}_{F^{p,q}_\alpha(\beta)}
 	\]
 	for every $x\in G$.  
 	The assertion follows. 
 \end{proof}

 \begin{deff}
 	Take $p\in (1,\infty)$, $q\in [1,\infty]$, and $\alpha>0$. We define $F^{p,q}_{\alpha,\unif}(\beta)$ as the space of $f\in \Sc'(G)$ such that $\sup_{x\in G} \norm{\eta(x\,\cdot\,) f}_{F^{p,q}_\alpha(\beta)}$ is finite 
 	for some non-zero $\eta\in C^\infty_c(G)$, endowed with the corresponding topology.
 \end{deff}

 \begin{deff}
 	Take $p\in (1,\infty)$, $q\in [1,\infty]$, and $\alpha\in \R$. We define $\Mc( F^{p,q}_\alpha(\beta))$ as the space of $u\in \Sc'(G)$ such that the mapping  $\Sc(G)\ni f\mapsto u f\in \Sc'(G)$ induces a continuous linear mapping $\mathring F^{p,q}_\alpha(\beta)\to F^{p,q}_\alpha(\beta)$, endowed with the corresponding topology.
 \end{deff}
 
 \begin{teo}\label{teo:14}
 	Take $p\in (1,\infty)$, $q\in [1,\infty]$, and  $\alpha>Q_*/p$. Then, $\Mc(F^{p,q}_\alpha(\beta))= F^{p,q}_{\alpha,\unif}(\beta_L)$. In addition, the canonical   bilinear mapping\footnote{Notice that $F^{p,q}_\alpha(\beta)\subseteq L^p(\beta)$, and that $F^{p,q}_{\alpha,\unif}(\beta_L)\subseteq L^\infty (\beta)$ by Proposition~\ref{prop:2}, so that $u f$ may be defined unambiguously as an element of $L^p (\beta)$.}
 	\[
 	F^{p,q}_\alpha(\beta)\times F^{p,q}_{\alpha,\unif}(\beta_L)\ni (f,u)\mapsto u f\in F^{p,q}_\alpha(\beta)
 	\] 
 	is continuous.
 \end{teo}

 This result extends~\cite[Theorem 1]{BPV3} to the case of general filtrations.
 
 \begin{proof}
 	By Proposition~\ref{prop:5}, we may reduce to the case $\beta=\beta_L$.
 	Take a $(1,2)$-lattice $(x_j)_{j\in J}$ on $G$ and a bounded family $(\phi_j)_{j\in J}$ of elements of $C^\infty_c(G)$ such that $\sum_j \widetilde \phi_j=1$, where $\widetilde \phi_j=\phi_j(x_j^{-1}\,\cdot\,)$ for every $j\in J$. In addition, take $\psi\in C^\infty_c(G)$ such that $\psi=1$ on $B(e,R+1)$, where $R>0$ is chosen so that  the $\phi_j$ are supported in $B(e,R)$.
 	Observe that, by Proposition~\ref{prop:30} and Corollary~\ref{cor:20}, there is a constant $C_1>1$ such that
 	\[
 	\begin{split}
 	\norm{f g}_{F^{p,q}_\alpha(\beta)}&\meg C_1 \norm*{\norm{ \widetilde \phi_j f g}_{F^{p,q}_\alpha(\beta)}}_{\ell^p_j(J)}\\
 		&\meg C_1 \norm*{\norm{ \widetilde \phi_j f \psi(x_j^{-1}\,\cdot\,) g}_{F^{p,q}_\alpha(\beta)}}_{\ell^p_j(J)}\\
 		&\meg C_1^2\norm*{\norm{ \widetilde \phi_j f}_{F^{p,q}_\alpha(\beta)} \norm{\psi(x_j^{-1}\,\cdot\,)g}_{F^{p,q}_\alpha(\beta)}}_{\ell^p_j(J)}\\
 		&\meg C_1^2 \norm*{\norm{ \widetilde \phi_j f}_{F^{p,q}_\alpha(\beta)}}_{\ell^p_j(J)}\sup_j \norm{\psi(x_j^{-1}\,\cdot\,)g}_{F^{p,q}_\alpha(\beta)}\\
 		&\meg C_1^3 \norm{f}_{F^{p,q}_\alpha(\beta)} \norm{g}_{F^{p,q}_{\alpha,\unif}(\beta)},
 	\end{split}
 	\]
 	whence the last assertion and the continuous inclusion $F^{p,q}_{\alpha,\unif}(\beta)\subseteq \Mc(F^{p,q}_\alpha(\beta))$. For what concerns the converse inclusion, take   $u\in \Mc(F^{p,q}_\alpha(\beta))$ and a non-zero $\phi\in C^\infty_c(G)$. Then,   for every $x\in G$,
 	\[
 	\begin{split}
 		\norm{\phi(x\,\cdot\,) u  }_{F^{p,q}_\alpha(\beta)}& \meg  \norm{u}_{\Mc(F^{p,q}_\alpha(\beta))} \norm{\phi(x\,\cdot\,)  }_{F^{p,q}_\alpha(\beta)}\\
 		&=\norm{u}_{\Mc(F^{p,q}_\alpha(\beta))}   \norm{ \phi   }_{F^{p,q}_\alpha(\beta)},
 	\end{split}
 	\]
 	whence the conclusion.
 \end{proof}

\section{Pointwise Multipliers of Besov Spaces}\label{sec:6}

The space of pointwise multipliers of the Besov space $B^{p,q}_\alpha(\beta)$ looks like its counterpart for the Triebel--Lizorkin space $F^{p,q}_\alpha(\beta)$ only when $p\meg q$. When $p=\infty$, it is particularly simple to describe, since it coincides with the space $B^{\infty,q}_\alpha(\beta)$ itself (for $\alpha>0$). For $p<q$, we may only provide a slightly simplified description, showing that one may `get away' with testing the continuity of pointwise multiplication on a relatively simple subspace.

\begin{deff}
	Take $p,q\in [1,\infty]$ and $\alpha\in \R$. We define $\Mc(B^{p,q}_\alpha(\beta))$ as the space of $u\in \Sc'(G)$ such that the mapping $\Sc(G)\ni f \mapsto u f\in \Sc'(G)$ induces a continuous linear mapping $\mathring B^{p,q}_\alpha(\beta)\to B^{p,q}_\alpha(\beta)$, endowed with the corresponding topology.
\end{deff}

\begin{prop}\label{prop:21}
	Take $q\in [1,\infty]$ and $\alpha> 0$. Then, $\Mc(B^{\infty,q}_\alpha(\beta))=B^{\infty,q}_\alpha(\beta)$.
\end{prop}

This result extends~\cite[Theorem 4.14]{BPV3} to the case of general filtrations.

\begin{proof}
	By Proposition~\ref{prop:5} we may reduce to the case $\beta=\beta_L$. Then, Corollary~\ref{cor:20} shows that $B^{\infty,q}_\alpha(\beta)\subseteq \Mc(B^{\infty,q}_\alpha(\beta))$ continuously. The converse inclusion essentially follows from the fact that the function identically equal to $1$ belongs to $B^{\infty,q}_\alpha(\beta)$. Since, however, that function need not belong to $\mathring B^{\infty,q}_\alpha(\beta)$, we need to proceed a little more carefully. Take a positive $\phi \in C^\infty_c(G)$ with $\supp \phi\subseteq B(e,1)$ and $\int_G \phi\,\dd \beta=1$, and define $\tau_k\coloneqq \chi_{B(e,k+1)}*\phi$, so that $\chi_{B(e,k)}\meg \tau_k \meg \chi_{B(e,k+2)}$ for every $k\in\N$, and the $\tau_k$ are uniformly bounded in $W^{\infty,\infty}(\beta)$, hence in $B^{\infty,q}_\alpha(\beta)$ by Proposition~\ref{prop:7}. Therefore, if $f\in \Mc(B^{\infty,q}_\alpha(\beta))$, then the $\tau_k f$ must be bounded in $B^{\infty,q}_\alpha(\beta)$. Letting $k\to \infty$, we then get $f\in B^{\infty,q}_\alpha(\beta)$. 
\end{proof}

\begin{lem}\label{lem:51b}
	Take  $p,q\in [1,\infty]$ and $\alpha> 0$. Take two   $\eta_1,\eta_2\in C^\infty_c(G)$ with $\eta_2\neq 0$. Then, there is a constant $C>0$ such that 
	\[
	\norm{\eta_1(x\,\cdot\,) f}_{B^{p,q}_\alpha(\beta)}\meg C  \norm{f}_{B^{p,q}_\alpha(\beta)}
	\]
	for every $x\in G$ and for every $f\in \Sc'(G)$, and such that
	\[
	\sup_{x\in G}\norm{\eta_1(x\,\cdot\,) f}_{B^{p,q}_\alpha(\beta)}\meg C\sup_{x\in G} \norm{\eta_2(x\,\cdot\,) f}_{B^{p,q}_\alpha(\beta)}
	\]
	for every $f\in \Sc'(G)$.
\end{lem}

The proof is similar to that of Lemma~\ref{lem:51} and is omitted.

\begin{deff}
	Take  $p,q\in [1,\infty]$ and $\alpha> 0$.   We define $B^{p,q}_{\alpha,\unif}(\beta)$ as the space of $f\in \Sc'(G)$ such that $\sup_{x\in G} \norm{\eta(x\,\cdot\,) f}_{B^{p,q}_\alpha(\beta)}$ is finite 
	for some non-zero $\eta\in C^\infty_c(G)$, endowed with the corresponding topology.
\end{deff}

\begin{lem}\label{lem:53}
	Take $p,q\in [1,\infty]$ and $\alpha> 0$.  Then, $\Mc(B^{p,q}_\alpha(\beta))\subseteq B^{p,q}_{\alpha,\unif}(\beta_L)$ continuously. If, in addition, $\alpha\Meg Q_*/p$ and either $q=1$ or $\alpha>Q_*/p$, then $B^{p,q}_{\alpha,\unif}(\beta_L)\subseteq L^\infty(\beta)$ continuously.
\end{lem}

\begin{proof}
	By Proposition~\ref{prop:5} we may reduce to the case $\beta=\beta_L$.
	Take $u\in \Mc(B^{p,q}_\alpha(\beta))$, and a non-zero $\phi\in C^\infty_c(G)$. Then,  for every $x\in G$,
	\[
	\norm{\phi(x\,\cdot\,) u  }_{B^{p,q}_\alpha(\beta)}\meg \norm{u}_{\Mc(B^{p,q}_\alpha(\beta))} \norm{\phi(x\,\cdot\,)  }_{B^{p,q}_\alpha(\beta)}=\norm{u}_{\Mc(B^{p,q}_\alpha(\beta))}   \norm{ \phi   }_{B^{p,q}_\alpha(\beta )},
	\]
	whence the continuous inclusion $\Mc(B^{p,q}_\alpha(\beta))\subseteq B^{p,q}_{\alpha,\unif}(\beta)$.
	
	Next, assume that $\alpha\Meg Q_*/p$ and either $q=1$ or $\alpha>Q_*/p$. Take $f\in B^{p,q}_{\alpha,\unif}(\beta)$ and assume that $\phi=1$ on $B(e,1)$. Then, by Proposition~\ref{prop:2}, there is a constant $C>0$ such that
	\[
	\norm{f}_{L^\infty(\beta)}\meg \sup_{x\in G} \norm{\phi(x\,\cdot\,) f}_{L^\infty(\beta)}\meg C \sup_{x\in G} \norm{\phi(x\,\cdot\,) f}_{B^{p,q}_{\alpha}(\beta)}.
	\]
	The assertion follows.
\end{proof}

\begin{teo}\label{teo:20}
	Take $p,q\in [1,\infty]$ and $\alpha \Meg Q_*/p$ such that $p\meg q$ and either $p=q=1$ or $\alpha>Q_*/p$. Then, $\Mc(B^{p,q}_\alpha(\beta))=B^{p,q}_{\alpha,\unif}(\beta_L)$. 
	In addition, the canonical   bilinear mapping 
	\[
	B^{p,q}_\alpha(\beta)\times B^{p,q}_{\alpha,\unif}(\beta_L)\ni (f,u)\mapsto u f\in B^{p,q}_\alpha(\beta)
	\] 
	is continuous.
\end{teo}

This result extends~\cite[Theorems 4.1 and 4.2]{BPV3} to the case of general filtrations and general $\alpha>Q_*/p$. The proof is different.

\begin{proof}
	By Proposition~\ref{prop:5} we may reduce to the case $\beta=\beta_L$.
	The proof when $p=q$ is similar to the proof of Theorem~\ref{teo:14} and is omitted. We may therefore assume that $p<q$. In fact, we shall only assume that $\alpha>Q_*/p$ (and $p\meg q$). Since the inclusion $\Mc(B^{p,q}_\alpha(\beta))\subseteq B^{p,q}_{\alpha,\unif}(\beta)$ follows from Lemma~\ref{lem:53}, we may reduce to proving the continuity of $B^{p,q}_\alpha(\beta)\times B^{p,q}_{\alpha,\unif}(\beta) \ni (f,g)\mapsto fg\in B^{p,q}_\alpha(\beta)$. 
	We proceed as in the proof of Theorem~\ref{teo:6}.
	
	\textsc{Step I}  Set $m\coloneqq 2[\alpha/\grado]+2$. By the Gaussian estimates  (cf.~Theorem~\ref{teo:7}), there are $b,C_1>0$ such that $\abs{X\Lc^k h_t}\meg C_1\abs{X} t^{-k-\deg(X)/\grado} p_{b,t}$ for every $k=0,\dots, 2m$, for every $X\in U_m$, and for every $t\in (0,3]$. 
	Let us first estimate $\Bs^{p,q}_{\alpha,m}(\Pi_f^{(t')} g)$. Observe that
	\[
	\begin{split}
		\abs{W^{(m)}_s W^{(h)}_t[(W^{(m)}_t f)(W^{(k)}_t g)]}&= s^m\abs{(t \Lc)^h \ee^{-(t/2) \Lc} \Lc^m  \ee^{-(s+t/2)\Lc} [(W^{(m)}_t f)(W^{(k)}_t g)] }  \\
		&\meg 2^h C_1 s^{m}  T_{b,t/2} \Lc^{m} \ee^{-(t/2+s)\Lc}[(W^{(m)}_t f)(W^{(k)}_t g)]\\
		&\meg 2^h C_1^2 s^{m}(s+t/2)^{-m}  T_{b, t/2} T_{b,t/2+s}   [(W^{(m)}_t f)(W^{(k)}_t g)].
	\end{split}
	\]
	Then, by Young's inequality and Lemmas~\ref{lem:3} and~\ref{lem:25} there is a constant $C_2>0$ such that
	\[
	\begin{split}
		\Bs^{p,q}_{\alpha,m}(\Pi_f^{(t')} g)&\meg \sum_{h,k=0}^{m-1} \norm*{s^{-\alpha/\grado}\norm*{  \int_0^1 \abs{(W^{(m)}_s W^{(h)}_t [(W^{(m)}_t f)(W^{(k)}_t g)])(x)}\,\frac{\dd t}{t} }_{L^p_x(\beta) }}_{L^q_s(\mi_1)}  \\
		&\meg  C_2\sum_{ k=0}^{m-1} \norm*{s^{m-\alpha/\grado}  \int_0^1  (s+t)^{-m}  \norm*{(W^{(m)}_t f)(W^{(k)}_t g)}_{L^p(\beta)}\,\dd\mi_1(t) }_{L^q_s(\mi_1)}  \\
		&\meg C_2^2\sum_{ k=0}^{m-1} \norm*{t^{ -\alpha/\grado}   \norm*{(W^{(m)}_t f)(W^{(k)}_t g)}_{L^p(\beta)} }_{L^q_t(\mi_1)} .
	\end{split}
	\]
	Now, take $\eta\in C^\infty_c(G)$ so that $\chi_{B(e,1/2)}\meg \eta \meg \chi_{B(e,1 )}$, and define $\widetilde W^{(k)}_t u\coloneqq u*(\eta (t\Lc)^k h_t)$ for every $t>0$, for every $k\in\N$, and for every $u\in \Sc'(G)$. Observe that, by Lemmas~\ref{lem:3} and~\ref{lem:53}, and Proposition~\ref{prop:2}, there is a constant $C_3>0$ such that
	\[
	\begin{split}
		&\norm{(W^{(m)}_t f-\widetilde W^{(m)}_t f)(W^{(k)}_t g)}_{L^p(\beta)}\meg \norm{f*((1-\eta)(t \Lc)^m h_t)}_{L^p(\beta)} \norm{W^{(k)}_t g}_{L^\infty(\beta)}\\
		&\qquad \qquad\qquad\qquad\qquad\meg C_1^2 \ee^{-(b/2) (2^{\grado} t)^{-1/(\grado-1)}}\norm{ p_{b/2,t} \Delta_R^{-1/p'}}_{L^1(\beta)}\norm{f}_{L^p(\beta)} \norm{g}_{L^\infty(\beta)} \norm{p_{b,t}\Delta_R^{-1}}_{L^1(\beta)}\\
		&\qquad\qquad\qquad\qquad\qquad\meg C_3 t^m \norm{f}_{L^p\beta)} \norm{g}_{B^{p,q}_{\alpha,\unif}(\beta)}
	\end{split}
	\]
	and, analogously,
	\[
	\norm{(\widetilde W^{(m)}_t f)(W^{(k)}_t g-\widetilde W^{(k)}_t g)}_{L^p(\beta)}\meg C_3 t^m \norm{f}_{L^p(\beta)} \norm{g}_{B^{p,q}_{\alpha,\unif}(\beta)}.
	\]
	Consequently, we may reduce to estimating
	\[
	\sum_{ k=0}^{m-1} \norm*{t^{ -\alpha/\grado}   \norm*{(\widetilde W^{(m)}_t f)(\widetilde W^{(k)}_t g)}_{L^p(\beta)} }_{L^q_t(\mi_1)} .
	\]	
	Then, take a $(\delta,2)$-lattice $(x_j)_{j\in J}$ for some $\delta>0$, and let $(\phi_j)$ be a bounded family of elements of $C^\infty_c(G)$ such that, setting $\widetilde \phi_j\coloneqq \phi_j(x_j^{-1}\,\cdot\,)$ for every $j\in J$, one has $\sum_j \widetilde \phi_j=1$ on $G$. Take $R>0$ so that each $\phi_j$ is supported in $B(e,R)$, and take $\psi\in C^\infty_c(G)$ such that $\chi_{B(e,R+2)}\meg \psi\meg \chi_{B(e,R+3)}$. Observe that, by Lemma~\ref{lem:50}, there is a partition $J_1,\dots, J_N$ of $J$ such that $d(x_j, x_{j'})\Meg 2R+6$ for every two distinct $j,j'\in J_h$, for every $h=1,\dots, N$.
	Observe that, since $\widetilde W^{(m)}_t (\widetilde \phi_j f)$ is supported in $B(x_j, R+1)$, and since the $B(x_j,R+1)$, for $j\in J_h$ ($h=1,\dots, N$), are pairwise disjoint,
	\[
	\begin{split}
		\norm{(\widetilde W^{(m)}_t f)(\widetilde W^{(k)}_t g)}_{L^p(\beta)}&= \norm*{ \sum_j\widetilde W^{(m)}_t (\widetilde \phi_j f)(\widetilde W^{(k)}_t g)  }_{L^p(\beta)}\\
		&\meg \sum_{h=1}^N\norm*{ \sum_{j\in J_h}\widetilde W^{(m)}_t (\widetilde \phi_j f)(\widetilde W^{(k)}_t g)  }_{L^p(\beta)}\\
		&=\sum_{h=1}^N\norm*{\norm{  \widetilde W^{(m)}_t (\widetilde \phi_j f)(\widetilde W^{(k)}_t g)  }_{L^p(\beta)}}_{\ell^p_{j}(J_h)}\\
		&\meg \sum_{h=1}^N\norm*{\norm{  \widetilde W^{(m)}_t (\widetilde \phi_j f)}_{L^p(\beta)} }_{\ell^p_{j}(J_h)}\norm{\widetilde W^{(k)}_t g}_{L^\infty(\beta)} \\
		&\meg  \sum_{h=1}^N\norm{  \widetilde W^{(m)}_t (\widetilde \phi^{(h)} f)}_{L^p(\beta)} \norm{\widetilde W^{(k)}_t g}_{L^\infty(\beta)},
	\end{split} 
	\]
	where $\widetilde \phi^{(h)}=\sum_{j\in J_h} \widetilde \phi_j$. Using Proposition~\ref{prop:34} and the uniform boundedness of the $\widetilde W^{(k)}_t$ on $L^\infty(\beta)$, one may then conclude that there is a constant $C_4>0$ such that
	\[
	\Bs^{p,q}_{\alpha,m}(\Pi_f g) \meg C_4\bigg(\norm{f}_{L^p(\beta)}+\sum_{j\in J_h}\norm{\widetilde \phi^{(h)} f }_{B^{p,q}_\alpha(\beta)}\bigg) \norm{g}_{B^{p,q}_{\alpha,\unif}(\beta)}.
	\]
	Since clearly $\widetilde \phi^{(h)}\in W^{\infty,\infty}(\beta)$ for every $h=1,\dots ,N$, the desired estimate follows from Theorem~\ref{teo:6}.
	
	Let us now estimate $\Bs^{p,q}_{\alpha,m}(\Pi_g f)$. Arguing as above one may reduce to estimating
	\[
	\sum_{ k=0}^{m-1} \norm*{t^{ -\alpha/\grado}   \norm*{(\widetilde W^{(k)}_t f)(\widetilde W^{(m)}_t g)}_{L^p(\beta)} }_{L^q_t(\mi_1)}. 
	\]
	In this case, one may observe that $\chi_{B(x_j,R+1)}\widetilde W^{(m)}_t g=\chi_{B(x_j,R+1)}\widetilde W^{(m)}_t (\widetilde \psi_j g)$ for every $j\in J$, thanks to our choice of $\psi$, so that by Lemma~\ref{lem:3} there is a constant $C_5>0$ such that
	\[
	\begin{split}
		\norm*{(\widetilde W^{(k)}_t f)(\widetilde W^{(m)}_t g)}_{L^p(\beta)}&= \norm*{\sum_{j\in J} \widetilde W^{(k)}_t (\widetilde \phi_j f) \widetilde W^{(m)}_t g  }_{L^p(\beta)}\\
		&\meg \sum_{h=1}^N\norm*{ \sum_{j\in J_h}\widetilde W^{(k)}_t (\widetilde \phi_j f) \widetilde W^{(m)}_t (\widetilde \psi_j g)  }_{L^p(\beta)}\\
		&=\sum_{h=1}^N\norm*{\norm{ \widetilde W^{(k)}_t (\widetilde \phi_j f) \widetilde W^{(m)}_t (\widetilde \psi_j g)  }_{L^p(\beta)}}_{\ell^p_{j}(J_h)}\\
		&\meg \sum_{h=1}^N\norm*{\norm{\widetilde W^{(k)}_t (\widetilde \phi_j f)}_{L^\infty(\beta)} \norm{ \widetilde W^{(m)}_t (\widetilde \psi_j g)  }_{L^p(\beta)}}_{\ell^p_{j}(J_h)}\\
		&\meg C_5\sum_{h=1}^N\norm*{\norm{\widetilde \phi_j f}_{L^\infty(\beta)} \norm{ \widetilde W^{(m)}_t (\widetilde \psi_j g)  }_{L^p(\beta)}}_{\ell^p_{j}(J_h)}.
	\end{split} 
	\]
	Using the fact that $p\meg q$, by Proposition~\ref{prop:34} we see that there is a constant $C_6>1$ such that
	\[
	\begin{split}
		\norm*{t^{-\alpha/\grado}\norm*{(\widetilde W^{(k)}_t f)(\widetilde W^{(m)}_t g)}_{L^p(\beta)}}_{L^q_t(\mi_1)}&\meg C_5 \sum_{h=1}^N \norm*{\norm{\widetilde \phi_j  f}_{L^\infty(\beta)}\norm*{t^{-\alpha/\grado} \norm{ \widetilde W^{(m)}_t (\widetilde \psi_j g)  }_{L^p(\beta)}}_{L^q_t(\mi_1)}}_{\ell^p_{j}(J_h)}\\
		&\meg C_6 \sum_{h=1}^N \norm*{\norm{\widetilde \phi_j  f}_{L^\infty(\beta)}\norm{\widetilde \psi_j g}_{B^{p,q}_{\alpha}(\beta)}}_{\ell^p_{j}(J_h)}\\
		&\meg C_6^2 N^{1/p'} \norm*{\norm{\widetilde \phi_j  f}_{L^\infty(\beta)}}_{\ell^p_{j}(J)}\norm{ g}_{B^{p,q}_{\alpha,\unif}(\beta)}.
	\end{split}
	\]
	Since, by Propositions~\ref{prop:2} and~\ref{prop:30}, there is a constant $C_7>1$ such that
	\[
	\begin{split}
		\norm*{\norm{\widetilde \phi_j  f}_{L^\infty(\beta)}}_{\ell^p_{j}(J)}&\meg C_7 \norm*{\norm{\widetilde \phi_j  f}_{B^{p,p}_{Q_*/p+\eps}(\beta)}}_{\ell^p_{j}(J)} \\
		&\meg C_7^2 \norm{   f}_{B^{p,p}_{Q_*/p+\eps}(\beta)}\\
		&\meg C_7^3 \norm{f}_{B^{p,q}_\alpha(\beta)}
	\end{split}
	\] 
	for any fixed $\eps\in (0,\alpha-Q_*/p)$, the desired estimates follow also in this case.

	\textsc{Step II} We now estimate $\Bs^{p,q}_{\alpha,m}(\Pi(f,g))$. Observe first that there are two families $(X_\ell)_{\ell\in L}$ and $(Y_\ell)_{\ell\in L}$ of left-invariant differential operators on $G$ such that
	\[
	\Lc^{m}(\phi \psi)=\sum_{\ell \in L} (X_\ell \phi)(Y_\ell \psi)
	\]
	for every $\phi,\psi\in C^\infty(G)$, and
	such that $\deg(X_\ell)+\deg(Y_\ell)\meg m\grado $ for every $\ell\in L$.
	Then, observe that
	\[
	\begin{split}
		&\abs{W_s^{(m)} W^{(m)}_t[(W^{(h)}_t f)(W^{(k)}_t g)]}= (s t)^{m} \abs{ \Lc^{m}\ee^{-(s+t)\Lc} \Lc^m[(W^{(h)}_t f)(W^{(k)}_t g)]}\\
		&\qquad \meg C_1 s^{m}(s+t)^{-m} t^{m } \sum_{\ell\in L}T_{b,s+t} [ (X_\ell W^{(h)}_t f)(Y_\ell W^{(k)}_t g)].
	\end{split}
	\]
	Consequently, by Lemma~\ref{lem:3} there is a constant $C_8>0$ such that
	\[
	\begin{split}
		\norm{W_s^{(m)} W^{(m)}_t[(W^{(h)}_t f)(W^{(k)}_t g)]}_{L^p(\beta)}&\meg C_1 s^{m}(s+t)^{-m} t^m \sum_{\ell\in L}\norm{T_{b,s+t} [ (X_\ell W^{(h)}_t f)(Y_\ell W^{(k)}_t g)]}_{L^p(\beta)}\\
		&\meg C_{8} s^{m}(s+t)^{-m} t^m \sum_{\ell\in L}\norm{(X_\ell W^{(h)}_t f)(Y_\ell W^{(k)}_t g)}_{L^p(\beta)},
	\end{split}
	\]
	so that by Lemma~\ref{lem:25} there is a constant $C_9>0$ such that
	\[
	\begin{split}
		&\norm*{s^{-\alpha/\grado} \int_0^1\norm{W_s^{(m)} W^{(m)}_t[(W^{(h)}_t f)(W^{(k)}_t g)]}_{L^p(\beta)}\,\dd \mi_1(t)  }_{L^q_s(\mi_1)}\\
		&\qquad\meg C_9\sum_{\ell\in L}  \norm*{t^{m-\alpha/\grado} \norm{(X_\ell W^{(h)}_t f)(Y_\ell W^{(k)}_t g)}_{L^p(\beta)} }_{L^q_t(\mi_1)}.
	\end{split}
	\]
	In addition, arguing as in~\textsc{step I}, we see that there is a constant $C_{10}>0$ such that
	\[
	\norm{(X_\ell W^{(h)}_t f)(Y_\ell W^{(k)}_t g)}_{L^p(\beta)}\meg \norm{(\widetilde W^{(h,1,\ell)}_t f)(\widetilde  W^{(k,2,\ell)}_t g)}_{L^p(\beta)}+C_{10}  \norm{f}_{L^p(\beta)} \norm{g}_{B^{p,q}_{\alpha,\unif}(\beta)},
	\]
	where $\widetilde W^{(h,1,\ell)}_t f= f*(\eta X_\ell (t \Lc)^h h_t)$ and $\widetilde W^{(k,2,\ell)}_t g= g*(\eta Y_\ell (t \Lc)^k h_t)$. Now, if $\deg(Y_\ell)\meg m \grado/2=[\alpha/\grado]+1$, then arguing as above we see that there is a constant $C_{11}>1$ such that
	\[
	\begin{split}
		\norm{(\widetilde W^{(h,1,\ell)}_t f)(\widetilde  W^{(k,2,\ell)}_t g)}_{L^p(\beta)}&\meg \norm*{\sum_{j\in J}  \widetilde W^{(h,1,\ell)}_t (\widetilde \phi_j f)(\widetilde  W^{(k,2,\ell)}_t g)  }_{L^p(\beta)}\\
		&\meg \sum_{r=1}^N\norm*{\sum_{j\in J_r}  \widetilde W^{(h,1,\ell)}_t (\widetilde \phi_j f)(\widetilde  W^{(k,2,\ell)}_t g)  }_{L^p(\beta)}\\
		&\meg \sum_{r=1}^N\norm*{\norm*{  \widetilde W^{(h,1,\ell)}_t (\widetilde \phi_j f)(\widetilde  W^{(k,2,\ell)}_t g)  }_{L^p(\beta)}}_{\ell^p_j(J_r)}\\
		&\meg  \sum_{r=1}^N\norm*{\norm*{  \widetilde W^{(h,1,\ell)}_t (\widetilde \phi_j f)}_{L^p(\beta)}}_{\ell^p_j(J_r)} \norm{\widetilde  W^{(k,2,\ell)}_t g}_{L^\infty(\beta)}\\
		&\meg C_{11}t^{-\deg(Y_\ell)/\grado}\sum_{r=1}^N \norm*{  \widetilde W^{(h,1,\ell)}_t (\widetilde \phi^{(r)} f)}_{L^p(\beta)}  \norm{g}_{L^\infty(\beta)}\\
		&\meg C_{11}^2 t^{-\deg(Y_\ell)/\grado}\sum_{r=1}^N\bigg(   \norm{ X_\ell W^{(h)}_t (\widetilde \phi^{(r)} f)}_{L^p(\beta)}  + \norm{\widetilde \phi^{(r)} f}_{L^p(\beta)} \bigg)\norm{g}_{B^{p,q}_{\alpha,\unif}(\beta)},
	\end{split}
	\]
	so that by~\cite[Propositions 6.9]{BCP} we may find a constant $C_{12}>0$ such that
	\[
	\norm*{t^{m-\alpha/\grado} \norm{(X_\ell W^{(h)}_t f)(Y_\ell W^{(k)}_t g)}_{L^p(\beta)} }_{L^q_t(\mi_1)}\meg C_{12} \sum_{r=1}^N\norm{\widetilde \phi^{(r)} f}_{B^{p,q}_{\alpha }(\beta)}  \norm{g}_{B^{p,q}_{\alpha,\unif}(\beta)},
	\] 
	whence the desired estimate by Theorem~\ref{teo:6} as before.
	If, otherwise, $\deg(Y_\ell)> m \grado/2$, so that $\deg(X_\ell)\meg m\grado/2$, then there is a constant $C_{13}>0$ such that
	\[
	\begin{split}
		\norm{(\widetilde W^{(h,1,\ell)}_t f)(\widetilde  W^{(k,2,\ell)}_t g)}_{L^p(\beta)}&\meg \norm*{\sum_{j\in J}  \widetilde W^{(h,1,\ell)}_t (\widetilde \phi_j  f) \widetilde  W^{(k,2,\ell)}_t(\widetilde \psi_j g)  }_{L^p(\beta)}\\
		&\meg \sum_{r=1}^N\norm*{\sum_{j\in J_r} \widetilde W^{(h,1,\ell)}_t (\widetilde \phi_j  f) \widetilde  W^{(k,2,\ell)}_t(\widetilde \psi_j g)   }_{L^p(\beta)}\\
		&= \sum_{r=1}^N\norm*{\norm*{ \widetilde W^{(h,1,\ell)}_t (\widetilde \phi_j  f) \widetilde  W^{(k,2,\ell)}_t(\widetilde \psi_j g)  }_{L^p(\beta)}}_{\ell^p_j(J_r)}\\
		&\meg  \sum_{r=1}^N\norm*{\norm{\widetilde W^{(h,1,\ell)}_t (\widetilde \phi_j  f) }_{L^\infty(\beta)} \norm{\widetilde  W^{(k,2,\ell)}_t( \widetilde \psi_j g)}_{L^p(\beta)}}_{\ell^p_j(J_r)}\\
		&\meg C_{13} t^{-\deg(X_\ell)/\grado}\sum_{r=1}^N\norm*{\norm{\widetilde \phi_j f }_{L^\infty(\beta)} \norm{\widetilde  W^{(k,2,\ell)}_t( \widetilde \psi_j g)}_{L^p(\beta)}}_{\ell^p_j(J_r)}.
	\end{split}
	\]
	Using the fact that $p\meg q$, by Proposition~\ref{prop:2} we see that there is a constant $C_{14}>0$ such that
	\[
	\begin{split}
		&\norm*{ t^{m-\deg(X_\ell)/\grado-\alpha/\grado} \norm*{\norm{\widetilde \phi_j f }_{L^\infty(\beta)} \norm{\widetilde  W^{(k,2,\ell)}_t( \widetilde \psi_j g)}_{L^p(\beta)}}_{\ell^p_j(J_r)} }_{L^q_t(\mi_1)} \\
			&\qquad\meg   \norm*{\norm{\widetilde \phi_j f }_{L^\infty(\beta)} \norm*{t^{m-\deg(X_\ell)/\grado-\alpha/\grado}\norm{\widetilde  W^{(k,2,\ell)}_t( \widetilde \psi_j g)}_{L^p(\beta)} }_{L^q_t(\mi_1)}}_{\ell^p_j(J_r)}\\
		&\qquad\meg C_{14} \norm*{ \norm{\widetilde \phi_j f}_{B^{p,p}_{Q_*/p+\eps}(\beta)}  }_{\ell^p_j(J_r)} \sup_j\norm*{t^{m-\deg(X_\ell)/\grado-\alpha/\grado}\norm{\widetilde  W^{(k,2,\ell)}_t( \widetilde \psi_j g)}_{L^p(\beta)} }_{L^q_t(\mi_1)}
	\end{split}
	\]
	with $\eps$ as in~\textsc{step I}. Arguing as in~\textsc{step I} we then deduce the desired estimate.
	
	\textsc{Step III}   We now complete the proof. Observe first that there is a constant $C_{15}>0$ such that
	\[
	\norm{f g}_{L^p(\beta)}\meg \norm{f}_{L^{p }(\beta)}\norm{g}_{L^{\infty}(\beta)}\meg C_{15} \norm{f}_{B^{p ,q}_\alpha(\beta)} \norm{g}_{B^{p,q}_{\alpha,\unif}(\beta)}.
	\]
	By~\textsc{steps I} and~\textsc{II}, it only remains to estimate $\Bs^{p,q}_{\alpha,m}( W^{(h)}_1[ (W^{(k)}_1 f)(W^{(n)}_1 g) ] )$ for every $h,k,n=0,\dots, m-1$. Then, observe that
	\[
	\abs{W^{(m)}_t W^{(h)}_1[ (W^{(k)}_1 f)(W^{(n)}_1 g) ]}\meg C_1 t^m T_{b,1+t} [ (W^{(k)}_1 f)(W^{(n)}_1 g) ]\meg C_1 t^m 2^{Q_*/\grado}T_{b,2} [ (W^{(k)}_1 f)(W^{(n)}_1 g) ]
	\]
	for every $t\in (0,1]$, so that by Lemma~\ref{lem:3} there is a constant $C_{16}>1$ such that
	\[
	\begin{split}
		\norm{W^{(m)}_t W^{(h)}_1[ (W^{(k)}_1 f)(W^{(n)}_1 g) ]}_{L^p(\beta)}&\meg C_{16}\norm{(W^{(k)}_1 f)(W^{(n)}_1 g)}_{L^p(\beta)}\\
		&\meg C_{16}\norm{ W^{(k)}_1 f}_{L^p(\beta)}\norm{W^{(n)}_1 g}_{L^\infty(\beta)}\\
		&\meg C_{16}^2 \norm{f}_{L^p(\beta)}\norm{g}_{L^\infty(\beta)}\\
		&\meg C_{16}^3\norm{f}_{B^{p ,q}_\alpha(\beta)} \norm{g}_{B^{p,q}_{\alpha,\unif}(\beta)}.
	\end{split}
	\]
	The assertion follows.	 
\end{proof}

We now pass to the case $p<q$. We begin with a lemma that proves the consistency of the definition of the spaces $M^{p,q}_\alpha$ introduced below.

\begin{lem}\label{lem:57}
	Take $p,q\in [1,\infty]$ and $\alpha>0$.
	In addition, take $\delta>0$, $R\Meg 2$, a $(\delta,R)$-lattice  $(x_j)_{j\in J}$  on $G$, and two bounded families $(\phi_j)_{j\in J}$ and $(\psi_{j'})_{j'\in J'}$ in $C^\infty_c(G)$. Set $\widetilde \phi_j\coloneqq \phi_j(x_j^{-1}\,\cdot\,)$   for every $j\in J$, and assume that $\sum_j  \widetilde \phi_j=1$. Then, there is a constant $C>0$ such that 
	\[
	\norm*{ \sum_{j'} a'_{j'} \widetilde \psi_{j'} f }_{B^{p,q}_\alpha(\beta_L)}\meg C \norm{a'_{j'}}_{\ell^p_{j'}(J')} \sup_{\norm{a_j}_{\ell^p_j(J)}\meg 1} \norm*{  \sum_j a_j \widetilde \phi_j f }_{B^{p,q}_\alpha(\beta_L)}
	\]
	for every $(\delta,R)$-lattice  $(x'_{j'})_{j'\in J'}$ on $G$, for every $f\in \Sc'(G)$, and for every $(a'_{j'})\in \ell^p(J')$, where  $\widetilde \psi_{j'}\coloneqq \psi_{j'}(x'^{-1}_{j'}\,\cdot\,)$  for every $j'\in J'$.
\end{lem}

\begin{proof}
	Observe that, by Lemma~\ref{lem:50}, there is $N\in\N$ such that $\card(J_{j'})\meg N$ for every $j'\in J'$, where $J_{j'}=\Set{j\in J\colon \widetilde \psi_{j'} \widetilde \phi_j\neq 0}$. In addition, by the same reference, we may partition $J'$ in subsets $J'_1,\dots, J'_N$ such that the $J_{j'}$, $j'\in J'_h$, are pairwise disjoint for every $h=1,\dots, N$. Then, take $(a'_{j'})\in \ell^p(J')$ and set $a^{(h)}_j\coloneqq \sum_{j'\in J'_h} \chi_{J_{j'}}(j) a'_{j'}$, so that\footnote{Notice that, for every $h=1,\dots, N$ and for every $j\in J$, there is at most one $j'\in J'$ such that $j\in J_{j'}$, so that $a^{(h)}_j=a'_{j'}$ if such a $j'$ exists, and $a^{(h)}_j=0$ otherwise.}
	\[
	\begin{split}
		\sum_{j'} a'_{j'} \widetilde \psi_{j'}&=\sum_{h=1}^N \sum_{j'\in J'_h}\sum_{j\in J_{j'}} a'_{j'}\widetilde \psi_{j'}  \widetilde \phi_{j}\\
		&= \sum_{h=1}^N \sum_{j'\in J'_h}\sum_{j\in J_{j'}} \widetilde \psi_{j'}  a^{(h)}_j\widetilde \phi_{j}\\
		&=  \sum_{h=1}^N \sum_{j'\in J'_h}  \widetilde \psi_{j'}  \sum_{j\in J} a^{(h)}_j\widetilde \phi_{j}.
	\end{split}
	\]
	Consequently, by Theorem~\ref{teo:6} there is a constant $C_1>0$ such that
	\[
	\begin{split} 
		\norm*{ \sum_{j'} a'_{j'} \widetilde \psi_{j'} f }_{B^{p,q}_\alpha(\beta_L)}\meg \sum_{h=1}^N \norm*{  \sum_{j'\in J'_h}  \widetilde \psi_{j'}  \sum_{j\in J} a^{(h)}_j\widetilde \phi_{j} f }_{B^{p,q}_\alpha(\beta_L)} \meg \sum_{h=1}^N \norm*{  \sum_{j'\in J'_h}  \widetilde \psi_{j'}  }_{B^{\infty,q}_\alpha(\beta_L)}\norm*{  \sum_{j\in J} a^{(h)}_j\widetilde \phi_{j} f }_{B^{p,q}_\alpha(\beta_L)}.
	\end{split}
	\]
	Since clearly $ \sum_{j'\in J'_h}  \widetilde \psi_{j'} \in W^{\infty,\infty}(\beta_L)\subseteq B^{\infty,q}_\alpha(\beta_L)$ (cf.~Proposition~\ref{prop:7}) and   $\norm{a^{(h)}_j}_{\ell^p_j(J)}\meg N^{1/p} \norm{a'_{j'}}_{\ell^p_{j'}(J')}$ for every $h=1,\dots,N$,  the assertion follows.
\end{proof}

\begin{deff}\label{def:4}
	Take $p,q\in [1,\infty]$ and $\alpha>0$. We define $M^{p,q}_\alpha$ as the space of $f\in \Sc'(G)$ such that
	\[
	\norm{f}_{M^{p,q}_\alpha}\coloneqq \sup_{\norm{a_j}_{\ell^p_j(J)}\meg 1} \norm*{ \sum_j a_j \widetilde \phi_j f }_{B^{p,q}_\alpha(\beta_L)}<\infty,
	\]
	endowed with the corresponding topology, where the $\widetilde \phi_j$ are as in Lemma~\ref{lem:57}.
\end{deff}

Observe that, if $N\coloneqq\max \sum_j \chi_{B(x_j,  (R+1)\delta)}$, then we may take $(\phi_j)$ such that $\phi_j\Meg 1/N$ on $B(e,\delta)$ for every $j\in J$.\footnote{For example, one may take $\psi\in C^\infty_c(G)$ in such a way that $\chi_{B(e,R\delta)}\meg\psi\meg \chi_{B(e,(R+1)\delta)}$, so that $\Psi\coloneqq \sum_j \psi(x_j^{-1}\,\cdot\,)$ takes values in $[1,N]$, and then define $\phi_j\coloneqq \psi/\Psi(x_j\,\cdot\,)$ for every $j\in J$. } In this case, it is readily seen that there is a constant $c>0$ such that  
\[
\frac 1 c\norm{a_j}_{\ell^p_j(J)}\meg \norm{a_j \widetilde \phi_j}_{B^{p,q}_\alpha(\beta_L)} \meg c\norm{a_j}_{\ell^p_j(J)}.
\]
Consequently, the space $M^{p,q}_\alpha$ is defined essentially in the same way as $\Mc(B^{p,q}_\alpha(\beta_L))$, except for the fact that, instead of testing the continuity of pointwise multiplication on the whole of $B^{p,q}_\alpha(\beta_L)$, we restrict to a particularly simple closed vector subspace (namely, the one generated by the $\widetilde \phi_j$).

\begin{lem}\label{lem:58}
	Take $p,q\in [1,\infty]$ and $\alpha >0$. Then, $\Mc(B^{p,q}_\alpha(\beta)) \subseteq M^{p,q}_\alpha\subseteq B^{p,q}_{\alpha,\unif}(\beta_L)$ continuously.
\end{lem}

\begin{proof}
	The continuity of the inclusion $M^{p,q}_\alpha\subseteq B^{p,q}_{\alpha,\unif}(\beta_L)$ follows from Lemma~\ref{lem:57}, choosing $J'=J$, $x'_j\coloneqq x x_j$, $a'_j=\chi_{\Set{j_0}}(j)$ for some $j_0\in J$, and any non-zero $\psi_{j_0}$.  
	For the other inclusion, by Proposition~\ref{prop:5}  it will suffice to show that the mapping
	\[
	\ell^p(J)\ni(a_j)\mapsto \sum_j a_j \widetilde \phi_j \in B^{p,q}_\alpha(\beta_L)
	\]
	is continuous. In fact, it is sufficient to observe that the same mapping is continuous from $\ell^p(J)$ into $W^{p,\infty}(\beta_L)$, and to apply Proposition~\ref{prop:7} (with the aid of Proposition~\ref{prop:5}).
\end{proof}

\begin{teo}\label{teo:21}
	Take $p,q\in [1,\infty]$ and take $\alpha\Meg Q_*/p$ such that either $q=1$ and $\alpha>0$, or $\alpha>Q_*/p$. 
	Then, $\Mc(B^{p,q}_\alpha(\beta))=M^{p,q}_\alpha$. In addition, the canonical bilinear mapping
	\[
	B^{p,q}_\alpha(\beta)\times M^{p,q}_\alpha\ni (f,g)\mapsto f g \in B^{p,q}_\alpha(\beta)
	\]
	is continuous.
\end{teo}

This result extends~\cite[Theorem 4.3]{BPV3} to the case of general filtrations and general $\alpha$. The proof is different.
 
Observe that this result, combined with Theorem~\ref{teo:20} and Proposition~\ref{prop:21}, shows that $M^{p,q}_\alpha=B^{p,q}_{\alpha,\unif}(\beta_L)$ when $p\meg q$, and that $M^{\infty,q}_\alpha=B^{\infty,q}_\alpha(\beta)$.

\begin{proof}
	By means  of Proposition~\ref{prop:5}, we may reduce to the case $\beta=\beta_L$. Take $(x_j)$, $(\phi_j)$, and $(\widetilde \phi_j)_{j\in J}$ as in Lemma~\ref{lem:57}. In addition, by Lemma~\ref{lem:50} we may find $N\in\N$ and a partition $J_1,\dots, J_N$ of $J$ such that the $d(x_j,x_{j'})\Meg 2R'+6$  for every two distinct $j,j'\in J_h$, $h=1,\dots, N$, where $R'>0$ is such that each $\phi_j$ is supported in $B(e, R')$. Take $\psi\in C^\infty_c(G)$ such that $\chi_{B(e,R'+2)}\meg \psi\meg \chi_{B(e,R'+3)}$, and define $\widetilde \psi_j\coloneqq \psi_j(x_j^{-1}\,\cdot\,)$ for every $j\in J$.
	
	Observe that the continuous  inclusion $\Mc(B^{p,q}_\alpha(\beta))\subseteq M^{p,q}_\alpha$ follows from Lemma~\ref{lem:58}, so that we may reduce to proving that the bilinear mapping $B^{p,q}_\alpha(\beta)\times M^{p,q}_\alpha \ni(f,g)\mapsto f g\in  B^{p,q}_\alpha(\beta)$ is continuous. Since $M^{p,q}_\alpha\subseteq B^{p,q}_{\alpha,\unif}(\beta)$ by Lemma~\ref{lem:58}, arguing as in the proof of Theorem~\ref{teo:20}, we may reduce to estimating $\Bs^{p,q}_{\alpha,m}(\Pi_g ^{(t')}f)$ and  $\Bs^{p,q}_{\alpha,m}((\widetilde W^{(h,1,\ell)}_t f) (\widetilde W^{(k,2,\ell)}_t g)  )$ for $h,k=0,\dots, m-1$ and $\ell\in L$ such that $\deg(Y_\ell)> m\grado/2$, with the notation of the proof of Theorem~\ref{teo:20}.   
	
	Concerning $\Bs^{p,q}_{\alpha,m}(\Pi_g^{(t')} f)$, arguing as in the proof of Theorem~\ref{teo:20} we see that there is a constant $C_1>0$ such that
	\[
	\begin{split}
		\Bs^{p,q}_{\alpha,m}(\Pi_g^{(t')} f)&\meg C_1 \norm{f}_{B^{p,q}_\alpha(\beta)} \norm{g}_{B^{p,q}_{\alpha,\unif}(\beta)}+ C_1\sum_{h=1}^N \norm*{t^{-\alpha/\grado} \norm*{\norm{\widetilde \psi_j f}_{L^\infty(\beta)}  \norm{\widetilde W^{(m)}_t(\widetilde \phi_j g)}_{L^p(\beta)}  }_{\ell^p_j(J_h)}   }_{L^q_t(\mi_1)}\\
			&=C_1 \norm{f}_{B^{p,q}_\alpha(\beta)} \norm{g}_{B^{p,q}_{\alpha,\unif}(\beta)}+ C_1\sum_{h=1}^N \norm*{t^{-\alpha/\grado} \norm*{ \widetilde W^{(m)}_t\bigg(\sum_{j\in J_h}a_j \widetilde \phi_j g\bigg)}_{L^p(\beta)}     }_{L^q_t(\mi_1)}\\
	\end{split}
	\]
	where $a_j\coloneqq \norm{\widetilde \psi_j f}_{L^\infty(\beta)} $ for every $j\in J$. Now, by Proposition~\ref{prop:34} and Lemma~\ref{lem:57} we see that there is a constant $C_2>0$ such that
	\[
	\begin{split}
	\norm*{t^{-\alpha/\grado} \norm*{ \widetilde W^{(m)}_t\bigg(\sum_{j\in J_h}a_j \widetilde \phi_j g\bigg)}_{L^p(\beta)}     }_{L^q_t(\mi_1)}\meg C_2 \norm{a_j}_{\ell^p_j(J_h)} \norm{g}_{M^{p,q}_\alpha}. 
	\end{split}
	\]
	If $\alpha >Q_*/p$, then the desired estimate   follows  from the fact that, by Propositions~\ref{prop:2} and~\ref{prop:30}, there is a constant $C_3>1$ such that
	\[
	\begin{split}
	\norm{a_j}_{\ell^p_j(J_h)}&\meg C_3 \norm*{\norm{\widetilde \phi_j f}_{B^{p,p}_{Q_*/p+\eps}(\beta)} }_{\ell^p_j(J)}\\
		&\meg C_3^2 \norm{f}_{B^{p,p}_{Q_*/p+\eps}(\beta)}\\
		&\meg C_3^3 \norm{f}_{B^{p,q}_\alpha(\beta)}
	\end{split}
	\]
	for any fixed $\eps\in (0,\alpha-Q_*/p)$. If, otherwise, $\alpha=Q_*/p$, then $q=1$ and $p<\infty$,  so that by Proposition~\ref{prop:2} there is a constant $C_4>1$ such that
	\[
	\begin{split}
		\norm{a_j}_{\ell^p_j(J_h)}&\meg C_4 \norm*{\norm{\widetilde \phi_j f}_{B^{p,1}_{Q_*/p}(\beta)} }_{\ell^p_j(J)}\\
			&\meg C_4^2 \norm*{ t^{-\alpha/\grado}\norm{\widetilde \phi_j f}_{L^p(\beta)} }_{\ell^p_j(J)}+ C_4^2\norm*{\norm*{ t^{-\alpha/\grado} \norm{\widetilde W^{(m)}_t(\widetilde \phi_j f)}_{L^p(\beta)}}_{L^1_t(\mi_1)} }_{\ell^p_j(J)}\\
			&\meg C_4^2 \norm{f}_{L^p(\beta)}+ C_4^2\sum_{r=1}^N \norm*{\norm*{ t^{-\alpha/\grado} \norm{\widetilde W^{(m)}_t(\widetilde \phi_j f)}_{L^p(\beta)}}_{L^1_t(\mi_1)} }_{\ell^p_j(J_r)}\\
			&\meg C_4^2 \norm{f}_{L^p(\beta)}+ C_4^2\sum_{r=1}^N \norm*{ t^{-\alpha/\grado}\norm*{ \norm{\widetilde W^{(m)}_t(\widetilde \phi_j f)}_{L^p(\beta)}}_{\ell^p_j(J_r)}}_{L^1_t(\mi_1)} \\
			&=C_4^2 \norm{f}_{L^p(\beta)}+ C_4^2\sum_{r=1}^N \norm*{   t^{-\alpha/\grado}\norm{  \widetilde W^{(m)}_t(\widetilde \phi^{(r)} f)}_{L^p(\beta)} }_{L^1_t(\mi_1)},
	\end{split}
	\]
	so that the desired estimate follows from Corollary~\ref{cor:20}.
	
	The other term may be estimated in a similar way.
\end{proof}

\section{Characterization by Differences}\label{sec:7}

\begin{deff}
	Define $\Delta_y f(x) \coloneqq f(x y)-f(x)$ for every function $f$ on $G$ and for every $x,y\in G$.  Define inductively $\Delta_y^{(h)}$, $h\in\N$, so that $\Delta_y^{(1)}=\Delta_y$ and $\Delta_y^{(h+1)}=\Delta_y \Delta_y^{(h)}$ for every $h\in\N$.
	
	In addition, set $\dd_0\coloneqq \min\Set{\lambda>0\colon \gf_\lambda \neq \Set{0}}$ and define 
	\[
	(S^q_{\alpha,h} f)(x)\coloneqq \norm*{t^{-\alpha-Q_*} \int_{B(e,t)} \abs{(\Delta_y^{(h)} f)(x)}\,\dd \beta(y)  }_{L^q_t(\mi_1)}
	\]
	for every $q\in [1,\infty]$, for every $\alpha>0$, for every $h\in\N$, for every $x\in G$, and for every $\beta$-measurable function $f$ on $G$. We simply write $S^q_{\alpha}$ instead of $S^q_{\alpha,1}$.
\end{deff}

Notice that the functional $S^q_{\alpha,h}  $ is equivalent to the somewhat more natural functional
\[
(\tilde S^q_{\alpha,h} f)(x)\coloneqq \norm*{t^{-\alpha } \dashint_{B(e,t)} \abs{(\Delta^{(h)}_y f)(x)}\,\dd \beta(y)  }_{L^q_t(\mi_1)}.
\]
We prefer to use $S^q_{\alpha,h}$ for technical reasons.  

\begin{teo}\label{teo:11}
	Take $p,q\in [1,\infty]$ and $\alpha\in (0,\dd_0)$. Then, there is a constant $C>1$ such that the following hold:
	\begin{enumerate}
		\item[\textnormal{(i)}] for every $f\in \Sc'(G)$,
		\begin{equation}\label{eq:2}
		\frac{1}{C}\norm{f}_{B^{p,q}_\alpha(\beta)}\meg \norm{f}_{L^p(\beta)}+\norm*{ \abs{y}_*^{-\alpha-Q_*/q} \norm{\Delta_y f}_{L^p(\beta)} }_{L^q_y(B(e,1),\beta)} \meg C\norm{f}_{B^{p,q}_\alpha(\beta)}; 
		\end{equation}
		
		\item[\textnormal{(ii)}] if $p\in (1,\infty)$ and $q\in [1,\infty]$, then for every $f\in \Sc'(G)$,
		\begin{equation}\label{eq:3}
		\frac{1}{C}\norm{f}_{F^{p,q}_\alpha(\beta)}\meg \norm{f}_{L^p(\beta)}+\norm{S^{q}_\alpha f }_{L^p(\beta)} \meg C\norm{f}_{F^{p,q}_\alpha(\beta)}.
		\end{equation}
		\end{enumerate}
\end{teo}

This result extends~\cite[Theorems 8 and 9]{BPV2} to the case of general filtrations. The proof is similar. 
Let us now make some remarks on the statement. First, one would expect the above result to hold for $\alpha\in (0,1)$ instead of $\alpha\in (0,\dd_0)$.  In fact, this result suggests that we should have chosen a different `parametrization' of Besov and Triebel--Lizorkin spaces, and defined the Besov space corresponding to $p,q,\alpha$   as $B^{p,q}_{\alpha \dd_0}(\beta)$, rather than $B^{p,q}_\alpha(\beta)$ (analogously for Triebel--Lizorkin spaces). However, this would have lead to unnatural relations with the action of differential operators. In order to solve this issue one may also require $\gf_1\neq \Set{0}$, that is, $\dd_0=1$; this requirement does no harm to the generality of the construction, since one may simply replace the filtration $(\gf_\lambda)_{\lambda\Meg 0}$ with $(\gf_{\lambda \dd_0})_{\lambda \Meg 0}$: apart from some reparametrizations, this only alters the control modulus $\abs{\,\cdot\,}_*$ near $e$, and this only leads to an improvement of the Gaussian estimates of the heat kernels.

Secondly, one may expect a similar result for higher order differences to hold, as in the classical case. In fact, such generalization is known to hold in the Riemannian case (that is, $\gf_1=\gf$), cf., e.g.,~\cite[Theorem 7.6.3]{TriebelFS2}. Unfortunately, even though we are able to prove one of the two inequalities (and the remaining one for the case of second order differences, for Besov spaces, following~\cite{Saka}), our techniques do not seem to allow us to provide a complete result. In order to better explain this fact, observe that in the classical case (to which the Riemannian case essentially reduces through a clever, yet not effortless, localization procedure) the proof essentially proceeds by finding an inverse of the operator $\Delta_y^{(k)}$ on the space of functions whose Fourier transform is supported in a fixed ball (of arbitrarily large radius). This is relatively easy to do using the Fourier transform (in the classical case), but appears to be substantially more difficult in the general case. 
When dealing with homogeneous groups, using a completely different system of higher order differences (based on the group dilations, in a certain sense), one may in fact follow the classical footsteps  (cf.~\cite{Giulini} for the case of stratified Lie groups). Since filtered groups are essentially modelled on homogeneous groups (for what concerns the analysis of weighted subcoercive operators and their heat semigroups, at least), one might hope to `transfer' a characterization by higher order differences from homogeneous groups to general filtered Lie groups, exactly as one `transfers' the classical result on $\R^n$ to the Riemannian case. However, this does not seem to be possible, since the local geometry of a filtered Lie group and the corresponding contraction do not seem to be sufficiently `comparable' for this programme to work (in contrast to what happens in the Riemannian case).

For these reasons, we shall actually state as separate results all the best statements we can prove on the various inequalities which constitute the content of Theorem~\ref{teo:11}. We begin with a technical lemma.

\begin{lem}	\label{lem:29}
	Take $b,c,\alpha>0$ and $d>1$. Take $\omega>0$ so that $\sup_{r\Meg 1} \ee^{-(\omega/2)r}\beta(B(e,r)) $ is finite, and let $\nu$ be the measure on $G\setminus \Set{e}$ such that $\dd \nu(x)= \abs{x}_*^{-Q_*} \ee^{-\omega \abs{x}_*}\,\dd \beta(x)$. Then, there is a constant $C>0$ such that, for every $q\in [1,\infty]$,
	\[
	\norm*{\int_{G\setminus \Set{e}} \ee^{c\abs{x}_*} (\abs{x}_*/t^{1/d})^\alpha \ee^{-b(\abs{x}_*/t^{1/d})^{d/(d-1)}}   \abs{f(x)}\,\dd \nu(x)}_{L^q_t(\mi_1)}\meg C \norm{f}_{L^q( \nu)}
	\]
	for every $\nu$-measurable function $f$ on $G \setminus \Set{e}$.
\end{lem}

\begin{proof} 	
	Observe first that, by Lemma~\ref{lem:3}, there is a constant $C_1>0$ such that
	\[
	 \ee^{c\abs{x}_*}  \ee^{-b(\abs{x}_*/t^{1/d})^{d/(d-1)}}  \meg  C_1  \ee^{-(b/2)(\abs{x}_*/t^{1/d})^{d/(d-1)}} 
	\]	
	for every $t\in (0,1]$,	so that we may assume that $c=0$, up to replacing $b$ with $b/2$.
	By Schur's lemma (cf., e.g.,~\cite[Lemma I.1]{GrafakosClassical}), it will then suffice to show that
	\[
	\sup_{t\in (0,1]} \int_{G\setminus \Set{e}}   (\abs{x}_*/t^{1/d})^{\alpha } \ee^{-(b/2)(\abs{x}_*/t^{1/d})^{d/(d-1)}} \,\dd \nu(x)
	\]
	and
	\[
	\sup_{x\in G\setminus \Set{e}} \int_0^1    (\abs{x}_*^d/t)^{\alpha/d} \ee^{-(b/2)(\abs{x}_*^d/t)^{1/(d-1)}}   \,\dd \mi_1(t)
	\]
	are finite. The second assertion follows easily, since for every $x\in G\setminus \Set{e}$
	\[
	\int_0^1    (\abs{x}_*^d/t)^{\alpha/d} \ee^{-(b/2)(\abs{x}_*^d/t)^{1/(d-1)}}   \,\dd \mi_1(t)\meg\int_0^{\infty} t^{-\alpha/d} \ee^{- (b/2) t^{-1/(d-1)}}\,\frac{\dd t}{t},
	\]
	which is finite.  
	Concerning the first assertion, observe that there is a constant $C_2>0$ such that $\nu(B(e,2r)\setminus B(e,r))\meg C_2 $ for every $r>0$, thanks to our choice of $\omega$. Then,
	\[
	\begin{split}
		&\int_{G\setminus \Set{e}}   (\abs{x}_*/t^{1/d})^{\alpha } \ee^{-(b/2)(\abs{x}_*/t^{1/d})^{d/(d-1)}} \,\dd \nu(x)\\
			&\qquad\meg   \sum_{j\in \Z} \int_{B(e, 2^{j+1} t^{1/d})\setminus B(e, 2^j t^{1/d})}   (\abs{x}_*/t^{1/d})^{\alpha } \ee^{-(b/2)(\abs{x}_*/t^{1/d})^{d/(d-1)}} \,\dd \nu(x)\\
			&\qquad\meg  C_2  \sum_{j\in \Z} 2^{(j+1)\alpha} \ee^{-(b/2) 2^{j d/(d-1)} },
	\end{split}
	\]
	which is finite.
\end{proof}

\begin{prop}\label{prop:32a}
	Take $p,q\in [1,\infty]$, $k\in \Set{1,2}$, and $\alpha>0$. Then, there is a constant $C>1$ such that  
		\[
			\frac{1}{C}\norm{f}_{B^{p,q}_\alpha(\beta)}\meg \norm{f}_{L^p(\beta)}+\norm*{ \abs{y}_*^{-\alpha-Q_*/q} \norm{\Delta_y^{(k)} f}_{L^p(\beta)} }_{L^q_y(B(e,1),\beta)}  
		\]
		for every $f\in \Sc'(G)$.
\end{prop}

The proof is inspired by the proof of~\cite[Theorem 12]{Saka}.

\begin{proof}
	By means of~\cite[Propositions 6.9]{BCP} and Proposition~\ref{prop:4}, we may assume that $\Lc$ is real, has no constant terms, and is formally self-adjoint with respect to  $ \beta_R$ (that is, $\Lc=\Lc^+$). Then,  $\int_G \Lc h_t\,\dd \beta_R=0$ for every $t>0$. Consequently,
	\[
	\begin{split}
		(\Lc \ee^{-t \Lc}f) (x)&= \int_G f(x y^{-1}) (\Lc h_t)(y) \,\dd \beta_R(y)\\
		&=\int_G [f(x y^{-1}) -f(x)](\Lc h_t)(y) \,\dd \beta_R(y)\\
		&=\int_G (\Delta_{y^{-1}} f)(x) (\Lc h_t)(y) \,\dd \beta_R(y)\\
		&=\int_G (\Delta_{y } f)(x) (\Lc h_t)(y^{-1}) \,\dd \beta_L(y).
	\end{split}
	\]
	In addition, for every $f,g\in C^\infty_c(G)$,\footnote{Notice that $\langle\,\cdot\,\vert \,\cdot\,\rangle$ is the sesquilinear pairing associated with the measure $\beta$, so that $\langle\,\cdot\,\vert \,\cdot\,\Delta_R^{-1}\rangle$ is the sesquilinear pairing associated with the measure $\beta_R$. Observe, in addition, that here we are using the fact that $h_t$ is real since $\Lc$ is.}
	\[
	\begin{split}
		\langle f*(\Lc h_t)\vert g \Delta_R^{-1}\rangle&= \langle \Lc \ee^{-t \Lc} f\vert g \Delta_R^{-1} \rangle\\
			&= \langle f \vert (\Lc \ee^{-t \Lc}g )\Delta_R^{-1}\rangle\\
			&=\langle f \vert (\Delta_R^{-1}g)*(\Delta_R^{-1}\Lc h_t)\rangle\\
			&=\langle f* (\Delta_R \Rc(\Delta_R^{-1} \Lc h_t))\vert g \Delta_R^{-1}\rangle \\
			&=\langle f*(\Delta_L^{-1} \Delta_R  (\Lc h_t)\check{\,})\vert g \Delta_R^{-1}\rangle,
	\end{split}
	\]
	so that $\Lc h_t=\Delta_L^{-1} \Delta_R  (\Lc h_t)\check{\,} $. Consequently, 
	\[
	\begin{split}
	(\Lc \ee^{-t \Lc}f) (x)&= \int_G f(x y^{-1}) (\Lc h_t)(y) \,\dd \beta_R(y)\\
		&= \int_G f(x y^{-1}) (\Delta_L^{-1} \Delta_R)(y)(\Lc h_t)(y^{-1}) (\Delta_R^{-1}\Delta_L)(y) \,\dd \beta_L(y)\\
		&=\int_G f(x y^{-1}) (\Lc h_t)(y^{-1})  \,\dd \beta_L(y),
		\end{split}
	\]
	so that, arguing as above,
	\[
	\begin{split}
		2(\Lc \ee^{-t \Lc}f) (x)&= \int_G [f(x y ) -2f(x)+f(x y^{-1})](\Lc h_t)(y^{-1})  \,\dd \beta_L(y)\\
			&= \int_G (\Delta_y^{(2)} f)(x y^{-1}) (\Lc h_t)(y^{-1})  \,\dd \beta_L(y)
	\end{split}
	\]
	for every $x\in G$ and for every $t>0$.	
	
	Now, take $C_1,b>0$ so that $\abs{\Lc^j h_t}\meg C_1 t^{-j}  p_{b,t}$ for every $t\in (0,1]$ and for every $j=0,1$ (cf.~Theorem~\ref{teo:7}). Choose $c>1$ so that $\Delta_L(x),\Delta_R(x)\meg c\ee^{c\abs{x}_*}$ for every $x\in G$ (cf.~Lemma~\ref{lem:31}) and so that $\beta(B(e,r))\meg \ee^{r c/2}$ for every $r\Meg 1$, and define $\nu$ as the measure on $G\setminus\Set{e}$ such that $\dd \nu(x)= \abs{x}_*^{-Q_*} \ee^{-c \abs{x}_*} \dd \beta(x)$.	
	Then,
	\[
	\begin{split}
		t^{-\alpha/\grado}\norm{W_t^{(1)} f}_{L^p(\beta)}&\meg C_1  t^{-\alpha/\grado}\int_G \norm{\Delta_y f}_{L^p(\beta)} p_{b,t}(y) \,\dd \beta_L(y)\\
		&\meg c C_1   \int_{G\setminus \Set{e}} \abs{y}_*^{-\alpha} \ee^{-c\abs{y}_*}\norm{\Delta_y f}_{L^p(\beta)} \ee^{3 c \abs{y}_*} (\abs{y}_*/t^{1/\grado})^{\alpha+Q_*}\ee^{-b(\abs{y}_*/t^{1/\grado})^{\grado/(\grado-1)}} \,\dd \nu(y)
	\end{split}
	\]
	for every $t\in (0,1]$.  Then, Lemma~\ref{lem:29} shows that there is a constant $C_2>0$ such that
	\[
	\begin{split}
		\Bs^{p,q}_{\alpha,1}(f)&\meg C_2 \norm*{ \abs{y}_*^{-\alpha } \ee^{- c \abs{y}_*} \norm{\Delta_y f}_{L^p(\beta)}   }_{L^q_y(\nu)}\\
		& = C_2 \norm*{ \abs{y}_*^{-\alpha-Q_*/q} \ee^{-(1+1/q)c \abs{y}_*} \norm{\Delta_y f}_{L^p(\beta)}   }_{L^q_y(\beta)}\\
		&\meg C_2 \norm*{ \abs{y}_*^{-\alpha-Q_*/q}   \norm{\Delta_y f}_{L^p(\beta)}   }_{L^q_y(B(e,1),\beta)}\\
		&\qquad+C_2\norm{f}_{L^p(\beta)} \norm*{ \abs{y}_*^{-\alpha-Q_*/q} \ee^{-(1+1/q)c \abs{y}_*}(1+ \Delta_R(y)^{1/p})  }_{L^q_y(G\setminus B(e,1),\beta)},
	\end{split}
	\]
	whence the conclusion when $k=1$ by means of~\cite[Propositions 6.9]{BCP}, since $\Delta_R(y)^{1/p}\meg c^{1/p}\ee^{c\abs{y}_*/p}$ by our choice of $c$. Analogously,
	\[
	\begin{split}
	t^{-\alpha/\grado}\norm{W_t^{(1)} f}_{L^p(\beta)}&\meg C_1  t^{-\alpha/\grado}\int_G \norm{(\Delta_y^{(2)} f)(\,\cdot\, y^{-1})}_{L^p(\beta)} p_{b,t}(y) \,\dd \beta_L(y)\\
		&= C_1  t^{-\alpha/\grado}\int_G \norm{\Delta_y^{(2)} f }_{L^p(\beta)} p_{b,t}(y) \Delta_R(y)^{1/p}\Delta_L(y^{-1}) \,\dd \beta(y),
	\end{split}
	\]
	from which one may proceed as in the previous case, and complete the proof also when $k=2$.
\end{proof}

In order to deal with the reverse inequality, we need a technical lemma. 
Even though we only need Lemma~\ref{lem:30bis}, we state and prove separately Lemma~\ref{lem:30} since its statement  is more precise and its proof is more elementary (and provides more `global' information). We state both results for vector-valued functions for future reference.

\begin{lem}\label{lem:30}
	Let $(X_j)_{j\in J}$ be an orthonormal basis of $\gf$ compatible with the filtration, and let $B$ be a Banach space. Take $c\Meg 1$ so that $\Delta_R(x)\meg c\ee^{c \abs{x}_*}$ for every $x\in G$ (cf.~Lemma~\ref{lem:31}). Then, 
	\[
	\norm{\Delta_x f}_{L^p(\beta;B)}\meg c\ee^{c\abs{x}_*/p}\sum_{j\in J} \min(\abs{x}_*^{d_j},\abs{x}_*) \norm{X_j f}_{L^p(\beta;B)}
	\]
	for every $p\in [1,\infty]$, for every $x\in G$, and for every $f\in L^1_\loc(\beta;B)$, where $d_j=\deg X_j$ for every $j\in J$.
	In addition,
	\[
	\abs{(\Delta_x f)(y)}\meg \sum_{j\in J} \min(\abs{x}_*^{d_j},\abs{x}_*) \sup_{B(y,\abs{x}_*)}\abs{X_j f}
	\]
	for every $x,y\in G$ and for every $f\in C^1(G;B)$.
\end{lem}

\begin{proof}
	Up to convolving (on the left) $f$ with an approximate identity, we may reduce to the case $f\in C^1(G;B)$ also for the first assertion.
	Take $\eps>\abs{x}_*$. Then, there is an absolutely continuous curve $\gamma\colon [0,1]\to G$ such that $\gamma(0)=e$, $\gamma(1)=x $, and $\gamma'(t)= \sum_{j\in J} c_j(t) (X_j)_{\gamma(t)}$ for almost every $t\in [0,1]$, with $\norm{c_j}_{L^\infty([0,1])}\meg \min(\eps^{d_j},\eps)$ for every $j\in J$. Then,
	\[
	\abs{\Delta_x f(y)}\meg \sum_{j\in J}\int_0^1 \abs{c_j(t) (X_j f)(y \gamma(t))}\,\dd t\meg   \sum_{j\in J}\min(\eps^{d_j},\eps)\int_0^1 \abs{  (X_j f)(y \gamma(t))}\,\dd t
	\]
	for every $y\in G$,
	whence
	\[
	\norm{\Delta_x f}_{L^p(\beta)}\meg \sum_{j\in J}\min(\eps^{d_j},\eps) \int_0^1 \Delta_R^{-1/p}(\gamma(t))\norm{  X_j f}_{L^p(\beta)}\,\dd t\meg c\ee^{c\eps/p}  \sum_{j\in J}\min(\eps^{d_j},\eps) \norm{  X_j f}_{L^p(\beta)},
	\]
	so that the first assertion follows by the arbitrariness of $\eps>\abs{x}_*$. The second assertion follows since $y\gamma(t)\in B(y,\eps)$ for every $t\in [0,1]$ and for every $y\in G$, and since $X_j f$ is continuous for every $j\in J$.
\end{proof}

\begin{lem}\label{lem:30bis}
	Let $(X_j)_{j\in J}$ be an orthonormal basis of $\gf$ compatible with the filtration, and let $B$ be a Banach space. Then, there are $\eps,C>0$ such that
	\[
	\norm{\Delta_x^{(k)}f}_{L^p(\beta;B)}\meg C^k \sum_{j\in J^k} \abs{x}_*^{d_j} \norm{X_{j_1}\cdots X_{j_k}f}_{L^p(\beta;B)}
	\]
	for every $p\in [1,\infty]$, for every $k\Meg 1$, for every $x\in B(e,\eps)$, and for every $f\in L^1_\loc(\beta;B)$, where $d_j=\deg(X_{j_1})+\cdots+\deg(X_{j_k})$ for every $j\in J^k$.
	
	In addition, there is a constant $c>0$ such that
	\[
	\abs{(\Delta^{(k)}_x f)(y)}\meg C^k\sum_{j\in J^k}  \abs{x}_*^{d_j}  \sup_{B(y,ck\abs{x}_*)}\abs{X_{j_1}\cdots X_{j_k}f}
	\]
	for every $x\in B(e,\eps)$, for every $y\in G$, and for every $f\in C^k(G;B)$.
\end{lem}

We observe explicitly that the proof below only works for higher-order differences such as $\Delta_x^{(k)}$ (with $x$ small), and would fail for  higher order differences  such as  $\Delta_{x_1}\cdots \Delta_{x_k}$  with  $\card(\Set{x_1,\dots, x_k})\Meg 2$.

\begin{proof}
	Up to convolving (on the left) $f$ with an approximate identity, we may reduce to the case $f\in C^k(G;B)$ also for the first assertion.
	Take $\eps'>0$ and $C'>1$ so that the exponential map $\exp_G$ induces a diffeomorphism of $U\coloneqq \sum_{j\in J} [-\eps'^{d_j},\eps'^{d_j}] X_j$ onto a neighbourhood of $e$ in $G$, and so that 
	\[
	\frac{1}{C'} \abs{\exp_G(X)}_*\meg \max_{j\in J} \abs{\langle X\vert X_j\rangle}^{1/d_j}\meg C' \abs{\exp_G(X)}_* 
	\]
	for every $X\in \overline U$ (cf.~\cite[Proposition 6.1]{ElstRobinson}).  Fix $\eps\in (0,\eps'/C')$ so that $B(e,\eps)\subseteq \exp_G(U )$. 
	Take $x\in B(e, \eps)$, and observe that $x=\exp_G(X)$ for some $X=\sum_{j\in J} a_jX_j$ such that $\abs{a_j}\meg (C'\abs{x}_*)^{d_j}$ for every $j\in J$. For every $y\in G$, define $g_{x,y}\colon \R \ni t \mapsto f(y\exp_G(tX))\in B$, and observe that $g_{x,y}\in C^k(\R;B)$ and that
	\[
	(\Delta_{1}^{(k)} g_{x,y})(0)=\int_0^1 (\Delta_{1}^{(k-1)}g'_{x,y})(t)\,\dd t= \cdots =\int_{[0,1]^k} g^{(k)}_{x,y}(t_1+\cdots+t_k)\,\dd (t_1,\dots, t_k)=\int_0^k g_{x,y}^{(k)}(t) h_k(t)\,\dd t,
	\]
	where $h_k=\chi_{[0,1]}^{*k}$, that is, the convolution of $k$ functions all equal to $\chi_{[0,1]}$. 
	In particular, $h_k$ is a bounded continuous (unless $k=1$) function, and $\int_0^k h_k(t)\,\dd t=1$. Now, observe that
	\[
	(\Delta_{x}^{(k)} f)(y)= (\Delta_{1}^{(k)} g_{x,y})(0)= \int_0^k g_{x,y}^{(k)}(t) h_k(t)\,\dd t=\int_0^k (X^k f)(y \exp_G(t X)) h_k(t)\,\dd t,
	\]
	so that
	\[
	\norm{\Delta_x^{(k)}f}_{L^p(\beta)} \meg C''\int_0^k \norm{X^k f}_{L^p(\beta)} h_k(t)\,\dd t=C''\norm{X^k f}_{L^p(\beta)} ,
	\]
	where $C''=\max_{B(0,\eps)} \Delta_R$. The first assertion follows, observing that
	\[
	\norm{X^k f}_{L^p(\beta)}\meg \sum_{j\in J^k} \abs{a_{j_1}\cdots a_{j_k}} \norm{X_{j_1}\cdots X_{j_k} f}_{L^p(\beta)}
	\]
	and applying the previous estimates of the $a_j$. For what concerns the second assertion, observe first that, if $k C' \abs{x}_*\meg \eps'$, then $\abs{\exp_G(tX)}_*\meg C'^2 k \abs{x}_*$ for every $t\in [0,k]$. If, otherwise, $k C' \abs{x}_*>\eps'$, then there is $t_0\in (1,k)$ such that $t_0C' \abs{x}_*=\eps'$. Consequently, $\abs{\exp_G(t X)}_*\meg C'\eps'$ for every $t\in [0,t_0]$, so that $\abs{\exp_G(tX)}_*\meg C'\eps'([k/t_0]+1)\meg 2 C'^2 k\abs{x}_*$ for every $t\in [0,k]$. The assertion then follows easily.
\end{proof}

\begin{prop}\label{prop:32b}
	Take $p,q\in [1,\infty]$, $h\in \N$, and $\alpha\in (0,h\dd_0)$. Then, there is a constant $C>1$ such that  
	\[
	\norm{f}_{L^p(\beta)}+\norm*{ \abs{y}_*^{-\alpha-Q_*/q} \norm{\Delta_y^{(h)} f}_{L^p(\beta)} }_{L^q_y(B(e,1),\beta)}  \meg C \norm{f}_{B^{p,q}_\alpha(\beta)}
	\]
	for every $f\in \Sc'(G)$.
\end{prop} 

\begin{proof}
	Observe that $f= \sum_{k=0}^{h-1} \frac{1}{k!} \Lc^k\ee^{-\Lc }f+\frac{1}{(h-1)!} \int_0^1 t^{h-1}\Lc^h \ee^{-t \Lc} f\,\dd t$, thanks to Lemma~\ref{lem:9}. In addition, by Lemma~\ref{lem:30bis} there are  $\eps>0$ and $C_1>0$ such that, for every $k=0,\dots, h$, for every $t\in (0,1]$, and for every $y\in G$,
	\[
	\norm{\Delta_y^{(h)}(\Lc^k \ee^{-t \Lc} f)}_{L^p(\beta)}\meg (1+\Delta_R(y)^{1/p})^h \norm{\Lc^k\ee^{-t \Lc}f}_{L^p(\beta)},
	\]
	but also, for $y\in B(e,\eps)$,
	\[
	\begin{split}
		\norm{\Delta_y^{(h)}(\Lc^k \ee^{-t \Lc} f)}_{L^p(\beta)}&\meg C_1\sum_{j\in J^h} \abs{y}_*^{d_j}\norm{X_{j_1}\cdots X_{j_h} \Lc^k\ee^{-t \Lc}f}_{L^p(\beta)}, 
	\end{split}
	\]
	where $(X_j)_{j\in J}$ denotes an orthonormal basis of $\gf$ compatible with the filtration and $d_j=\deg (X_{j_1})+\cdots+ \deg(X_{j_h})$ for every $j\in J^h$.
	By means of Lemma~\ref{lem:18}, we then see that there is a constant $C_2>0$ such that
	\[
	\norm{\Delta_y^{(h)}(\Lc^k \ee^{-t \Lc} f)}_{L^p(\beta)}\meg C_2\sum_{j\in J^h}\frac{\abs{y}_*^{d_j}}{t^{d_j/\grado}+\abs{y}_*^{d_j}} \norm{  \Lc^k \ee^{-(t/2) \Lc}f}_{L^p(\beta)}
	\]
	for every $k=0,\dots, h$, for every $t\in (0,1]$, and for every $y\in B(e,\eps)$.  In fact, it is readily seen that the same inequality may be assumed to hold for every $y\in B(e,1)$. 
	Observe that, by Lemma~\ref{lem:18}, there is a constant $C_3>0$ such that, for every $k=0,\dots,h-1$,
	\[
	\norm*{\abs{y}_*^{-\alpha-Q_*/q} \norm{\Delta^{(h)}_y(\Lc^k\ee^{- \Lc} f)}_{L^p(\beta)}}_{L^q_y(B(e,1),\beta)}\meg C_3 \norm{f}_{L^p(\beta)}\sum_{j\in J^h}\norm*{\frac{\abs{y}_*^{d_j-\alpha-Q_*/q}}{(1+\abs{y}_*)^{d_j}}}_{L^q_y(B(e,1),\beta)},
	\]
	where the $L^q$ norms are finite since $d_j\Meg h \dd_0>\alpha$ for every $j\in J^h$.
	
	Now, let $\nu$ be the image of the measure $\chi_{B(e,1)} \abs{\,\cdot\,}_*^{-Q_*}\cdot \beta$ (on $G\setminus \Set{e}$) through the map $\abs{\,\cdot\,}_*$. Observe that $\nu\in \Mc_\Car$, so that by means of Lemma~\ref{lem:25} we see that there is a constant $C_4>0$ such that
	\[
	\begin{split}
		&\norm*{\int_0^1 \abs{y}_*^{-\alpha-Q_*/q} t^{h}\norm{\Delta_y^{(h)}(\Lc^h\ee^{-t \Lc} f)}_{L^p(\beta)}\,\dd\mi_1(t)}_{L^q_y(B(e,1),\beta)}\\
		&\qquad\meg C_2 2^h \sum_{j\in J^h}\norm*{ \int_0^1    \frac{r^{d_j-\alpha}}{t^{d_j/\grado} +r^{d_j}} \norm{ W^{(h)}_{t/2}f}_{L^p(\beta)}\,\dd \mi_1(t)  }_{L^q_r(\nu)}\\ 
		&\qquad\meg C_4  \norm*{ t^{-\alpha/\grado} \norm{ W^{(h)}_{t/2}f}_{L^p(\beta)}   }_{L^q_t(\mi_1)}
	\end{split}
	\]
	The assertion follows since $h>\alpha/\dd_0\Meg \alpha/\grado$.
\end{proof}

We now pass to Triebel--Lizorkin spaces.

\begin{prop}\label{prop:32c}
	Take $p\in (1,\infty)$, $q\in [1,\infty]$, and $\alpha>0$. Then, there is a constant $C>1$ such that 
		\[
			\frac{1}{C}\norm{f}_{F^{p,q}_\alpha(\beta)}\meg \norm{f}_{L^p(\beta)}+\norm{S^{q}_\alpha f }_{L^p(\beta)}  
		\]
		for every $f\in \Sc'(G)$.
\end{prop}

One may also adapt the latter part of the proof of Proposition~\ref{prop:32a} in order to prove a similar assertion where $(\Delta_y f)(x)$ is replaced by the symmetric difference $f(x y^{-1})+f(x y)-2 f(x)$. However, it does not seem that the formulation with symmetric differences may be easily `corrected' to a formulation with second order differences as in the case of Besov spaces.

\begin{proof}
	By means of~\cite[Proposition 6.9]{BCP} and Proposition~\ref{prop:4}, we may assume that $\Lc$ has no constant terms. Consequently, also its formal transpose $\Lc^+$ with respect to $\beta_R$ has no constant terms, so that $\int_G \Lc h_t\,\dd \beta_R=0$ for every $t>0$.
	In addition, by Theorem~\ref{teo:7} we may find $b,C_1>0$ such that $\abs{(t\Lc)^j h_t}\meg C_1 p_{b,t}$ for every $t\in (0,1]$ and for every $j=0,1$, while by Lemma~\ref{lem:31} we may find $c\Meg 1$ such that $\Delta_R(x), \Delta_L(x)\meg c \ee^{c \abs{x}_*}$ for every $x\in G$, and such that $\beta(B(e,r))\meg \ee^{c r}$ for every $r\Meg 1$.
	We set $\gamma\coloneqq \grado/(\grado-1)$ to simplify the notation. Arguing as in the proof of Proposition~\ref{prop:32a} and using Lemma~\ref{lem:3}, we see that there is a constant $C_2>0$ such that
	\[
	\begin{split}
		&\norm{ t^{-\alpha/\grado}\abs{(W^{(1)}_t f)(x)} }_{L^q_t(\mi_1)} =\grado^{-1/q}\norm{ t^{-\alpha }\abs{(W^{(1)}_{t^{\grado}} f)(x)} }_{L^q_t(\mi_1)}\\
		& \qquad\meg c\grado^{-1/q} C_1 \norm*{t^{-\alpha}\int_G \abs{(\Delta_y f)(x)} p_{b,t^{\grado}}(y) \ee^{c\abs{y}_*} \,\dd \beta(y) }_{L^q_t(\mi_1)}\\
		&\qquad\meg C_2( S^q_\alpha f)(x)+C_2 \norm*{t^{-\alpha-Q_*}\int_{G\setminus B(e,t)} \abs{(\Delta_y f)(x)} \ee^{-(b/2) (\abs{y}_*/t)^{\gamma} } \ee^{-4c\abs{y}_*} \,\dd \beta(y) }_{L^q_t(\mi_1)}.
	\end{split}
	\]
	Now, observe that there is a constant $C_3>0$ such that
	\[
	\begin{split}
		&  \norm*{t^{-\alpha-Q_*}\int_{G\setminus B(e,t)} \abs{(\Delta_y f)(x)} \ee^{-(b/2) (\abs{y}_*/t)^{\gamma} } \ee^{-4c\abs{y}_*} \,\dd \beta(y) }_{L^q_t(\mi_1)}   \\
		& \meg \sum_{j\in \N}   \norm*{t^{-\alpha-Q_*}\int_{B(e, 2^{j+1}t)\setminus B(e,2^j t)} \abs{(\Delta_y f)(x)} \ee^{-(b/2) (\abs{y}_*/t)^{\gamma} } \ee^{-4c\abs{y}_*} \,\dd \beta(y) }_{L^q_t(\mi_1)}  \\
		&\meg \sum_{j\in \N} \ee^{ -(b/2)2^{j\gamma}}  \norm*{t^{-\alpha-Q_*}\int_{B(e, 2^{j+1}t) } \abs{(\Delta_y f)(x)}   \ee^{-4c\abs{y}_*} \,\dd \beta(y) }_{L^q_t(\mi_1)}   \\
		& \meg  \sum_{j\in \N} 2^{(j+1)(\alpha+Q_*)}\ee^{ -(b/2)2^{j\gamma}}  \norm*{t^{-\alpha-Q_*}\int_{B(e, t) } \abs{(\Delta_y f)(x)}   \ee^{-4c\abs{y}_*} \,\dd \beta(y) }_{L^q_t(\mi_{2^{j+1}})}  \\
		&\meg C_3   (S^q_\alpha f)(x)+ \sum_{j\in \N} 2^{(j+1)(\alpha+Q_*)}\ee^{ -(b/2)2^{j\gamma}}  \norm*{t^{-\alpha-Q_*}\int_{B(e, t) } \abs{(\Delta_y f)(x)}   \ee^{-4c\abs{y}_*} \,\dd \beta(y) }_{L^q_t([1,2^{j+1}],\mi_{2^{j+1}})}   .
	\end{split}
	\]
	Now, observe that there is a constant $C_4>0$ such that
	\[
	\begin{split}
		&\norm*{t^{-\alpha-Q_*}\int_{B(e, t) } \abs{(\Delta_y f)(x)}   \ee^{-4c\abs{y}_*} \,\dd \beta(y) }_{L^q_t([1,2^{j+1}],\mi_{2^{j+1}})}\\
		 &\qquad\meg \norm*{t^{-\alpha-Q_*}\int_{G} (\abs{f(xy)} +\abs{f(x)} ) \ee^{-4c\abs{y}_*} \,\dd \beta(y) }_{L^q_t([1,2^{j+1}],\mi_{2^{j+1}})}\\
		&\qquad\meg \norm*{t^{-\alpha-Q_*} }_{L^q_t([1,2^{j+1}],\mi_{2^{j+1}})} \left(c\int_G \abs{f(x y^{-1})} \ee^{-3 c\abs{y}_*}\,\dd \beta_R(y) + \abs{f(x)} \int_G \ee^{-4 c\abs{y}_*}\,\dd \beta(y)   \right)\\
		&\qquad\meg C_4 \big[\big(\abs{f}* \ee^{-3 c\abs{\,\cdot\,}_*}\big)(x)+\abs{f(x)}  \big],
	\end{split}
	\]
	so that by Young's inequality and the fact that $\ee^{-2 c\abs{\,\cdot\,}_*}\in L^1(\beta)$, we see that there is a constant $C_5>0$ such that
	\[
	\begin{split}
		&\norm*{\sum_{j\in \N} 2^{(j+1)(\alpha+Q_*)}\ee^{ -(b/2)2^{j\gamma}}  \norm*{t^{-\alpha-Q_*}\int_{B(e, t) } \abs{(\Delta_y f)(x)}   \ee^{-4c\abs{y}_*} \,\dd \beta(y) }_{L^q_t([1,2^{j+1}],\mi_{2^{j+1}})}  }_{L^p_x(\beta)}\\
		&\qquad \meg C_4\norm*{\big(\abs{f}* \ee^{-3 c\abs{\,\cdot\,}_*}\big) +\abs{f } }_{L^p(\beta)}\sum_{j\in \N} 2^{(j+1)(\alpha+Q_*)}\ee^{ -(b/2)2^{j\gamma}}\\
		&\qquad\meg C_5\norm{f}_{L^p(\beta)},
	\end{split}
	\]
	whence the conclusion.
\end{proof}

\begin{prop}\label{prop:32d}
	Take $p\in (1,\infty)$, $q\in [1,\infty]$, $h\in \N$, and $\alpha\in (0,h\dd_0)$. Then, there is a constant $C>1$ such that 
	\[
	\norm{f}_{L^p(\beta)}+\norm{S^{q}_{\alpha,h} f }_{L^p(\beta)}  \meg C\norm{f}_{F^{p,q}_\alpha(\beta)}
	\]
	for every $f\in \Sc'(G)$.
\end{prop}

\begin{proof}
	Take $\eps\in (0,1)$ so that Lemma~\ref{lem:30bis} applies, and observe first that
	\[
	\begin{split}
	\norm{S^q_{\alpha,h} f}_{L^p(\beta)}&\meg\norm*{ \norm*{ t^{-\alpha-Q_* } \int_{B(e,t)} \abs{(\Delta^{(h)}_y f)(x)}\,\dd \beta(y)  }_{L^q_t(\mi_\eps)}}_{L^p_x(\beta)}\\
		&\qquad+ \eps^{-\alpha-Q_*} \abs{\log \eps}^{1/q} \beta(B(e,1))\bigg(1+ \sup_{B(e,1)} \Delta_R^{1/p}\bigg)^h \norm{f}_{L^p(\beta)} .
	\end{split}
	\]
	By some abuse of notation, we may thus assume that $S^q_{\alpha,h} f=\norm*{ t^{-\alpha-Q_* } \int_{B(e,t)} \abs{(\Delta^{(h)}_y f)(x)}\,\dd \beta(y)  }_{L^q_t(\mi_\eps)}$. 
	In addition, observe that, by the area formula, there is a constant $C_1>0$ such that $(\Phi_j)_*(\chi_{B(e,t)}\cdot \beta)\meg C_1 \chi_{B(e, j t)}\cdot \beta$ for every $t\in (0,\eps]$ and for every $j=1,\dots, h$, where $\Phi_j\colon G\ni y\mapsto y^j \in G$. By Theorem~\ref{teo:7}, we may also assume that there is $b>0$ such that $\abs{(t \Lc)^k h_t}\meg C_1 p_{b,t}$ for every $t\in (0,1]$ and for every $k=0,\dots, h$. We set $\gamma\coloneqq \grado/(\grado-1)$ to simplify the notation.

	Then, we proceed as in the proof of Proposition~\ref{prop:32b} and estimate $S^q_{\alpha,h} \Lc^k\ee^{-\Lc}f$ for $k=0,\dots, h-1$ and $S^q_{\alpha,h} \big(\int_0^1 s^{h}\Lc^h\ee^{-s\Lc}f\,\dd \mi_1(s)\big)$ separately. For technical reasons, we shall also split the integral in the second term as $\int_0^{t^{\grado}}+\int_{t^{\grado}}^1$, where $t$ is the variable in the $L^q(\mi_1)$-norm which appears in the definition of $S^q_{\alpha,h}$.
	Observe that, by Lemmas~\ref{lem:3} and~\ref{lem:30bis}, there are $c,C_2>1$ such that, for every $k=0,\dots, h$, for every $x\in G$, for every $t\in (0,\eps)$, and for every $s\in (0,1]$,  
	\[
	\begin{split}
		t^{-Q_*}\int_{B(e,t)} \abs{\Delta_y^{(h)}(\Lc^k\ee^{-s\Lc} f)(x)} \,\dd \beta(y) &\meg t^{-Q_*}\sum_{j=0}^h \binom{h}{j} \int_{B(e,t)} \abs{ (\Lc^k\ee^{-s\Lc} f)(x y^j)} \,\dd \beta(y) \\
			&\meg t^{-Q_*}\int_{B(e,t)} \abs{(\Lc^k\ee^{-s\Lc} f)(x)}  \,\dd \beta(y) \\
			&\qquad+ 2^h C_1 t^{-Q_*} \int_{B(e, h t)} \abs{(\Lc^k\ee^{-s\Lc} f)(x y)}   \,\dd \beta(y)\\
			&\meg t^{-Q_*}\int_{B(e,t)} \abs{(\Lc^k\ee^{-s\Lc} f)(x)}  \,\dd \beta(y) \\
			&\qquad+ 2^h C_1  \ee^{ 2b h^{\gamma}  } \int_{G} \abs{(\Lc^k\ee^{-s\Lc} f)(x y^{-1})} p_{ 2 b, t^{\grado}}(y) \Delta_L^{-1}(y) \,\dd \beta_R(y)\\
		&\meg C_2  \abs{(\Lc^k\ee^{-s\Lc} f)(x)}+ C_2 (T_{b , t^{\grado}} \Lc^k\ee^{-s\Lc} f)(x)
	\end{split}
	\]
	and such that 
	\[
	\begin{split}
		t^{-Q_*}\int_{B(e,t)} \abs{\Delta_y^{(h)}(\Lc^h\ee^{-s\Lc} f)(x)} \,\dd \beta(y)&\meg C_2\sum_{j\in J^h} \max_{y\in B(e,t)}\abs{y}_*^{d_j} \max_{z\in B(x,c h\abs{y}_*)} \abs{(X_{j_1}\cdots X_{j_h} \Lc^h\ee^{-s\Lc} f)(z) } \\
		&\meg  C_2^2\sum_{j\in J^h} t^{d_j} (s/2)^{-d_j/\grado}   \max_{z\in B(x,c h t)} (T_{ b,s/2 }\Lc^h\ee^{-(s/2)\Lc} f)(z) \\
		&\meg C_2^2 \sum_{j\in J^h} t^{d_j} (s/2)^{-d_j/\grado} \ee^{2b (2c h t/s^{1/\grado})^{\gamma}}   (T_{ b/2,s/2 }\Lc^h\ee^{-(s/2)\Lc} f)(x).
	\end{split}
	\]
	Consequently, by Lemma~\ref{lem:3} and Young's inequality we see that there is a constant $C_{3}>0$ such that
	\[
	\norm{S^q_{\alpha,h}\Delta_y^{(h)} \Lc^k\ee^{-\Lc}f }_{L^p(\beta)}\meg C_3\norm{f}_{L^p(\beta)}
	\]
	for $k=0,\dots, h-1$.
	In addition,
	\[
	\begin{split} 
		&\norm*{t^{-\alpha-Q_*} \int_{B(e,t)}  \int_{t^{\grado}}^1  s^{h} \abs*{(\Delta_y^{(h)} \Lc^h \ee^{-s \Lc} f)(x)} \,\dd \mi_1(s) \,\dd \beta(y)  }_{L^q_t(\mi_\eps)}\\
		& \qquad  \meg\ee^{2b (2 c h)^\gamma }C_2^2 \sum_{j\in J^h}  2^{d_j/\grado}\norm*{t^{d_j-\alpha }    \int_{t^{\grado}}^1 s^{h-d_j/\grado}  (T_{  b/2,s/2 }\Lc^h\ee^{-(s/2)\Lc} f)(x) \,\dd \mi_1(s)   }_{L^q_t(\mi_\eps)}\\
		&\qquad \meg\grado^{-1/q} \ee^{2b (2 c h)^\gamma }C_2^2 \sum_{j\in J^h}  2^{d_j/\grado+h}\norm*{     \int_{t }^1 (t/s)^{(d_j-\alpha)/\grado}  [s^{-\alpha/\grado}(T_{  b/2,s/2 }W_{s/2}^{(h)} f)(x)]  \,\dd \mi_1(s)    }_{L^q_t(\mi_1)}\\
		&\qquad \meg  \grado^{-1/q} \ee^{2b (2 c h)^\gamma  }C_2^2 \sum_{j\in J^h} 2^{d_j/\grado+h} \norm{\tau^{(d_j-\alpha)/\grado}}_{L^1_\tau(\mi_1)}   \norm*{   s^{-\alpha/\grado}(T_{  b/2,s/2 }W_{s/2}^{(h)} f)(x)   }_{L^q_s(\mi_1)}
	\end{split}
	\]
	by Young's inequality (applied to the multiplicative group $(0,+\infty)$). Notice that $ \norm{\tau^{(d_j-\alpha)/\grado}}_{L^1_\tau(\mi_1)} $ is finite since $d_j\Meg h\dd_0>\alpha$ for every $j\in J^h$. Then, by Lemma~\ref{lem:2} we see that there is a constant $C_{4}>0$ such that
	\[
	\norm*{\norm*{t^{-\alpha-Q_*} \int_{B(e,t)}  \int_{t^{\grado}}^1  s^h \abs*{(\Delta_y^{(h)} \Lc^h \ee^{-s \Lc} f)(x)} \,\dd \mi_1(s) \,\dd \beta(y)  }_{L^q_t(\mi_\eps)}}_{L^p_x(\beta)}\meg C_{4} \norm*{ \norm{s^{-\alpha/\grado} (W_{s/2}^{(h)} f)(x)}_{L^q_s(\mi_1)} }_{L^p_x(\beta)}.
	\] 
	It remains to estimate (the $L^p_x(\beta)$-norm of)
	\[
	\begin{split}
		&\norm*{t^{-\alpha-Q_*} \int_{B(e,t)}  \int_0^{t^{\grado}} s^h  \abs*{(\Delta_y^{(h)} \Lc^h \ee^{-s \Lc} f)(x)} \,\dd\mi_1(s) \,\dd \beta(y)  }_{L^q_t(\mi_\eps)}\\
		& \qquad \meg C_2\grado^{-1/q}\norm*{t^{-\alpha/\grado } \int_0^t  [\abs{(W^{(h)}_s f)(x)}+ (T_{b,t }W^{(h)}_s f)(x)] \,\dd \mi_1(s)   }_{L^q_t(\mi_1)}. 
	\end{split}
	\]
	Observe first that
	\[
	\begin{split}
		\norm*{t^{-\alpha/\grado } \int_0^t  \abs{(W^{(h)}_s f)(x)} \,\dd \mi_1(s)  }_{L^q_t(\mi_1)}&=\norm*{  \int_0^t (t/s)^{-\alpha/\grado} [s^{-\alpha/\grado} \abs{(W^{(h)}_s f)(x)]} \,\dd \mi_1(s)  }_{L^q_t(\mi_1)}\\
		&\meg  \norm{\tau^{\alpha/\grado}}_{L^1_\tau(\mi_1)} \norm*{s^{-\alpha/\grado }   (W^{(h)}_s f)(x)}_{L^q_s(\mi_1)}
	\end{split}
	\]
	by Young's inequality again. Concerning the last term, by means of Lemma~\ref{lem:2} se wee that there is a constant $C_5>0$ such that
	\[
	\begin{split}
		 \norm*{t^{-\alpha/\grado } \int_0^t    (T_{b,t }W^{(h)}_s f)(x)  \,\dd \mi_1(s)   }_{L^{q,p}_{t,x}(\mi_1,\beta)}  &=\norm*{T_{b,t }\bigg( t^{-\alpha/\grado } \int_0^t    \abs{W^{(h)}_s f}  \,\dd \mi_1(s) \bigg)(x)  }_{L^{q,p}_{t,x}(\mi_1,\beta)}\\
		&\meg C_5\norm*{ t^{-\alpha/\grado } \int_0^t    \abs{(W^{(h)}_s f)(x)}  \,\dd \mi_1(s)    }_{L^{q,p}_{t,x}(\mi_1,\beta)},
	\end{split}
	\]
	so that one may conclude by the above arguments. 
\end{proof}

\end{document}